\documentclass{amsart}
\usepackage{amsmath,amsthm,amssymb}
\usepackage{enumerate}
\usepackage{graphicx}
\usepackage{cite}
\usepackage{comment}
\usepackage{oands}
\usepackage{tikz}
\usepackage{changepage}
\usepackage{bbm}
\usepackage{mathtools}
\usepackage[margin=1.15in]{geometry}
\usepackage{hyperref}
\usepackage{appendix}
\usepackage[pagewise,mathlines]{lineno}
\usepackage{multicol}

\theoremstyle{plain}
\newtheorem{thm}{Theorem}[section]

\newtheorem{lem}[thm]{Lemma}
\newtheorem{prop}[thm]{Proposition}

\newtheorem{notation}[thm]{Notation}

\numberwithin{equation}{section}
 
\theoremstyle{definition}
\newtheorem{defn}[thm]{Definition}
\newtheorem{remark}[thm]{Remark}

\def\@rst #1 #2other{#1}

\hypersetup{
    colorlinks=false,
    linktocpage,
    }

\newcommand{\dsb}{\begin{adjustwidth}{2.5em}{0pt}
\begin{footnotesize}}
\newcommand{\dse}{\end{footnotesize}
\end{adjustwidth}}

\newcommand{\ssb}{\begin{adjustwidth}{2.5em}{0pt}}
\newcommand{\sse}{\end{adjustwidth}}

\newcommand{\aryb}{\begin{eqnarray*}}
\newcommand{\arye}{\end{eqnarray*}}
\def\alb#1\ale{\begin{align*}#1\end{align*}}
\def\allb#1\alle{\begin{align}#1\end{align}}
\newcommand{\eqb}{\begin{equation}}
\newcommand{\eqe}{\end{equation}}
\newcommand{\eqbn}{\begin{equation*}}
\newcommand{\eqen}{\end{equation*}}

\newcommand{\BB}{\mathbbm}
\newcommand{\ol}{\overline}
\newcommand{\ul}{\underline}
\newcommand{\op}{\operatorname}

\newcommand{\frk}{\mathfrak}
\newcommand{\eqD}{\overset{d}{=}}
\newcommand{\ep}{\epsilon}
\newcommand{\rta}{\rightarrow}

\newcommand{\wt}{\widetilde}
\newcommand{\wh}{\widehat} 
\newcommand{\mcl}{\mathcal}

\newcommand*\tc[1]{\tikz[baseline=(char.base)]{\node[shape=circle,draw,inner sep=1pt] (char) {#1};}}
\newcommand*\tb[1]{\tikz[baseline=(char.base)]{\node[shape=rectangle,draw,inner sep=2.5pt] (char) {#1};}}

\let\originalleft\left
\let\originalright\right
\renewcommand{\left}{\mathopen{}\mathclose\bgroup\originalleft}
\renewcommand{\right}{\aftergroup\egroup\originalright}

\newcommand*\patchAmsMathEnvironmentForLineno[1]{  \expandafter\let\csname old#1\expandafter\endcsname\csname #1\endcsname
  \expandafter\let\csname oldend#1\expandafter\endcsname\csname end#1\endcsname
  \renewenvironment{#1}     {\linenomath\csname old#1\endcsname}     {\csname oldend#1\endcsname\endlinenomath}}\newcommand*\patchBothAmsMathEnvironmentsForLineno[1]{  \patchAmsMathEnvironmentForLineno{#1}  \patchAmsMathEnvironmentForLineno{#1*}}\AtBeginDocument{\patchBothAmsMathEnvironmentsForLineno{equation}\patchBothAmsMathEnvironmentsForLineno{align}\patchBothAmsMathEnvironmentsForLineno{flalign}\patchBothAmsMathEnvironmentsForLineno{alignat}\patchBothAmsMathEnvironmentsForLineno{gather}\patchBothAmsMathEnvironmentsForLineno{multline}}

\title[Scaling limits for the FK model II]{Scaling limits for the critical Fortuin-Kasteleyn model on a random planar map II: local estimates and empty reduced word exponent} 
 
\author{Ewain Gwynne}
\author{Xin Sun}
\subjclass[2010]{Primary 60F17, 60G50; Secondary 82B27}
\keywords{Fortuin-Kasteleyn model, random planar maps, hamburger-cheeseburger bijection, random walks in cones, scaling limits, local limit theorems, Liouville quantum gravity}

\address{Department of Mathematics\\
  Massachusetts Institute of Technology\\
  Cambridge, MA 02139}
\email{ewain@mit.edu \\  xinsun89@math.mit.edu}

\begin{document}

\begin{abstract}
We continue our study of the inventory accumulation introduced by Sheffield (2011), which encodes a random planar map decorated by a collection of loops sampled from the critical Fortuin-Kasteleyn (FK) model. We prove various \emph{local estimates} for the inventory accumulation model, i.e., estimates for the precise number of symbols of a given type in a reduced word sampled from the model. Using our estimates, we obtain the scaling limit of the associated two-dimensional random walk conditioned on the event that it stays in the first quadrant for one unit of time and ends up at a particular position in the interior of the first quadrant. 
We also obtain the exponent for the probability that a word of length $2n$ sampled from the inventory accumulation model corresponds to an empty reduced word, which is equivalent to an asymptotic formula for the partition function of the critical FK planar map model.
The estimates of this paper will be used in a subsequent paper to obtain the scaling limit of the lattice walk associated with a finite-volume FK planar map.
\end{abstract}
 
\maketitle

\tableofcontents

\section{Introduction}
\label{sec-intro}

\subsection{Overview}

For $q>0$, a \emph{(critical) Fortuin-Kasteleyn (FK) planar map} of size $n\in\BB N$ is a pair $(M , S)$ consisting of a planar map $M$ and a subset $S$ of the edge set of $M$, sampled from the uniform distribution on such pairs weighted by the partition function of the critical Fortuin-Kasteleyn model on $M$. The set $S$ can equivalently be described by means of the collection $\mcl L$ of loops which separate connected components of $S$ from connected components of the set $S^*$ of edges in the dual map $M^*$ which do not cross edges of $S$. In~\cite{shef-burger}, Sheffield introduces a method to encode an FK planar map in terms of a certain random word in an alphabet of five symbols. This word, in turn, gives rise to a two-dimensional lattice walk. This encoding is called the \emph{hamburger-cheeseburger bijection} because it has a natural interpretation as the inventory accumulation process of a certain burger restaurant. The hamburger-cheeseburger bijection generalizes a bijection due to Mullin~\cite{mullin-maps} (see also~\cite{bernardi-maps}) and is equivalent for a fixed choice of map $M$ to the bijection described in~\cite[Section 4]{bernardi-sandpile}.

 In addition to its interest as a tool for studying FK planar maps, the hamburger-cheeseburger bijection serves as the main source of discrete intuition behind the recent works~\cite{wedges,sphere-constructions}  
which introduce the ``peanosphere construction" to encode a conformal loop ensemble ($\op{CLE}_\kappa$)~\cite{shef-cle,shef-werner-cle,ig1,ig2,ig3,ig4} on a Liouville quantum gravity (LQG) surface with parameter $\gamma = 4/\sqrt\kappa$\cite{shef-kpz,shef-zipper,wedges}. The main result in \cite{shef-burger} is that the lattice walk constructed from the word associated with an infinite-volume FK-weighted random planar map converges in the scaling limit to a positively correlated two-sided Brownian motion (see~\cite[Theorem 2.5]{shef-burger}). This Brownian motion has the same correlation as the Brownian motion appearing in the peanosphere construction in~\cite[Theorem 1.13]{wedges} when the FK paramter satisfies $q = 2 + 2 \cos(8\pi/\kappa)$. Hence \cite{shef-burger} can be seen as a scaling limit result from FK planar maps to CLE-decorated LQG in a certain topology, namely the one in which two loop-decorated surfaces are said to be close if the two-dimensional paths which encode them are close. 

There have been several recent works regarding the hamburger-cheeseburger approach to critical FK planar maps. In the article~\cite{gms-burger-cone} the present authors and C. Mao improve the topology in the scaling limit result of \cite{shef-burger} by proving a statement which implies, among other things, the convergence of the quantum areas and quantum lengths associated with macroscopic FK loops. The authors of~\cite{blr-exponents} identify the tail exponents for the laws of several quantities associated with FK loops (\cite{gms-burger-cone} independently proves that the tail is actually regularly varying with the same exponent for several of these quantities). The work~\cite{chen-fk} studies the infinite-volume version of the hamburger-cheeseburger bijection. The paper~\cite{sun-wilson-unicycle} studies the sandpile model and unicycles on a random planar map using the hamburger-cheeseburger bijection. 

We note that in the special case of a uniform planar map (without loop decoration), there is also another approach based on the bijection of Schaeffer~\cite{schaeffer-bijection}, which has met with substantial success in showing that the scaling limit of uniform random planar maps is a continuum metric space called the \emph{Brownian map}~\cite{legall-uniqueness,miermont-brownian-map,bjm-uniform}. The recent works~\cite{sphere-constructions,tbm-characterization,lqg-tbm1,lqg-tbm2,lqg-tbm3} (some of which are still in preparation) construct a metric on LQG for $\gamma = \sqrt{8/3}$ under which it is isometric to the Brownian map.

In this paper, we continue the theme of 
\cite{gms-burger-cone} by studying the fine asymptotic properties of the word associated with a critical FK planar map. In particular, we will prove a variety of \emph{local estimates} which give us up-to-constants asymptotics for the probability that the reduced word corresponding to a word sampled from the inventory accumulation model contains a \emph{particular} number of symbols of each type. Such estimates play a crucial role in the study of small-scale events associated with the inventory accumulation model, e.g.\ the event that the associated lattice walk ends up at a particular point after a given amount of time.
Local estimates are not proven in the works~\cite{shef-burger,gms-burger-cone,blr-exponents}, which focus mainly on the behavior of the word at large scales. The starting point of the proofs of our local estimates is the bivariate local limit theorem of Doney~\cite{doney-bivariate}. 

As an application of our estimates, in Theorem \ref{thm-local-conv} we will prove that if we condition on the lattice walk in the hamburger-cheeseburger model to stay in the first quadrant for a certain amount of time and end up at a fix interior point, then the walk converges in the scaling limit to a correlated Brownian bridge conditioned on staying in the first quadrant. As another application, in Theorem~\ref{thm-empty-prob} we obtain the exact exponent of the probability that a word of length $2n$ reduces to the empty word. This latter estimate is equivalent to a certain estimate for the partition function of the critical FK planar map model, up to an error of $n^{o_n(1)}$; see Section~\ref{sec-partition-function}.

The estimates established in this paper will also be used in the forthcoming work~\cite{gms-burger-finite}, in which we will prove analogues of the scaling limit results of~\cite{shef-burger,gms-burger-cone} for the finite-volume version of the model in~\cite{shef-burger} (which is encoded by a word of length $2n$ conditioned to reduce to the empty word); and in~\cite{gwynne-miller-cle}, in which the first author and J. Miller prove convergence of the full topological structure of FK planar maps to that of a conformal loop ensemble on an independent Liouville quantum gravity surface.
\bigskip

\noindent{\bf Acknowledgments}
We thank Ga\"etan Borot, Nina Holden, Cheng Mao, Jason Miller, and Scott Sheffield for helpful discussions. We thank an anonymous referee for many helpful comments on an earlier version of this article. We thank the Isaac Newton Institute for its hospitality during part of our work on this project. The first author was supported by the U.S. Department of Defense via an NDSEG fellowship. The second author was partially supported by NSF grant DMS-1209044. 
 
\subsection{Notation}
\label{sec-burger-prelim}
In this section we will introduce some notation which will remain fixed throughout the paper. This notation is in agreement with that used in~\cite{gms-burger-cone}. 

\subsubsection{Basic notation}
   
%Let $X$ be a random variable taking values in a countable state space $\Omega$. A \emph{realization} of $X$ is an element $x\in\Omega$ such that $\BB P(X=x) > 0$.  

\begin{notation} \label{def-discrete-intervals}
For $a < b \in \BB R$, we define the discrete intervals $[a,b]_{\BB Z} := [a, b]\cap \BB Z$ and $(a,b)_{\BB Z} := (a,b)\cap \BB Z$. 
\end{notation}

\begin{notation}\label{def-asymp}
If $a$ and $b$ are two quantities, we write $a\preceq b$ (resp. $a \succeq b$) if there is a constant $C$ (independent of the parameters of interest) such that $a \leq C b$ (resp. $a \geq C b$). We write $a \asymp b$ if $a\preceq b$ and $a \succeq b$. 
\end{notation}

\begin{notation} \label{def-o-notation}
If $a$ and $b$ are two quantities which depend on a parameter $x$, we write $a = o_x(b)$ (resp. $a = O_x(b)$) if $a/b \rta 0$ (resp. $a/b$ remains bounded) as $x \rta 0$ (or as $x\rta\infty$, depending on context). We write $a = o_x^\infty(b)$ if $a = o_x(b^s)$ for each $s \in\BB R$. 
\end{notation}

Unless otherwise stated, all implicit constants in $\asymp, \preceq$, and $\succeq$ and $O_x(\cdot)$ and $o_x(\cdot)$ errors involved in the proof of a result are required to satisfy the same dependencies as described in the statement of said result.

\subsubsection{Inventory accumulation model} 
\label{sec-burger-prelim'}
 
Let $p\in (0,1/2)$. We will always treat $p$ as fixed and do not make dependence on $p$ explicit. As explained in~\cite[Section 4.2]{shef-burger}, the parameter $p$ corresponds to an FK-weighted map of parameter $q=4p^2/(1-p)^2$, which is conjectured to converge in the scaling limit to a $\gamma$-LQG surface decorated by a $\op{CLE}_\kappa$ with $\kappa \in (4,8)$ and $\gamma \in (0,2)$ satisfying
\eqb \label{eqn-p-kappa}
p = \frac{\sqrt{2 + 2\cos (8\pi/\kappa)}}{2 + \sqrt{2 + 2\cos (8\pi/\kappa)}} \quad \op{and} \quad  \gamma  = \frac{16}{\kappa} .
\eqe

Let $\Theta:= \{\tc{H} , \tc{C} , \tb{H} , \tb{C} , \tb{F}\}$. We view elements of $\Theta$ as representing a hamburger, a cheeseburger, a hamburger order, a cheeseburger order, and a flexible order, respectively. The set $\Theta$ generates a semigroup, which consists of the set of all finite words in elements of $\Theta$, modulo the relations
\eqb \label{eqn-theta-ful}
\tc C \tb C = \tc H \tb H = \tc C \tb F = \tc H \tb F = \emptyset  \quad \text{(order fulfilment)}
\eqe
 and
\eqb\label{eqn-theta-com}
\tc C \tb H = \tb H \tc C ,\qquad \tc H \tb C = \tb C \tc H \quad \text{(commutativity)} .
\eqe 
Given a word $x$ consisting of elements of $\Theta$, we denote by $\mathcal R(x)$ the word reduced modulo the above relations, with all burgers to the right of all orders. We also write $|x|$ for the number of symbols in $x$. We definite a probability measure on $\Theta$ by
\eqb \label{eqn-theta-prob}
\BB P\left(\tc{H}\right) = \BB P\left(     \tc{C} \right) = \frac14 ,\quad  \BB P\left( \tb{H} \right) = \BB P\left( \tb{C} \right) = \frac{1-p}{4} ,\quad \BB P\left(\tb{F} \right) =  \frac{p}{2} .
\eqe 
 
Let $X = \dots X_{-1} X_0 X_1 \dots$ be an infinite word with each symbol sampled independently according to the probabilities~\eqref{eqn-theta-prob}. For $a \leq b\in \BB R$, let
\eqb \label{eqn-X(a,b)}
X(a,b) := \mcl R\left(X_{\lfloor a\rfloor} \dots X_{\lfloor b\rfloor} \right) .
\eqe 
We adopt the convention that $X(a,b) = \emptyset$ if $b < a$.  
 
By \cite[Proposition 2.2]{shef-burger}, it is a.s.\ the case that the ``infinite reduced word" $X(-\infty,\infty)$ is empty, i.e.\ each symbol $X_i$ in the word $X$ has a unique match which cancels it out in the reduced word.

\begin{notation}\label{def-match-function}
For $i \in \BB Z$ we write $\phi(i)$ for the index of the match of $X_i$. 
\end{notation}

\begin{notation} \label{def-theta-count}
For $\theta\in \Theta$ and a word $x$ consisting of elements of $\Theta$, we write $\mcl N_{\theta}(x)$ for the number of $\theta$-symbols in $x$. We also let
\alb 
d(x)  := \mcl N_{\tc H}(x) - \mcl N_{\tb H}(x) ,\quad 
d^*(x) := \mcl N_{\tc C}(x) - \mcl N_{\tb C}(x),\quad 
D(x)  := \left(d(x) , d^*(x)\right) .
\ale
\end{notation}

The reason for the notation $d$ and $d^*$ is that these functions (applied to segments of the word $Y$, defined just below) give the distances from the root edge in the tree and dual tree which encode the collection of loops in the bijection of \cite[Section 4.1]{shef-burger}. 
  
For $i\in\BB Z$, we define $Y_i = X_i$ if $X_i \in \{\tc{H} , \tc{C} , \tb{H} ,  \tb{C}\}$; $Y_i = \tb{H}$ if $X_i = \tb F$ and $X_{\phi(i)} = \tc{H}$; and $Y_i = \tb{C}$ if $X_i = \tb F$ and $X_{\phi(i)} = \tc{C}$. For $a\leq b \in \BB R$, define $Y(a,b)$ as in~\eqref{eqn-X(a,b)} with $Y$ in place of $X$.  
 
For $n \geq0$, define $ d(n) =  d(Y(1,n))$ and for $n<0$, define $ d(n) = -  d(Y( n+1 , 0))$. Define $d^*(n)$ similarly. Extend each of these functions from $\BB Z$ to $\BB R$ by linear interpolation. 
Let 
\eqb \label{eqn-discrete-path}
  D(t) := (d(t) , d^*(t)) .
\eqe  
For $n\in\BB N$ and $t\in \BB R$, let 
\eqb \label{eqn-Z^n-def}
U^n(t) := n^{-1/2} d (n t) ,\quad V^n(t) := n^{-1/2} d^*(n  t) , \quad Z^n(t) := (U^n(t) , V^n(t) ) .
\eqe 
We note that the condition that $X(1,2n) =\emptyset$ is equivalent to the condition that $Z^n([0,2]) \subset [0,\infty)^2$ and $Z^n(2) = 0$. 

Let $Z= (U,V)$ be a two-sided two-dimensional Brownian motion with $Z(0) = 0$ and variances and covariances at each time $t\in \BB R$ given by 
\eqb \label{eqn-bm-cov}
\op{Var}(U(t) ) = \frac{1-p}{2} |t| \quad \op{Var}(V(t)) = \frac{1-p}{2} |t| \quad \op{Cov}(U(t) , V(t) ) = \frac{p}{2} |t| .
\eqe
It is shown in \cite[Theorem 2.5]{shef-burger} that as $n\rta \infty$, the random paths $Z^n$ defined in~\eqref{eqn-Z^n-def} converge in law in the topology of uniform convergence on compacts to the random path $Z$ of~\eqref{eqn-bm-cov}.

There are several stopping times for the word $X$ which we will use throughout this paper. Namely, let
\eqb \label{eqn-I-def}
I := \inf\left\{i \in \BB N \,:\, \text{$X(1,i)$ contains an order}\right\} ,
\eqe
so that $I$ is a stopping time for $X$, read forward, and $\{I > n\}$ is the event that $X(1,n)$ contains no orders. For $m\in\BB N$, let
\eqb \label{eqn-J^H-def}
J_m^H := \inf\left\{j \in \BB N \,:\, \mcl N_{\tc H}\left(X(-j,-1)\right) = m\right\} ,\quad L_m^H := d^*\left(X(-J_m^H,-1)\right) 
\eqe
be, respectively, the $m$th time a hamburger is added to the stack when we read $X$ backward and the number of cheeseburgers minus the number of cheeseburger orders in $X(-J_m^H,-1)$. Define $J_m^C$ and $L_m^C$ similarly with the roles of hamburgers and cheeseburgers interchanged. 
Then $J_m^H$ and $J_m^C$ are stopping times for $X$, read backward. 
Furthermore, by the strong Markov property that words $X_{-J_m^H} \dots X_{-J_{m-1}^H-1}$ for $m\in\BB N$ are iid, and each of the reduced words $X(-J_m^H,-J_{m-1}^H-1)$ contains exactly one hamburger and no hamburger orders or flexible orders. 
 
Let 
\eqb \label{eqn-cone-exponent}
 \mu := \frac{\pi}{2\left( \pi - \arctan \frac{\sqrt{1-2p} }{p} \right) }  =  \frac{\kappa}{8} \in (1/2,1) ,  \quad \mu' := \frac{\pi}{2\left( \pi + \arctan \frac{\sqrt{1-2p} }{p} \right) }   = \frac{\kappa }{4(\kappa -2)} \in (1/3,1/2) ,
\eqe 
where here $p$ and $\kappa$ are related as in~\eqref{eqn-p-kappa}. The parameters $\mu$ and $\mu'$ appear as exponents for several probabilities related to the inventory accumulation model studied in this paper. See~\cite{gms-burger-cone,blr-exponents} as well as the later results of this paper.

\subsection{Statements of main results}
\label{sec-main-results}

Here we will state the main results of this article. The following event will play a key role throughout the paper. 

\begin{defn} \label{def-local-event}
For $n , h , c \in \BB N$, we denote by $\mcl E_n^{h,c}$ the event that $X(1,n)$ contains no orders (i.e.\ $I>n$), $h$ hamburgers, and $c$ cheeseburgers.
\end{defn}

In terms of the path $Z^n$ of~\eqref{eqn-Z^n-def}, $\mcl E_n^{h,c}$ is the event that $Z^n$ stays in the first quadrant for one unit of time and satisfies $Z^n(1) = (n^{-1/2} h , n^{-1/2} c)$. 
Our first main result is the following scaling limit result for the path $Z^n$ of~\eqref{eqn-Z^n-def} conditioned on the event $\mcl E_n^{h,c}$. 

\begin{thm} \label{thm-local-conv}
Fix $C>1$. For each $\ep > 0$, there exists $n_* \in \BB N$ such that the following is true. For each $n\geq n_*$ and each $(h,c) \in \left[C^{-1} n^{1/2} , C n^{1/2}\right]_{\BB Z}^2$, the Prokhorov distance (with respect to the uniform metric) between the conditional law of $Z^n$ given the event $\mcl E_n^{h,c}$ of Definition~\ref{def-local-event} and the law of a correlated Brownian motion $Z$ as in~\eqref{eqn-bm-cov} conditioned to stay in the first quadrant for one unit of time and satisfy $Z(1) = (n^{-1/2} h , n^{-1/2} c)$ is at most $\ep$. 
\end{thm} 

See Section~\ref{sec-bm-cond} for a precise definition of a Brownian motion under the conditioning of Theorem~\ref{thm-local-conv}. Theorem~\ref{thm-local-conv} implies in particular that for any sequence of pairs $(h_n , c_n) \in \BB N^2$ such that $n^{-1/2} h_n \rta u > 0$ and $n^{-1/2} c_n \rta v > 0$, the conditional law of $Z^n$ given $\mcl E_n^{h_n,c_n}$ converges as $n\rta\infty$ to a correlated Brownian motion $Z$ as in~\eqref{eqn-bm-cov} conditioned to stay in the first quadrant for one unit of time and satisfy $Z(1) = (u,v)$. 
Theorem~\ref{thm-local-conv} extends the scaling limit results~\cite[Theorem 2.5]{shef-burger} (for the unconditioned law of $Z^n$) and~\cite[Theorem A.1]{gms-burger-cone} (for the path $Z^n$ conditioned to stay in the first quadrant, but without its location at time 1 specified). 
  
If we could allow $(h,c) = (0,0)$ in Theorem~\ref{thm-local-conv}, we would obtain convergence of the path $Z^n$ in the finite-volume version of Sheffield's bijection, which corresponds to a random planar map on the sphere. However, treating this case will take quite a bit of additional work, both because of the additional conditioning near the tip of the path (it has to stay in the first quadrant despite being very close to the origin) and because difficulties resulting from the presence of flexible orders. The $(h,c) = (0,0)$ case will be treated in the sequel~\cite{gms-burger-finite} to this paper. Theorem~\ref{thm-local-conv} is in some sense an intermediate step toward a proof of convergence of the path $Z^n$ conditioned on empty reduced word, but we will actually use only the estimates involved in the proof of Theorem~\ref{thm-local-conv} in~\cite{gms-burger-finite}, not the statement of Theorem~\ref{thm-local-conv} itself. 

Theorem~\ref{thm-local-conv} is also noteworthy in that it gives a scaling limit statement for a non-Markovian random walk conditioned to stay in a cone and end up at a particular point. We remark that there is a body of existing literature concerning scaling limits of random walks in dimension $\geq 2$ with independent increments conditioned to stay in a cone. See~\cite{shimura-cone-walk,garbit-cone-walk,dw-cones,dw-limit} and the references therein.  
%\cite{liggett-bridge,bolt-pos-walk,iglehart-pos-walk,doney-pos-walk,dur-ig-mil-meander,cc-pos-bridge,sohier-excursion,shimura-cone-walk,garbit-cone-walk,dw-cones,gms-burger-cone,dw-limit} 

\begin{remark}
As a consequence of \cite[Theorem 1.8]{gms-burger-cone} and Theorem~\ref{thm-local-conv}, one can also obtain an analogue of the cone time convergence statement \cite[Theorem 1.8]{gms-burger-cone} in the setting of Theorem~\ref{thm-local-conv}. A very similar argument will be given in~\cite{gms-burger-finite} to prove the analogous statement when we condition on $\{X(1,2n) = \emptyset\}$, so we do not give the details here. 
\end{remark}
  
Our second main result gives the exponent for the probability of the event that a word of length $2n$ sampled according to the probabilities~\eqref{eqn-theta-prob} reduces to the empty word. 
 
\begin{thm} \label{thm-empty-prob}
For $n\in\BB N$, we have
\eqbn
\BB P\left(X(1,2n) = \emptyset \right) = n^{-1-2\mu + o_n(1)} ,
\eqen
with $\mu$ as in~\eqref{eqn-cone-exponent}.
\end{thm}

\begin{comment}
\begin{remark} \label{remark-borot}
Knowing the asymptotics of $\BB P\left(X(1,2n) = \emptyset \right)$ is equivalent to knowing the asymptotics of the partition function of critical FK planar maps. We learned from G. Borot~\cite{borot-comm} that one can rigorously obtain an asymptotic formula for this partition function from certain results in the theoretical physics literature, which leads to the stronger statement that $\BB P\left(X(1,2n) = \emptyset \right)= (c + o_n(1)) n^{-1-2\mu}$ for a deterministic constant $c > 0$. These facts are explained in more detail in Section~\ref{sec-partition-function}. However, our proof of Theorem~\ref{thm-empty-prob} is independent of the arguments in the aformentioned physics literature and involves several intermediate estimates which will also be used in~\cite{gms-burger-local,gwynne-miller-cle}. 
\end{remark}
\end{comment}

Theorem~\ref{thm-empty-prob} confirms a prediction of Sheffield~\cite[Section 4.2]{shef-burger} that $\BB P\left(X(1,2n) = \emptyset \right)$ has polynomial decay, and in fact yields the exact tail exponent. We note that a trivial, but far from optimal, polynomial upper bound of $n^{-3/2+o_n(1)}$ follows from the fact that the total number of burgers in $X(1,i)$ minus the total number of orders in $X(1,i)$ evolves as a simple random walk. 
As we will explain in Section~\ref{sec-partition-function}, Theorem~\ref{thm-empty-prob} is equivalent to a certain estimate for the partition function of critical FK planar maps. This gives an alternative interpretation of the theorem and suggests an alternative proof based on results in the theoretical physics literature. 

In the course of proving Theorems~\ref{thm-local-conv} and~\ref{thm-empty-prob}, we will also prove several local estimates for quantities associated with the inventory accumulation model, i.e.\ estimates for the probability that a certain random reduced word contains a specified number of symbols of a given type. Local estimates like the ones in this paper are necessary for studying finer properties of the model, like questions about the event of Definition~\ref{def-local-event} or the event $\{X(1,2n) =\emptyset\}$. This paper provides a toolbox of local estimates which are applicable whenever one is interested in small scale properties of the word $X$, and illustrates how to use such estimates. 
The local estimates of this paper will also be used in the sequel~\cite{gms-burger-finite} to this work to obtain scaling limit results for the path $Z^n$ of~\eqref{eqn-Z^n-def} when we condition on $\{X(1,2n) =\emptyset\}$. 

Here we summarize the local estimates which are proven in this paper and highlight some particular estimates.
\begin{itemize}
\item In Section~\ref{sec-J^H-local}, we prove local estimates for the pairs $(J_m^H ,L_m^H)$ of~\eqref{eqn-J^H-def}. See in particular Proposition~\ref{prop-J^H-local} for a general estimate for $\BB P\left( (J_m^H ,L_m^H) = (k,l)\right)$ for given $m , k, l \in \BB N$ and Section~\ref{sec-J^H-reg} for extensions of this estimate which either include regularity conditions or concern unusually large or unusually small values of $k$ and $l$. 
\item In Section~\ref{sec-no-order-local} we prove local estimates when we condition on the event that $I >n$ (recall~\eqref{eqn-I-def}), i.e., the reduced word $X(1,n)$ contains no orders. In particular, lower and upper bounds for the probability of the event $\mcl E_n^{h,c}$ of Definition~\ref{def-local-event} are obtained in Propositions~\ref{prop-no-order-lower} and~\ref{prop-no-order-upper}, respectively. 
\item In Section~\ref{sec-few-order-local}, we will prove Proposition~\ref{prop-corner-local}, which gives an estimate for the probability that there exists a time $j\in\BB N$ which is close to $n$ for which the reduced word $X(-j,-1)$ contains few orders and approximately $h$ hamburgers and $c$ cheeseburgers (i.e., the probability of an approximate version of $\mcl E_n^{h,c}$). 
\end{itemize}

\subsection{Partition function}
\label{sec-partition-function}

In this subsection we will explain the relationship between Theorem~\ref{thm-empty-prob} and the partition function for critical FK planar maps. For $n\in\BB N$, let $\mcl M_n$ be the set of pairs $(M , S)$ consisting of a planar map $M$ with $n$ edges and a distinguished subset $S$ of the set of edges of $M$. Also let $K(S)$ be equal to the number of connected components of $S$ plus the number of complementary connected components of $S$. For $q\in (0,4)$, let
\eqb \label{eqn-partition-function}
\mcl Z_n := \sum_{(M , S) \in \mcl M_n} q^{K(S)/2}
\eqe 
be the critical FK planar map partition function.  

\begin{lem}  \label{prop-partition-function-equiv}
Let $q\in (0,4)$ and let $p := \sqrt q/(2+\sqrt q)  \in (0,1/2) $, so that $2p/(1-p) = \sqrt q$. If we sample $X$ as in Section~\ref{sec-burger-prelim'} for this choice of $p$, then 
\eqb \label{eqn-partition-function-equiv}
\mcl Z_n = \frac12 8^n (2+\sqrt q)^n   n^{-1 }  \BB P\left(X(1,2n) = \emptyset \right)   . 
\eqe 
\end{lem}
\begin{proof} 
Let $\mcl M_n'$ be the set of triples $(M , e_0 , S)$ with $(M,S) \in \mcl M_n$ and $e_0$ an oriented root edge for $M$. 
By~\cite{shef-burger}, there is a bijection between $\mcl M_n'$ and the set of words of length $2n$ consisting of elements of $\Theta$ which reduce to the empty word. Given $(M,S) \in \mcl M_n$, let $x = x(M,e_0,S)$ be the corresponding word. The quantity $K(S)$ is equal to the number of loops $\mcl L$ separating clusters and dual clusters on $M$, which in turn is equal to the number of $\tb F$-symbols in $x$. If $X$ is the bi-infinite word from Section~\ref{sec-burger-prelim'}, then 
\eqbn
 \BB P\left( X(1,2n) = \emptyset \right) =  \left( \frac{1-p}{16} \right)^n    \sum_{x : \mcl R(x) =\emptyset} \left(\frac{2p}{1-p} \right)^{  \mcl N_{\tb F}(x)} 
\eqen
where the sum is over all words $x$ of length $2n$ which reduce to the empty word. Note that here we used the probabilities~\eqref{eqn-theta-prob} and the fact that $x$ has $n$ burgers and $n$ orders. Applying Sheffield's bijection, recalling the relationship between $p$ and $q$, and re-arranging gives
\eqbn
  \sum_{(M , e_0 , S) \in \mcl M_n'} q^{K(S)/2}  =   8^n (2+\sqrt q)^n  \BB P\left( X(1,2n) = \emptyset \right)         ,
\eqen
with $\mu$ as in~\eqref{eqn-partition-function}.
The map from $\mcl M_n'$ to $\mcl M_n$ given by forgetting the root is $2n$-to-1, so dividing by $2n$ gives~\eqref{eqn-partition-function}.
\end{proof}

\begin{comment}
Solve[2 p/(1 - p) == Sqrt[q], p]
FullSimplify[16/(1 - p) /. {p -> Sqrt[q]/(2 + Sqrt[q])}, 0 < q < 4]
FullSimplify[
 Pi/(2 (Pi - ArcTan[Sqrt[1 - 2 p]/p])) /. {p -> 
    Sqrt[q]/(2 + Sqrt[q])}, 0 < q < 4]
\end{comment}

In light of Lemma~\ref{prop-partition-function-equiv}, the statement of Theorem~\ref{thm-empty-prob} is equivalent to the statement that
\eqb \label{eqn-partition-function-asymp}
\mcl Z_n = 8^n (2+\sqrt q)^n   n^{-2-2\mu + o_n(1)} ,\quad \op{where}\: \mu = \frac{\pi}{2(\pi - \arctan\sqrt{  4/q - 1})} .
\eqe 
It was pointed out to us by G. Borot in private communication~\cite{borot-comm} that it should be possible to obtain a stronger version of~\eqref{eqn-partition-function-asymp} (with $(c + o_n(1)) n^{-2-2\mu}$ in place of $n^{-2-2\mu + o_n(1)}$) using various results in the theoretical physics literature.
As explained in~\cite[Sections 2 and 3]{bbg-loop-models}, the critical FK model essentially reduces to the $O(\sqrt q)$ model, so its partition function has the same functional equations and critical exponents. 
The asymptotics of the partition function of the $O(\sqrt q)$ model on a planar map were first computed in the physics literature~\cite{gk-o-n-model,ez-o-n-model} by using random matrix models to non-rigorously derive a certain functional equation which was then solved rigorously. 
The FK model on a planar map is studied in~\cite{bbg-loop-models} from the point of view of analytic combinatorics. There, it is proven rigorously that $\mcl Z_n$ satisfies the functional equation of~\cite{gk-o-n-model,ez-o-n-model} using the results of~\cite[Section 6]{bbg-bending} (this can also be done using~\cite[Appendix A]{borot-eynard-o-n-enumeration}); see~\cite[Equations 3.22-3.23]{bbg-loop-models}. This leads to a rigorous derivation of the aforementioned stronger form of~\eqref{eqn-partition-function-asymp}. 
However, our proof of Theorem~\ref{thm-empty-prob} is more self-contained than this potential approach (we use only the results of~\cite{shef-burger,gms-burger-cone}, and elementary facts from probability theory) and involves several intermediate estimates which are of independent interest and will also be used in~\cite{gms-burger-local,gwynne-miller-cle}.

\subsection{Preliminaries}
\label{sec-prelims}

\subsubsection{Brownian motion conditioned to stay in the first quadrant}
\label{sec-bm-cond}

The statement of Theorem~\ref{thm-local-conv} refers to a correlated Brownian motion as in~\eqref{eqn-bm-cov} conditioned to stay in the first quadrant for one unit of time and satisfy $Z(1) = (u,v)$ for some fixed $(u,v) \in  (0,\infty)^2$. In this subsection we will describe how to make sense of this object. 

We first recall how to make sense of a correlated Brownian motion $Z$ conditioned to stay in the first quadrant (see \cite[Section 3.1]{gms-burger-cone} and \cite{shimura-cone} for more detail). 
In \cite{shimura-cone}, Shimura constructs an uncorrelated two-dimensional Brownian motion conditioned to stay in the cone $\{z\in\BB C \,:\, 0 \leq \op{arg} z  \leq \theta\}$ for one unit of time. By choosing $\theta = \theta(p)$ appropriately and applying a linear transformation which takes this cone to the first quadrant, we obtain a path $\wh Z$ which we interpret as the correlated two-dimensional Brownian motion $Z$ in~\eqref{eqn-bm-cov} conditioned to stay in the first quadrant for one unit of time. We note that the law of $\wh Z$ is uniquely characterized by the conditions that $\wh Z(t)$ a.s.\ lies in the interior of the first quadrant at each fixed time $t\in (0,1)$; and for each $t\in (0,1)$, the conditional law of $\wh Z|_{[t,1]}$ given $\wh Z|_{[0,t]}$ is that of a Brownian motion with variances and covariances as in~\eqref{eqn-bm-cov} started from $\wh Z(t)$ and conditioned on the positive probability event that it stays in the first quadrant for $1-t$ units of time (see \cite[Lemma 3.1]{gms-burger-cone}). 

Given $(u,v) \in  (0,\infty)^2$, the law of a Brownian motion $ Z$ conditioned to stay in the first quadrant for one unit of time and satisfy $Z(1) = (u,v)$ is the regular conditional law of the path $\wh Z$ described above given $\{\wh Z(1) = (u,v)\}$. This law can be sampled from as follows. First fix $t\in (0,1)$ and sample $ Z|_{[0,t]}$ from the law of a Brownian motion with variances and covariances as in~\eqref{eqn-bm-cov} conditioned to stay in the first quadrant for $t$ units of time weighted by $ f_{1-t}^{ Z(t)}(u,v)   $, where for $z\in (0,\infty)^2$, $f_{1-t}^z$ denotes the density with respect to Lebesgue measure of the law of a Brownian motion as in~\eqref{eqn-bm-cov} started from $z$ conditioned to stay in the first quadrant for $1-t$ units of time (see~\cite[Section 3]{shimura-cone} for a formula for the density of an uncorrelated two-dimensional Brownian motion conditioned to stay in a cone of given opening angle; a formula for $f_{1-t}^z$ can be obtained by applying a linear transformation). 
Then, conditioned on $  Z|_{[0,t]}$, sample a Brownian bridge from $  Z(t)$ to $(u,v)$ in time $1-t$ conditioned to stay in the first quadrant, and concatenate this Brownian bridge with our given realization of $Z|_{[0,t]}$. 

By Brownian scaling, one obtains the conditional law of a Brownian motion as in~\eqref{eqn-bm-cov} conditioned to stay in the first quadrant for a general $t >0$ units of time and satisfy $Z(t) = (u,v) \in (0,\infty)^2$. 
We note that since the density $f_{1-t}^z$ is continuous and the law of a Brownian bridge depends continuously on its parameters, this law depends continuously on $t$, $u$, and $v$. 

The following is the analogue of \cite[Lemma 3.1]{gms-burger-cone} when we further condition on the terminal point of the path, and follows from~\cite[Lemma 3.4]{sphere-constructions}. 

\begin{lem} \label{prop-bm-tip}
Fix $(u,v) \in (0,\infty)^2$. Let $\wh Z$ have the law of a correlated Brownian motion as in~\eqref{eqn-bm-cov} conditioned to stay in the first quadrant for one unit of time and end up at $(u,v)$. Then $\wh Z$ satisfies the following two conditions.
\begin{enumerate}
\item For each fixed $t\in [0,1]$, a.s.\ $\wh Z(t) \in (0,\infty)^2$. \label{item-bm-tip-pos}
\item For each $t  \in (0,1)$, the regular conditional law of $\wh Z|_{[t,1]}$ given $\wh Z|_{[0,t]}$ is that of a Brownian bridge from $\wh Z(t )$ to $(u,v)$ in time $1-t$ with variances and covariances as in~\eqref{eqn-bm-cov} conditioned on the (a.s.\ positive probability) event that it stays in the first quadrant.  \label{item-bm-tip-markov}
\end{enumerate}
If $\wt Z : [0,1]\rta [0,\infty)^2$ is another random continuous path satisfying the above two conditions, then $\wt Z\eqD \wh Z$. 
\end{lem}

\subsubsection{Regular variation}
\label{sec-reg-var}

A function $f : [1,\infty) \rta (0,\infty)$ is called \emph{regularly varying of exponent $\alpha \in \BB R$} if for each $\lambda >0$, 
\eqbn
\lim_{t\rta\infty} \frac{f(\lambda  t)}{f(t)} =\lambda^{-\alpha} .
\eqen
The function $f$ is called \emph{slowly varying} if $f$ is regularly varying of exponent 0. 
For example, $t\mapsto \log(1+t)^\beta$ is slowly varying for any $\beta \in \BB R$ and $t\mapsto t^{-\alpha} \log(1+t)^\beta$ is regularly varying of exponent $\alpha$ for any $\alpha,\beta\in\BB R$. 

Every function $f$ which is regularly varying of exponent $\alpha$ can be represented in the form $f(t) = \psi(t) t^{-\alpha}$, where $\psi$ is slowly varying. For each slowly varying function $\psi$, there exists $t_0 > 0$ and bounded functions $a , b : [0,\infty)\rta\BB R$ with $\lim_{t\rta\infty} b(t) = 0$ such that for $t\geq t_0$,
\eqbn
\psi(t) = \exp\left(a(t) + \int_{t_0}^t \frac{b(s)}{s} \, ds\right) .
\eqen
We refer the reader to~\cite{reg-var-book} for more on regularly varying functions.

The following lemmas are proven in \cite[Section A.2]{gms-burger-cone}. We recall the definition of the exponent $\mu$ from~\eqref{eqn-cone-exponent} and the definition of $I$ from~\eqref{eqn-I-def}.  

\begin{lem} \label{prop-I-reg-var}
The law of $I$ is regularly varying with exponent $\mu$, i.e.\ there exists a slowly varying function $\psi_0 : [0,\infty) \rta (0,\infty)$ such that
\eqbn
\BB P\left(I > n \right) = \psi_0(n) n^{-\mu} ,\qquad \forall n\in\BB N .
\eqen
\end{lem}

\begin{lem} \label{prop-P-reg-var}
Let $P$ be the smallest $j \in \BB N$ for which $X(-j,-1)$ contains no orders. Then the law of $P$ is regularly varying with exponent $1-\mu$, i.e.\ there exists a slowly varying function $\psi_1  : [0,\infty) \rta (0,\infty)$ such that
\eqbn
\BB P\left(P > n \right) = \psi_1(n) n^{-(1-\mu)} ,\qquad \forall n\in\BB N .
\eqen
\end{lem} 

We will treat the functions $\psi_0$ and $\psi_1$ as fixed throughout this paper. 
In Lemma~\ref{prop-J^H-law} below, we will prove that the laws of $J^H$ and $L^H$ defined as in~\eqref{eqn-J^H-def} are regularly varying with exponetns $1/2$ and $1$, respectively, with no slow-varying correction.

\subsection{Outline}
\label{sec-outline}

The remainder of this article is structured as follows. In Section~\ref{sec-J^H-local}, we will use the bivariate local limit theorem of~\cite{doney-bivariate} to prove various local estimates for the pairs $(J_m^H,L_m^H)$ of~\eqref{eqn-J^H-def}. The times $J_m^H$ are especially useful because each $X(-J_m^H ,-1)$ contains no flexible orders; and because the law of the pair $(J_1^H , L_1^H)$ is in the normal domain of attraction for a bivariate stable law (see Lemma~\ref{prop-J^H-law} below). 

In Section~\ref{sec-no-order-local}, we will use the results of Sections~\ref{sec-J^H-local} to obtain local estimates for the probability that the reduced word $X(1,n)$ contains no orders and a particular number of burgers of each type. From these estimates we will deduce Theorem~\ref{thm-local-conv}. 

In Section~\ref{sec-few-order-local}, we will use the estimates of Sections~\ref{sec-J^H-local} and~\ref{sec-no-order-local} to obtain estimates for the probability that a word has few orders and approximately a given number of burgers of each type. These estimates will be used in the proof of the upper bound in Theorem~\ref{thm-empty-prob} as well as in~\cite{gms-burger-finite}. 

In Section~\ref{sec-empty-prob}, we will complete the proof of Theorem~\ref{thm-empty-prob}. The proof of the lower bound follows from a relatively straightforward argument which is similar to those given in \cite[Section 3]{gms-burger-cone} (in fact, a version of this argument appeared in a previous version of \cite{gms-burger-cone}). The proof of the upper bound is more complicated, and requires the estimates of Section~\ref{sec-few-order-local} as well as some additional estimates, including a modification of the estimate \cite[Lemma 3.7]{gms-burger-cone} for the number of flexible orders in a reduced word. 

For the convenience of the reader we include in Appendix~\ref{sec-index} an index of commonly used symbols along with reminders of their definitions and the locations where they are defined.

\section{Local estimates for times when hamburgers are added}
\label{sec-J^H-local}

In this section, we will consider a ``local limit" type result (i.e.\ a uniform convergence statement for densities) for the pairs $(J_m^H , L_m^H)$ introduced in~\eqref{eqn-J^H-def}. This result will turn out to be a straightforward consequence of the bivariate local limit theorem for stable laws proven in~\cite{doney-bivariate}. We will then prove some refinements on this result in Section~\ref{sec-J^H-reg}. The estimates of this section are a key input in the proofs of Theorems~\ref{thm-empty-prob} and~\ref{thm-local-conv}. However, this section also has the following broader purpose. As we will see in Sections~\ref{sec-no-order-local} and~\ref{sec-few-order-local} below, local estimates for the pairs $(J_m^H , L_m^H)$ are the basic tools for the proofs of many other local estimates and scaling limit results for the inventory accumulation model considered in this paper; see also \cite{gms-burger-finite,gwynne-miller-cle} for more applications of these estimates. This section can be read as a general collection of estimates for the pairs $(J_m^H , L_m^H)$.

We note that (by symmetry) all of the results of this section are still valid if we instead consider the pairs $(J_m^C , L_m^C)$ defined just below~\eqref{eqn-J^H-def}. However, for the sake of brevity we state our results only for the pairs $(J_m^H , L_m^H)$. 
 
\subsection{Local limit theorem}
\label{sec-J^H-local-sub}

The starting point of all of the estimates in this section is the following straightforward consequence of the unconditioned Brownian motion scaling limit result~\cite[Theorem 2.5]{shef-burger}. 

\begin{lem} \label{prop-J^H-law}
Define $J_m^H$ and $L_m^H$ for $m\in\BB N$ as in~\eqref{eqn-J^H-def}. Also let $\tau$ be the smallest $t>0$ for which $U(-t) = - 1$. Then we have the following convergence in law:
\eqb \label{eqn-J^H-law}
 \left(m^{-2} J_m^H , m^{-1} L_m^H  \right) \rta \left(\tau ,  V(-\tau) \right) . 
\eqe 
Furthermore, there is a constant $a_0 > 0$ such that  
\eqb \label{eqn-J^H-tail}
\BB P\left(J^H_1  > n \right) = (a_0+o_n(1)) n^{-1/2 }   
\eqe  
and there are constants $a_1,a_2>0$ such that  
\begin{align} \label{eqn-d-J^H-tail}
\BB P\left(L_1^H   > n \right) = (a_1+o_n(1)) n^{-1}  \quad \op{and} \quad \BB P\left(L_1^H  \leq -n \right) = (a_2+o_n(1)) n^{-1  }  .
\end{align} 
\end{lem}
\begin{proof}
Suppose we have (using \cite[Theorem 2.5]{shef-burger} and the Skorokhod theorem) coupled $(Z^n)$ with $Z$ in such a way that $Z^n\rta Z$ uniformly on compacts a.s. Let $\tau_m := m^{-2} J_m^H$. The time $\tau_m$ is the smallest $t>0$ for which $U^{m^2}(-t) =1$. Since $Z$ a.s.\ crosses the line $\{(x,y) \in \BB R^2 \,:\, x = 1\}$ immediately after hitting this line when run in the reverse direction, it follows that a.s.\ $\tau_m \rta \tau$. By uniform convergence, a.s.\ $V(-\tau_m) \rta V(-\tau)$. Thus~\eqref{eqn-J^H-law} holds. 
%For each $\ep > 0$, we can find $\delta>0$ and $t\in [\tau , \tau+\ep]$ such that $Z(t) \geq 1+\delta$. Hence for sufficiently large $m$ we have $\tau_m \leq \tau+\ep$. By compactness we can find a subsequential limit $\wt\tau$ of the times $\tau_m$. We have $U(\wt\tau)  = 1$, so $\wt\tau \geq \tau$. By our above argument $\wt\tau \leq \tau$. 

The time $\tau$ has the law of a stable random variable with index $1/2$ and skewness parameter $1$, and~\eqref{eqn-J^H-law} implies that $J_m^H$ is in the normal domain of attraction for this law. By the elementary theory of stable processes, we infer that~\eqref{eqn-J^H-tail} holds (see, e.g.,~\cite{gk-limit-theorems} or the remark just below~\cite[Theorem 3.7.2]{durrett}).

By the Markov property of $Z$ and Brownian scaling, we find that $V(-\tau)$ has a $1$-stable distribution. Since each of the words $X(-J^H_k , -J^H_{k-1}-1)$ contains no $\tb F$'s, we infer that
\eqbn
m^{-1} L_m^H  = - \frac{1}{m} \sum_{k=1}^m d (X(-J^H_k , -J^H_{k-1}-1)) .
\eqen 
Since $m^{-1} L_m^H \rta V(-\tau)$, we conclude as above that~\eqref{eqn-d-J^H-tail} holds.
\end{proof}

From Lemma~\ref{prop-J^H-law}, we obtain the following statement, which will play a key role in the remainder of this paper.

\begin{prop}\label{prop-J^H-local}
Let $g$ be the joint probability density function of the pair $(\tau , V(-\tau))$ from Lemma~\ref{prop-J^H-law}. Then we have  
\eqb \label{eqn-J^H-local}
\lim_{m\rta\infty} \sup_{(k, l) \in \BB N\times \BB Z} \left| m^{ 3 } \BB P\left((J_m^H , L_m^H) = (k,l) \right)  - g\left(\frac{k}{m^2} ,   \frac{l}{m}\right)    \right|  = 0  .
\eqe 
\end{prop}
\begin{proof}
The words $X_{-J_m^H}\dots X_{-J_{m-1}^H-1}$ for $m\in\BB N$ are iid (where here we set $J_0^H = 0$). Furthermore, since each $X(-J_m^H , -J_{m-1}^H-1)$ contains no $\tb F$-symbols, we have
\eqbn
L_m^H = \sum_{k=1}^m d^*\left(X(-J_m^H , -J_{m-1}^H-1)\right)   .
\eqen
The local limit result~\eqref{eqn-J^H-local} now follows from Lemma~\ref{prop-J^H-law} and the local limit theorem for bivariate stable laws~\cite[Theorem 1]{doney-bivariate}. 
\end{proof}

When we apply the estimate of Proposition~\ref{prop-J^H-local}, it will be convenient to have a formula for the limiting density $g$. 

\begin{lem} \label{prop-g-form}
Let $g$ be as in Proposition~\ref{prop-J^H-local}. Then 
\eqb \label{eqn-g-form}
g(t,v) =   a_0 t^{-2} \exp\left(-\frac{a_1 + a_2  (v+a_3)^2  }{ t} \right) ,\quad \forall (t,v) \in (0,\infty)\times \BB R
\eqe
for constants $a_0 , a_1 , a_2,a_3 > 0$ depending only on $p$. In particular, $g$ is bounded.
\end{lem}
\begin{proof}
The function $g$ is the joint density of $(\tau , V(\tau))$, where $\tau$ is the first time the Brownian motion $U$ in~\eqref{eqn-bm-cov} hits $-1$ and $V$ is the other Brownian motion in~\eqref{eqn-bm-cov}. Thus marginal density of $\tau$ is given by
\eqbn
\BB P\left(t < \tau  < t+ dt \right) = C_1 t^{-3/2} \exp\left( -\frac{a_1}{t}  \right)
\eqen
for constants $C_1 > 0$ and $a_1 >0$ depending only on $p$. Recalling that $U$ and $V$ are positively correlated, we can write $V = \wt V + a_3 U$, where $a_3 > 0$ is a constant depending on $p$ and $\wt V$ is a constant times a standard linear Brownian motion independent from $U$. We have $V(-\tau) = \wt V(-\tau) -a_3$ so the conditional density of $V(\tau)$ given $\tau$ is given by
\eqbn
\BB P\left(v < V(\tau)  < v+ dv \,|\, t < \tau < t+ dt \right) = C_2 t^{-1/2} \exp\left( -\frac{a_2(v + a_3)^2}{t} \right) ,
\eqen
for constants $C_2 > 0$ and $a_2 >0$ depending only on $p$. Combining these formulas yields~\eqref{eqn-g-form}. 
\end{proof}

We next prove an analogue of Lemma~\ref{prop-J^H-law} for (roughly speaking) times at which hamburger orders are added when we run forward. For this we need the following basic fact.

\begin{lem}\label{prop-subsequence-law}
Let $(\xi_j)$ be a sequence of iid non-negative random variables. For $m\in\BB N$, let $S_m := \sum_{j=1}^m \xi_j$. Let $(N_n)$ be an increasing sequence of positive integer-valued random variables such that $N_n/n$ a.s.\ converges to a constant $q > 0$. For $n\in\BB N$, let $\wh S_n :=   S_{N_n}$. Suppose there is a random variable $X$ and an $\alpha > 0$ such that $n^{-\alpha} \wh S_n \rta X$ in law. Then $m^{-\alpha} S_m \rta q^{-\alpha} X$ in law.
\end{lem}  
\begin{proof}
Fix $\ep > 0$ and for $m\in \BB N$, set $n_m := \lfloor q^{-1}(1+\ep) m \rfloor$. 
Since $  N_n/n \rta q$, with probability tending to 1 as $m\rta\infty$ we have $N_{n_m} \geq m$. 
If $N_{n_m} \geq m$ then $\wh S_{n_m} \geq S_m$. Therefore, for each $a>0$, 
\alb
\BB P\left( m^{-\alpha} S_m > a \right) 
&\leq \BB P\left( m^{-\alpha} \wh S_{ n_m } > a \right) + \BB P\left( N_{n_m} < m \right) \\
&\leq \BB P\left( n_m^{-\alpha} \wh S_{n_m} > n_m^{-\alpha} m^\alpha a \right) + o_m(1) \\
&= \BB P\left( q^{-\alpha} X >  (1+\ep)^{-\alpha} a \right) + o_m(1) . 
\ale
Similarly, $\BB P\left( m^{-\alpha} S_m > a \right) \geq \BB P\left( q^{-\alpha} X >  (1-\ep)^{-\alpha} a \right) + o_m(1)$. 
Since $\ep > 0$ is arbitrary, we find that whenever $\BB P(q^{-\alpha} X = a) = 0$, 
\eqbn
\lim_{m\rta\infty} \BB P\left( m^{-\alpha} S_m > a \right) = \BB P\left( q^{-\alpha} X > a \right) ,
\eqen
which implies the desired convergence in law.
\end{proof} 
 
\begin{lem} \label{prop-I^H-law}
For $m\in\BB N$, let $\wt I_m^H$ be the $m$th smallest $i \in \BB N$ for which $X(1,i)$ contains no hamburgers. Also let $\tau$ be the smallest $t>0$ for which $U(t) = - 1$. Then  
\eqb \label{eqn-I^H-law}
 m^{-2} \wt I_m^H  \rta   \left(\frac{4}{1-p} \right)^2 \tau  \qquad \text{in law}.  
\eqe 
Furthermore, there is a constant $b_0 > 0$ such that  
\eqb \label{eqn-I^H-tail}
\BB P\left(\wt I^H_1  > n \right) = (b_0+o_n(1)) n^{-1/2 }    .
\eqe   
\end{lem} 
\begin{proof}
For $m\in\BB N$, let $ I_m^H$ be the smallest $i\in\BB N$ for which $X(1,i)$ contains at least $m$ hamburger orders. For $m\in\BB N$, let $E_m$ be the event that  $\wt I_m^H = \wt I_{m-1}^H + 1$ and $X_{\wt I_m^H} = \tb H$. Also let $N_m := \sum_{k\leq m} \BB 1_{E_k}$. Observe that $E_m$ occurs if and only if $\wt I_m^H = I_k^H$ for some $k\in\BB N$. Therefore $\wt I_{N_m}^H = I_m^H$. Furthermore, the events $E_m$ are independent, and each has probability $(1-p)/4$. By the law of large numbers, we have $ N_m/m \rta (1-p)/4$ a.s. 

Let $\tau_m := m^{-2} I_m^H$. We claim that $\tau_m\rta \tau$ in law. Suppose we have (using \cite[Theorem 2.5]{shef-burger} and the Skorokhod theorem) coupled $(Z^{m^2})$ with a correlated Brownian motion $Z$ as in~\eqref{eqn-bm-cov} in such a way that $Z^{m^2} \rta Z$ uniformly on compacts a.s. Observe that
\eqb \label{eqn-I^H-count}
  -1 - m^{-1} \mcl N_{\tb F}\left( X(1,I_m^H )\right) \leq  U^{m^2}(\tau_m) \leq  - 1  .
\eqe 
For $\delta>0$, let $\tau_\delta$ be the first time $t>0$ such that $U(t) = -1+\delta$. 
For each $\ep > 0$, we can find $\delta > 0$ such that with probability at least $1-\ep$, we have $\tau -\tau_\delta \leq \ep$. By \cite[Lemma 3.7]{gms-burger-cone}, we can find $m_* = m_*(\ep) \in \BB N$ such that for $m\geq m_*$, it holds with probability at least $1-\ep$ that $m^{-1} \mcl N_{\tb F}\left( X(1,I_m^H)\right) \leq \delta$. %Probability that $I_m^H > C m^2$ tends to 0 as $C\rta\infty$ by the BM convergence result, and the probability that $X(1,Cm)$ contains too many $\tb F$'s is small.
 By~\eqref{eqn-I^H-count}, for $m\geq m_*$, it holds with probability at least $1-2\ep$ that $|\tau_m - \tau| \leq \ep$. Since $\ep$ is arbitrary, this implies $\tau_m \rta \tau$ in probability, hence also in law. 

By Lemma~\ref{prop-subsequence-law}, we have~\eqref{eqn-I^H-law}.   
We now obtain~\eqref{eqn-I^H-tail} in exactly the same manner as in Lemma~\ref{prop-J^H-law}. 
\end{proof}

From Lemma~\ref{prop-I^H-law} we deduce the following, which will be used several times below. 
 
\begin{lem} \label{prop-J^H-time}
There is a constant $b_1 > 0$ such that for $n\in\BB N$, 
\eqbn
\BB P\left( \text{$n = J_m^H$ for some $m\in\BB N$} \right) = (b_1 +o_n(1))  n^{-1/2} .
\eqen 
\end{lem}
\begin{proof}
The event that $n = J_m^H$ for some $m\in\BB N$ is the same as the event that $X(-n+1,-1)$ contains no hamburger orders or flexible orders and $X_{-n} = \tc H$. By translation invariance this probability is the same as the probability that $X(1,i)$ contains a hamburger for each $i\in [1,n]_{\BB Z}$, i.e.\ the probability that $\wt I_1^H > n$, with $\wt I_1^H$ as in Lemma~\ref{prop-I^H-law}. The statement of the lemma now follows Lemma~\ref{prop-I^H-law}. 
\end{proof}

\subsection{Regularity and large deviation estimates}
\label{sec-J^H-reg}

In this section we will prove two (closely related) types of results, which sharpen the estimates which come from Proposition~\ref{prop-J^H-local}: estimates for the probability that $(J_m^H , L_m^H) = (k,l)$ when $k$ and $l$ are far from $m^2$ and $m$, respectively; and regularity estimates for the conditional law of $X_1\dots X_{J_m^H}$ given $\{(J_m^H , L_m^H) = (k,l)\}$.
The starting point of our estimates is Proposition~\ref{prop-J^H-local} plus the following tail bound for the pairs $(J_m^H , L_m^H)$.

\begin{lem} \label{prop-J^H-sup}
For $m\in\BB N$ and $R>0$, we have
\begin{align}  
&\BB P\left(J_m^H \geq R  \right) \preceq R^{-1/2} m,  \label{eqn-J^H-sup-big} \\
&\BB P\left(J_m^H \leq R^{-1}  \right) \preceq   e^{- a R m^2  }, \quad \op{and}\label{eqn-J^H-sup-small} \\ 
&\BB P\left(|L_m^H| \geq R   \right) \preceq R^{-1   } m   \label{eqn-J^H-sup-length}
\end{align}
with $a >0$ a universal constant and the implicit constants depending only on $p$. 
\end{lem} 
\begin{proof}
Let $(Y_m)_{m\in\BB N}$ be an iid sequence of positive stable random variables of index $1/2$, with zero centering parameter. 
Equivalently, each $Y_m$ has the law of the first time that a standard linear Brownian motion hits 1. 
By Lemma~\ref{prop-J^H-law} and since the increments $J_m^H - J_{m-1}^H$ are iid, we can find a constant $A >0$, depending only on $p$, and a coupling of two copies $(Y_m^1)$ and $(Y_m^2)$ of the sequence $(Y_m)$ with the sequence $(J_m^H)$ such that $A^{-1} Y_m^1 \leq J_m^H - J_{m-1}^H \leq A Y_m^2+1$ a.s.\ for each $m\in\BB N$. Then with $S_m^i := \sum_{j=1}^m Y_j^i$ for $i\in \{1,2\}$, we have
\eqbn
\BB P\left(J_m^H \geq R   \right) \leq \BB P\left(S_m^2 \geq \frac{R-m}{A}  \right) = \BB P\left(m^{-2} S_m^2 \geq    m^{-2} \frac{R-m}{A} \right) 
\eqen
and
\eqbn
\BB P\left(J_m^H \leq R^{-1}   \right) \leq \BB P\left(S_m^1 \leq A  R^{-1} \right) = \BB P\left(m^{-2} S_m^1 \leq A  R^{-1} m^{-2} \right) .
\eqen
By stability, $m^{-2} S_m^i \eqD Y_1$ for $i\in\{1,2\}$. Hence we obtain
\eqbn
\BB P\left(J_m^H \geq R  \right) \leq \BB P\left(Y_1 \geq  m^{-2} \frac{R-m}{A}  \right) \preceq \min\{ 1 , 0\vee (R-m)^{-1/2} m \} \preceq R^{-1/2} m  
\eqen
and (recalling the representation of $Y_1$ in terms of Brownian motion)
\eqbn
\BB P\left(J_m^H \leq R  \right) \leq \BB P\left(Y_1 \leq A  R^{-1} m^{-2} \right) \preceq e^{- a  R m^2    }
\eqen
for an appropriate constant $a > 0$ as in the statement of the lemma. This yields~\eqref{eqn-J^H-sup-big} and~\eqref{eqn-J^H-sup-small}.

To prove~\eqref{eqn-J^H-sup-length}, we observe that Lemma~\ref{prop-J^H-law} implies that we can find constants $A , B > 0$ and $\beta \in (0,1)$ (depending only on $p$) such that the law of $L_1^H$ is stochastically dominated by the law of $A\wt Y+B$, where $\wt Y$ is a 1-stable random variable with skewness parameter $\beta$ and zero centering parameter.  
%c.f. wikipedia: we can attain any desired left and right tails by varying scaling parameter and skewness. Take centering parameter to be zero so its strictly stable---no additive constant in stability relation. 
Let $(\wt Y_m)_{m\in\BB N}$ be an iid sequence of such 1-stable random variables. Then we can find a coupling of $(L_m^H)$ with $(\wt Y_m)$ such that a.s.\ 
\eqbn
   L_m^H -  L_{m-1}^H \leq   A\wt Y_m +  B \qquad \forall m\in\BB N.
\eqen
With $\wt S_m  =  \sum_{j=1}^m \wt Y_j $, we have $m^{-1} \wt S_m \eqD \wt Y_1 $. Therefore,
\alb
\BB P\left( L_m^H  \geq R   \right) &\leq \BB P\left(\wt S_m \geq \frac{R - \wt B m }{\wt A}  \right)   \\
&=  \BB P\left(\wt Y_1 \geq   m^{-1} \frac{ R - \wt B m}{\wt A}       \right)  \\
&\preceq \min\left\{1,  0\vee (R - \wt B m)^{-1} m  \right\} \preceq R^{-1} m .
\ale 
We similarly obtain $\BB P\left( L_m^H  \leq - R   \right) \preceq R^{-1} m$. 
\end{proof}

Now we will prove a sharper version of the upper bound implicit in Proposition~\ref{prop-J^H-local} which gives stronger estimates when the pair $(k,l)$ is in some sense unusual or if the reduced word $|X(-j,-1)|$ is unusually long for some $j\in [1,J_m^H]_{\BB Z}$.

\begin{lem} \label{prop-J^H-long-exact}
Suppose $m\in\BB N$ and $(k,l)\in \BB N \times \BB Z$. Then  
\eqb \label{eqn-J^H-long-exact}
 \BB P\left((J_m^H , L_m^H) = (k,l)\right) \preceq  m^{-3} \wedge\left( k^{-1/2} m^{-2 }  \right) \wedge \left(|l|^{-1} m^{-2}\right) \wedge \left( e^{-a_0 m^2/k} m^{-3} \right) ,
\eqe 
with $a_0 > 0$ a universal constant and the implicit constant depending only on $p$. Furthermore, there is a universal constant $a_1 > 0$ such that for $R>0$ we have
\eqb \label{eqn-J^H-long-exact-sup}
 \BB P\left( (J_m^H , L_m^H) = (k,l) ,\, \sup_{j\in [1,J_m^H]_{\BB Z}} |X(-j,-1)| \geq R \right) \preceq \exp\left( - \frac{a_1 m^2}{k}  -\frac{a_1 R}{k^{1/2}}     \right) m^{-3} ,
\eqe 
with the implicit constant depending only on $p$. 
\end{lem}
\begin{proof}
It is clear from Proposition~\ref{prop-J^H-local} and Lemma~\ref{prop-g-form} that for any $m\in\BB N$ and any $(k,l)\in\BB N \times \BB Z$, we have
\eqb \label{eqn-J^H-upper-basic}
 \BB P\left((J_m^H , L_m^H) = (k,l)\right) \preceq  m^{-3}  ,
\eqe 
with the implicit constant depending only on $p$. 

To prove the other estimates in the statement of the lemma, we will split the word $X_{-J_m^H} \dots X_{-1}$ into two independent segments, apply the estimate of Proposition~\ref{prop-J^H-local} in one segment, apply the estimate of Lemma~\ref{prop-J^H-sup} in the other segment, then multiply the probabilities for the two segments and sum over the two possibilities corresponding to which estimate is applied in which segment. 
To this end, define
\begin{align} \label{eqn-J^1-J^2}
&m' := \lfloor m/2 \rfloor ,\quad  J_m^{H,1} := J_{m'}^H  , \quad J_m^{H,2} := J_m^H - J_{m'}^H ,\notag\\
& L_m^{H,1}:= L_{m'}^H , \quad\op{and}\quad L_m^{H,2} := d^*\left(X(- J_m^H , -  J_{m'}^H-1)\right) .
\end{align}
Observe that the words $X_{-J_m^H} \dots X_{-J_{m'}^H-1}$ and $X_{-J_{m'}^H} \dots X_{-1}$, hence also the pairs $(J_m^{H,1} , L_m^{H,1})$ and $(J_m^{H,2} ,L_m^{H,2})$, are independent. 
Also note that $J_m^H = J_m^{H,1} + J_m^{H,2}$ and $L_m^H = L_m^{H,1} + L_m^{H,2}$. 

For $i\in \{1,2\}$, $R>0$, and $k\in\BB N$, define events
\alb
&E_m^i(R) := \left\{J_m^{H,i} \geq \frac{R}{2} \right\} ,\quad F_m^i(R) := \left\{ |L_m^{H,i}| \geq \frac{R}{2}  \right\} ,\quad G_m^i(R) := \left\{J_m^{H,i} \leq R \right\} \\
&H_m^1(R,k) := \left\{\sup_{j\in [1,J_m^{H,1}]_{\BB Z}} |X(-j,-1)| \geq \frac{R}{2} ,\, J_m^{H,1} \leq k \right\} \\
&H_m^2(R,k) := \left\{\sup_{j\in [1 ,J_m^{H,2}]_{\BB Z}} |X(-j -J_m^{H,1}   ,-J_m^{H,1}-1)| \geq \frac{R}{2} ,\, J_m^{H,2} \leq k \right\}.
\ale
Then we have
\alb
&\left\{J_m^H \geq R \right\} \subset E_m^1(R) \cup E_m^2(R) ,\, \quad \left\{|L_m^H| \geq R\right\} \subset F_m^1(R) \cup F_m^2(R)  ,\, \quad \left\{J_m^H \leq R \right\} \subset G_m^1(R) \cap G_m^2(R) \\
&\left\{\sup_{j\in [1,J_m^H]_{\BB Z}} |X(-j,-1)| \geq R ,\, J_m^H \leq k \right\} \subset H_m^1(R,k) \cup H_m^2(R,k) .
\ale
By Lemma~\ref{prop-J^H-sup}, for $i\in\{1,2\}$ we have
\eqbn
\BB P\left(E_m^i(R) \right) \preceq R^{-1/2} m ,\qquad \BB P\left(F_m^i(R) \right) \preceq R^{-1 } m ,\qquad \BB P\left(G_m^i(R) \right) \preceq e^{-a_0 m^2/R}  ,
\eqen
with $a_0 > 0$ a universal positive constant.
By \cite[Lemma 3.13]{shef-burger}, there is a universal constant $a_1' > 0$  such that
\eqb \label{eqn-J^H-sup-prob0}
\BB P\left(H_m^i(R,k)\right) \preceq e^{-a_1'   R/k^{1/2} } .
\eqe 
 
By~\eqref{eqn-J^H-upper-basic} and independence, for each $i\in\{1,2\}$, the conditional probability that $(J_m^H , L_m^H) = (k,l)$ given any realization of $(J_m^{H,i} , L_m^{H,i})$ is at most a constant (depending only on $p$) times $m^{-3}$. Hence
\alb
&\BB P\left((J_m^H , L_m^H) = (k,l)\right) \leq \sum_{i=1}^2 \BB P\left(   (J_m^H , L_m^H) = (k,l) \,|\, E_m^i(k) \right) \BB P\left(E_m^i(k)\right) \preceq k^{-1/2} m^{-2} \\
&\BB P\left((J_m^H , L_m^H) = (k,l)\right)  \leq \sum_{i=1}^2 \BB P\left(   (J_m^H , L_m^H) = (k,l) \,|\, F_m^i(|l|) \right) \BB P\left(F_m^i(|l|)\right) \preceq |l|^{-1} m^{-2} \\
&\BB P\left((J_m^H , L_m^H) = (k,l)\right) \leq   \BB P\left(   (J_m^H , L_m^H) = (k,l) \,|\, G_m^1(k) \right) \BB P\left(G_m^1(k)\right) \preceq  e^{-a_0 m^2/k} m^{-3} .
\ale
This yields~\eqref{eqn-J^H-long-exact}. For~\eqref{eqn-J^H-long-exact-sup}, we first use~\eqref{eqn-J^H-sup-prob0} to get
\alb
&\BB P\left((J_m^H , L_m^H) = (k,l)  ,\, \sup_{j\in [1,J_m^H]_{\BB Z}} |X(-j,-1)| \geq R  \right)\\
 &\qquad \leq \sum_{i=1}^2 \BB P\left(   (J_m^H , L_m^H) = (k,l)   \,|\,      H_m^{3-i} (R,k) \right)  \BB P\left(   H_m^{3-i} (R,k) \right)  \\
&\qquad \preceq \sum_{i=1}^2 \BB P\left(   (J_m^H , L_m^H) = (k,l)    \,|\,      H_m^{3-i} (R,k) \right) e^{-a_1'   R/k^{1/2} }       .
\ale
On the other hand, by the fourth inequality in~\eqref{eqn-J^H-long-exact} applied with either $m'$ or $m-m'$ in place of $m$ and since $(J_m^{H,i} , L_m^{H,i})$ is independent from $H_m^{3-i}(R,k)$, we have
\eqbn
\BB P\left(   (J_m^H , L_m^H) = (k,l)    \,|\,      H_m^{3-i} (R,k) \right) \preceq   e^{-a_0 m^2/4k}    m^{-3} .
\eqen
Combining these estimates yields~\eqref{eqn-J^H-long-exact-sup} with $a_1 = a_1' \wedge (a_0/4)$. 
\end{proof}

From Lemma~\ref{prop-J^H-long-exact}, we obtain a regularity estimate when we condition on a particular realization of $(J_m^H,L_m^H)$.

\begin{lem} \label{prop-J^H-local-reg}
For each $C > 1$ and each $q\in (0,1)$, there exists $A>0$ and $m_* = m_*  \in\BB N$, both depending only on $C$ and $q$, such that for $m\geq m_*$ and $(k,l) \in [C^{-1} m^2 , C m^2]_{\BB Z} \times [-Cm , Cm]_{\BB Z}$, we have 
\eqbn
\BB P\left(\sup_{j\in[1,J_m^H]} |X(-j,-1)|  \leq Am  \,\big|\, (J_m^H , L_m^H) = (k,l) \right) \geq 1-q .
\eqen  
\end{lem}
\begin{proof}
By~\eqref{eqn-J^H-long-exact-sup} of Lemma~\ref{prop-J^H-long-exact}, we can find $a  >0$, depending only on $C$, such that for each $A>0$, each $m\in\BB N$, and each $(k,l) \in [C^{-1} m^2 , C m^2]_{\BB Z} \times [-Cm , Cm]_{\BB Z}$,
\eqbn
\BB P\left( \sup_{j\in[1,J_m^H]} |X(-j,-1)| \geq  A  m  ,\,  (J_m^H , L_m^H) = (k,l)  \right) \preceq e^{-a  A} m^{-3}  
\eqen
with the implicit constant depending only on $C$. 
By Proposition~\ref{prop-J^H-local}, we can find $m_* \in\BB N$, depending only on $C$, such that for $m\geq m_*$ and $(k,l) \in [C^{-1} m^2 , C m^2]_{\BB Z} \times [-Cm , Cm]_{\BB Z}$, we have
\eqbn
\BB P\left((J_m^H , L_m^H) = (k,l) \right) \succeq m^{-3}   ,
\eqen
with the implicit constant depending only on $C$. Combining these observations yields the statement of the lemma.
\end{proof}

To complement Lemma~\ref{prop-J^H-long-exact}, we also have a lower bound for the probability that $(J_m^H ,L_m^H) = (k,l)$ and the word $X(-J_m^H, -1)$ has certain unusual behavior.

\begin{lem} \label{prop-J^H-corner-lower}
Fix $C>1$ and $\ep  > 0$. For sufficiently large $m\in\BB N$ (how large depends only on $C$ and $\ep$) and $(k,l) \in \left[C^{-1} m^2 , Cm^2\right]_{\BB Z}\times \left[-Cm , \frac12 \ep  m\right]_{\BB Z}$,
\eqbn
\BB P\left((J_m^H , L_m^H) = (k,l) ,\, \mcl N_{\tc C}\left(X(-J_m^H ,-1)\right) \leq \ep m\right)  \succeq m^{-3}
\eqen
with the implicit constant depending only on $\ep$ and $C$.
\end{lem}
\begin{proof}
Fix $\delta \in (0,\ep/8) $ to be chosen later, depending only on $\ep$ and $C$, and let $m_\delta:= \lfloor (1-\delta ) m\rfloor$. Let $E_m^{k,l}(\delta)$ be the event that
\alb
 &(J_{m_\delta}^H , L_{m_\delta}^H) \in \left[k-2\delta m , k -\delta m\right]_{\BB Z} \times \left[l-\delta m ,l+\delta m\right]_{\BB Z} , \quad
 \mcl N_{\tb C}\left(X(-J_{m_\delta}^H ,-1)\right) \geq  \frac{1}{8} \ep  m , \\
& \qquad \op{and} \quad \mcl N_{\tc C}\left(X(-J_{m_\delta}^H ,-1)\right) \leq  \ep  m.
\ale
By \cite[Theorem 2.5]{shef-burger} (c.f.\ the proof Lemma~\ref{prop-J^H-law}) we and find $m_* \in\BB N$, depending only on $C,\ep,$ and $\delta$ such that for $m\geq m_*$ and $(k,l) \in \left[C^{-1} m^2 , Cm^2\right]_{\BB Z}\times \left[-Cm , \frac12 \ep  m\right]_{\BB Z}$,
\eqb \label{eqn-J^H-corner-close}
\BB P\left(  E_m^{k,l}(\delta) \right) \succeq 1
\eqe 
with the implicit constant depending only on $C,\ep,$ and $\delta$. Here we use that $l \leq \frac12 \ep m$ so that $\ep m - \frac{1}{8} \ep m \geq l - \delta m + \frac38 \ep m$. 
By Lemma~\ref{prop-J^H-local-reg} (applied with the word $X_{-J_m^H} \dots X_{-J_{m_\delta}^H-1}$ in place of $X_{-J_m^H} \dots X_{-1}$) and independence, if $\delta$ is chosen sufficiently small, depending only on $C$ and $\ep$, then
\eqb \label{eqn-J^H-corner-lower-prob}
\BB P\left((J_m^H , L_m^H) = (k,l),\, \sup_{j\in [J_{m_\delta}^H +1 ,J_m^H]_{\BB Z}} |X(-j,-1)| \leq \frac{1}{8} \ep m \,|\, E_m^{k,l}(\delta) \right) \succeq m^{-3}
\eqe 
with the implicit constant depending only on $C$, $\ep$, and $\delta$. 
If $E_m^{k,l}(\delta)$ occur and $\sup_{j\in [J_{m_\delta}^H +1 ,J_m^H]_{\BB Z}} |X(-j,-1)| \leq \frac{1}{8} \ep m$, then each of the cheeseburgers in $X(-J_m^H,-J_{m_\delta}^H-1)$ has a match in $X(-J_{m_\delta}^H ,-1)$. 
Hence we obtain the statement of the lemma for $m\geq m_*$ by combining~\eqref{eqn-J^H-corner-close} and~\eqref{eqn-J^H-corner-lower-prob} and using that $X_{-J_m^H} \dots X_{-J_{m_\delta}^H-1}$ and $X_{-J_{m_\delta}^H} \dots X_{-1}$ are independent. The statement for the finite number of values of $m\leq m_*$, with implicit constant depending on $m_*$, is clear. 
\end{proof}

Next we have a regularity estimate for $X_{-n} \dots X_{-1}$ given only that $n =J_m^H$ for some (unspecified) $m\in\BB N$. 

\begin{lem} \label{prop-J^H-time-reg}
For $n\in\BB N$, let $E_n$ be the event that $n = J_m^H$ for some $m\in\BB N$. For each $q\in (0,1)$, is a constant $A> 0$ depending only on $q$ such that for each $n\in\BB N$, 
\eqb \label{eqn-J^H-time-big}
\BB P\left(\sup_{j\in [1,n]_{\BB Z}} |X(-j,-1)| \leq A n^{1/2} \,\big|\, E_n\right)  \geq 1-q .
\eqe 
\end{lem}
\begin{proof}
By the same argument used to prove~\eqref{eqn-J^H-long-exact-sup} of Lemma~\ref{prop-J^H-long-exact}, but with only the value of $J_m^H$ (not the value of $L_m^H$) specified, for each $m\in\BB N$ and $A>0$ we have
\eqbn
\BB P\left(  \sup_{j\in [1,n]_{\BB Z}} |X(-j,-1)| > A n^{1/2}  ,\, J_m^H = n \right) \preceq   e^{-a_0 A n^{1/2}/m}  m^{-2}
\eqen
with $a_0>0$ a universal constants and the implicit constant depending only on $p$. Hence for each $C > 0$,
\alb
 \BB P\left(\sup_{j\in[1,n]_{\BB Z}} |X(-j,-1)| > A n^{1/2} ,\,   E_n\right)
&\preceq    \sum_{m= 1 }^{\lfloor C n^{1/2} \rfloor}  e^{-a_0 A n^{1/2}/m} m^{-2}  + \sum_{m= \lfloor C n^{1/2} \rfloor}^n  m^{-2}    \\
&\preceq \int_0^{C n^{1/2}} e^{-a_0 A n^{1/2}/t } t^{-2} \, dt  + C^{-1} n^{-1/2}    \\ 
& = n^{-1/2} \int_0^{C } e^{-a_0 A s } s^{-2} \, dt  + C^{-1} n^{-1/2}    .
\ale
For any given $\alpha>0$, we can choose $C$ sufficiently large and then $A$ sufficiently large relative to $C$ such that this integral is at most $\alpha n^{-1/2}$. By Lemma~\ref{prop-J^H-time},
\eqbn
\BB P\left(E_n\right) \asymp n^{-1/2}  ,
\eqen
with the implicit constant depending only on $p$. We conclude by choosing $\alpha$ sufficiently small relative to $q$ and dividing. 
\end{proof}

Finally, we prove an analogue of Lemma~\ref{prop-J^H-corner-lower} in the case when we only condition on the event that $n =J_m^H$ for some unspecified $m\in\BB N$.

\begin{lem} \label{prop-J^H-time-lower}
For $n\in\BB N$, let $E_n$ be as in Lemma~\ref{prop-J^H-time-reg}. 
For each $\ep > 0$ and each $n\in\BB N$, 
\eqb \label{eqn-J^H-time-lower} 
\BB P\left(  \mcl N_{\tc C}\left(X(-n ,-1)\right) \leq \ep n^{1/2} ,\,  E_n \right)  \succeq n^{-1/2} 
\eqe 
with the implicit constant depending only on $\ep$.
\end{lem}
\begin{proof}
Fix $C>1$. By Lemma~\ref{prop-J^H-corner-lower}, there exists $n_* = n_*(C,\ep)  \in \BB N$ such that for $n\geq n_*$ and $(m,l) \in \left[C^{-1} n^{1/2} , C n^{1/2} \right]_{\BB Z}\times \left[-C n^{1/2} , 0 \right]_{\BB Z}$, 
\eqbn
\BB P\left((J_m^H , L_m^H) = (n,l) ,\, \mcl N_{\tc C}\left(X(-J_m^H ,-1)\right) \leq \ep n^{1/2} \right) \succeq n^{-3/2}  
\eqen
with the implicit constant depending only on $C$ and $\ep$.  
Summing over all such pairs $(k,l)$ yields~\eqref{eqn-J^H-time-lower} for $n\geq n_*$. Since there are only finitely many possible realizations of $X_{-n} \dots X_{-1}$, it is clear that~\eqref{eqn-J^H-time-lower} holds for $n\leq n_*$ (with the implicit constant depending on $n_*$). 
\end{proof}

\section{Local estimates with no orders}
\label{sec-no-order-local}

In this subsection, we will use the results of Sections~\ref{sec-J^H-local} to establish sharp estimates for the probability that the word $X(1,n)$ contains no orders and a specified number of burgers of each type, and for the conditional law of the word $X_1\dots X_n$ given that this is the case. These estimates will eventually lead to a proof of Theorem~\ref{thm-local-conv}. 

The basic idea of the estimates of this subsection and the proof of Theorem~\ref{thm-local-conv} is as follows. 
Suppose given $h,c , n\in \BB N$, and recall that we want to analyze the conditional law of $X_1\dots X_n$ given the event $\mcl E_n^{h,c}$ of Definition~\ref{def-local-event}.
In Section~\ref{sec-no-order-local-setup} just below, we will define a time $ K_{n,m}^H$ with the following properties.
\begin{enumerate}
\item The word $X(K_{n,m}^H+1, n)$ contains no flexible orders. \label{item-K-no-F}
\item Although $K_{n,m}^H$ is not a stopping time for $X$, the conditional law of $X_{K_{n,m}^H+1} \dots X_n$ given $X_1\dots X_{K_{n,m}^H}$ admits a simple description (which is closely related to the times $\{J_r^H\}_{r\in \BB N}$ studied in Section~\ref{sec-J^H-local}). \label{item-K-law}
\item If $m = \lfloor (1-\delta) h\rfloor$ for $\delta >0$ small but independent from $m$, then with high probability $K_{n,m}^H$ is close to $n$.  \label{item-K-close}
\end{enumerate}
For technical reasons we will also need to consider analogous times $K_{n,m}^C$ defined with the roles of hamburgers and cheeseburgers interchanged.

Conditions~\ref{item-K-no-F} and~\ref{item-K-law} combined with the estimates of Section~\ref{sec-J^H-local} will enable us to estimate the conditional probability of $\mcl E_n^{h,c}$ given a realization of $X_1\dots X_{K_{n,m}^H}$, which will lead to estimates for the probability of $\mcl E_n^{h,c}$ (see Section~\ref{sec-no-order-local-upper}). 
In Section~\ref{sec-no-order-cont-reg} (and Section~\ref{sec-endpoint-reg}), we will estimate the lengths of the reduced words $ X(K_{n,m}^H , i) $ for $i\in [ K_{n,m}^H+1,n]_{\BB Z}$, which, in light of condition~\ref{item-K-close}, will show that the restriction of the path $Z^n$ of Theorem~\ref{thm-local-conv} to $[0,1]$ is in some sense well-approximated by the restriction of $Z^n$ to $[0,n^{-1} K_{n,m}^H]$ for $m$ sufficiently close to $h$. 
By applying Bayes' rule to reverse the conditioning, we will obtain estimates for the conditional law of $X_1\dots X_{K_{n,m}^H}$ given $\mcl E_n^{h,c}$ (see Section~\ref{sec-no-order-cont}) which will lead to a proof of Theorem~\ref{thm-local-conv} in Section~\ref{sec-local-conv-proof}.

We emphasize that condition~\ref{item-K-no-F} is essential for our argument. Indeed, local estimates of the sort we prove in this paper concern only the \emph{number} of symbols of particular type in a certain reduced word. But, if we have two words $x_1$ and $x_2$ such that the reduced word $\mcl R(x_2)$ include $\tb F$'s, then the number of symbols of each type in the reduced word $\mcl R(x_1x_2)$ depends not just on the number of symbols of each type in $\mcl R(x_1)$ and $\mcl R(x_2)$, but also on the \emph{ordering} of these symbols. 
This is why we need to choose a random time $K_{n,m}^H$ above instead of just considering the word $X_1\dots X_{\lfloor (1-\delta) n\rfloor}$, say. 
 
\subsection{Setup}
\label{sec-no-order-local-setup}

In this subsection, we describe the notation we will use throughout this section and make some elementary observations about the objects involved.  
Recall the definition of the event $\mcl E_n^{h,c}$ from Definition~\ref{def-local-event}, the time $I$ from~\eqref{eqn-I-def}, the exponent $\mu$ from~\eqref{eqn-cone-exponent}, and the slowly varying function $\psi_0$ from Lemma~\ref{prop-I-reg-var}. We also introduce the following additional notation to simplify some of our formulas.

\begin{defn} \label{def-frk}
For a word $x$, we write
\eqbn
\frk h(x) := \mcl N_{\tc H}(\mcl R(x)) \quad \op{and}\quad \frk c(x) := \mcl N_{\tc C}(\mcl R(x)) ,
\eqen 
with $\mcl N_{\tc H}$, $\mcl N_{\tc C}$, and $\mcl R$ as in Section~\ref{sec-burger-prelim}. 
\end{defn}
 
For $n , m\in\BB N$, let $K_{n,m}^H  $ (resp. $K_{n,m}^C $) be the largest $i \in [1,n-1]_{\BB Z}$ for which $\mcl N_{\tc H}\left(X(1,i)\right) = m$ (resp. $\mcl N_{\tc C}\left(X(1,i)\right) = m$) and $X_{i+1}$ is a hamburger (resp. cheeseburger) which is not consumed before time $n$; or $K_{n,m}^H = 0$ (resp. $K_{n,m}^C = 0$) if no such $i$ exists. On the event $\{I > n\}$, $K_{n,m}^H$ can equivalently be described as the first time $i \in [1,n]_{\BB Z}$ for which $d(j) \geq m+1$ for each $j \geq i+1$ (or $K_{n,m}^H = 0$ of no such $i$ exists). The time $K_{n,m}^C$ admits a similar description. We write
\eqbn
Q_{n,m}^H := \mcl N_{\tc C}\left(X(1,K_{n,m}^H)\right) \qquad Q_{n,m}^C := \mcl N_{\tc H}\left(X(1,K_{n,m}^C)\right) .
\eqen
See Figure~\ref{fig-last-hit-time} for an illustration.

\begin{figure}[ht!]
 \begin{center}
\includegraphics{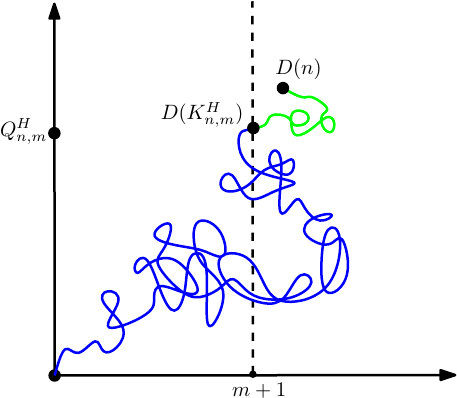} 
\caption{An illustration of the time $K_{n,m}^H$, which is $-1$ plus the last time at which the discrete path $D$ of~\eqref{eqn-discrete-path} crosses the vertical line at distance $m+1$ from the origin and subsequently stays to the right of this line. Here $d(n)$, the horizontal coordinate of $D(n)$, is $\geq m$. If $D(n)$ were to the left of the dotted line, then we would have $K_{n,m}^H = 0$. The quantity $Q_{n,m}^H$ is the vertical coordinate of $D(K_{n,m}^H)$. Note also that $I > n$ in this illustration, i.e.\ the path $D$ stays in the first quadrant.}\label{fig-last-hit-time}
\end{center}
\end{figure}
  
For $r\in\BB N$, let $J_{n,r}^H$ be the smallest $j \in\BB N$ for which $X(n-j,n)$ contains $r $ hamburgers and set $L_{n,r}^H := d^*\left(X(n-J_{n,r}^H , n)\right)$. That is, $(J_{n,r}^H , L_{n,r}^H)$ are defined in the same manner as the pairs~\eqref{eqn-J^H-def} but with the word read backward from $n$ rather than from $-1$. 

The main idea of the proofs in this section is to condition on a realization of the word $X$ up to time $K_{n,m}^H$ for some $m\in\BB N$; then read the word backward from time $n$ and apply the results of Section~\ref{sec-J^H-local} to estimate the pairs $(J_{n,r}^H , L_{n,r}^H)$. The next three lemmas are the basic tools needed to accomplish this. 

\begin{lem} \label{prop-J^H-D-equiv}
Let $D$ be defined as in~\eqref{eqn-discrete-path}.
For $(h,c) \in \BB N^2$ and $m\in [1,h-1]_{\BB Z}$, the event $\mcl E_n^{h,c}$ is the same as the event that  
\alb
&0 < K_{n,m}^H < I ,\quad J_{n,h-m}^H = n-1- K_{n,m}^H ,\quad L_{n,h-m}^H = c  - Q_{n,m}^H   ,\\
&\qquad \op{and}\quad \mcl N_{\tb C}\left(X(n-J_{n,h-m}^H,n)\right) \leq Q_{n,m}^H .
\ale
\end{lem}
\begin{proof}
It is clear that $K_{n,m}^H >0$ on $\mcl E_n^{h,c}$. On the event $\{K_{n,m}^H >0\}$, there is some $j\in\BB N$ for which $n - J_{n,j}^H  = K_{n,m}^H+1$. Since $X(n-J_{n,j}^H , n)$ contains no hamburger orders or flexible orders, it follows that for this choice of $j$ we have
\alb
&\mcl N_{\tc H}\left(X(1,n)\right) = m + j ,\quad
d^*\left(X(1,n)\right) = Q_{n,m}^H +  L_{n,j}^H  , \\
&\mcl N_{\tb C}\left(X(1,n)\right) =  0\vee \left( \mcl N_{\tb C}\left( X(n-J_{n,j}^H , n)     \right)  -  Q_{n,m}^H \right)   .
\ale
The statement of the lemma follows.
\end{proof}

Next we have a description of the law of the objects discussed above. 

\begin{lem}\label{prop-J^H-D-law}
The marginal law of the pairs $(J_{n,r}^H, L_{n,r}^H)_{r\in \BB N}$ is the same as the law of the pairs $(J_r^H , L_r^H)_{r\in \BB N}$ of Section~\ref{sec-J^H-local}. Furthermore, for each $m\in\BB N$ the conditional law of $X_{K_{n,m}^H+1} \dots X_n$ given any realization $x$ of $X_1\dots X_{K_{n,m}^H}$ for which $0 < K_{n,m}^H < I $ is the same as its conditional law given the event 
\eqb \label{eqn-J^H-hit-event}
R_n(x) := \{ \text{$n-1- |x|  =   J_{n,r}^H$ for some $r\in\BB N$}\}.
\eqe 
\end{lem}
\begin{proof}
The first statement is immediate from translation invariance. To verify this second statement, we observe that for each $k\leq n$, the event $\{K_{n,m}^H =k \} \cap \{K_{n,m}^H < I\}$ is the same as the event that $X(1,k)$ contains no orders and exactly $m$ hamburgers; and that $X_{k+1}$ is a hamburger which does not have a match in $[k+2,n]_{\BB Z}$, i.e.\ $k+1 = n-J_{n,r}^H$ for some $r\in\BB N$. Since $X_1\dots X_k$ is independent from $X_{k+1}\dots X_n$, it follows that the conditional law of $X_{k+1}\dots X_n$ given $\{X_1\dots X_{K_{n,m}^H} = x\}$ is the same as its conditional law given that $   n-|x|-1 = J_{n,r}^H$ for some $r\in\BB N$.
\end{proof}
 
Lemmas~\ref{prop-J^H-D-equiv} and~\ref{prop-J^H-D-law} law together yield the following formulae, which we will use frequently in the remainder of this subsection.

\begin{lem} \label{prop-J^H-D-formulas}
Let $(h,c) \in \BB N^2$, $n\in\BB N$, and $m\in\BB N$ with $m < h$. Let $x$ be any realization of $X_1 \dots X_{K_{n,m}^H}$ for which $0 < K_{n,m}^H < I$, so that $\frk c(x)$ (Definition~\ref{def-frk}) is the corresponding realization of $Q_{n,m}^H$. Then we have
\begin{align} \label{eqn-x-h-c-prob}
&\BB P\left(\mcl E_n^{h,c}  \,|\,  X_1\dots X_{K_{n,m }^H}  =x \right)\notag \\
&\qquad = \frac{\BB P\left((J_{n,h-m}^H , L_{n,h-m}^H) = (n- |x| - 1, c - \frk c(x) ) ,\, \mcl N_{\tb C}\left(X(n-J_{n,h-m}^H,n)\right) \leq \frk c(x) \right)}{    \BB P\left(R_n(x)   \right)}   ,
\end{align}  
with $R_n(x)$ as in~\eqref{eqn-J^H-hit-event}; and
\begin{align} \label{eqn-x-h-c-prob-I}
&\BB P\left(\mcl E_n^{h,c}  \,|\,  X_1\dots X_{K_{n,m }^H}  =x   ,\, I > n \right)\notag \\
&\qquad = \frac{\BB P\left((J_{n,h-m}^H , L_{n,h-m}^H) = (n- |x| - 1, c - \frk c(x) ) ,\, \mcl N_{\tb C}\left(X(n-J_{n,h-m}^H,n)\right) \leq \frk c(x) \right)}{    \BB P\left(R_n(x)  ,\, \mcl N_{\tb C}\left(X(|x|+1 , n)\right) \leq \frk c(x)  \right)}  .
\end{align}  
\end{lem}
\begin{proof}
The first formula is immediate from Lemmas~\ref{prop-J^H-D-equiv} and~\ref{prop-J^H-D-law}. 
The second formula follows from these same two lemmas after noting that, since $X(K_{n,m}^H +1 , n)$ always contains no hamburger orders or flexible orders, on the event $\{ X_1 \dots X_{K_{n,m}^H} = x\}$ we have $I> n$ (i.e.\ $X(1,n)$ contains no orders) if and only if $\mcl N_{\tb C}\left(X(|x|+1 , n)\right) \leq \frk c(x)$. 
\end{proof}

\subsection{Bounds for $\BB P(\mcl E_n^{h,c})$}
\label{sec-no-order-local-upper}

In this subsection we will prove estimates for the probability of the event $\mcl E_n^{h,c}$ of Definition~\ref{def-local-event}. We start with the lower bound, which is easier.

\begin{prop}[Lower bound] \label{prop-no-order-lower}
For each $C>1$, $n \geq C^2$, and $(h,c) \in [C^{-1} n^{1/2} , C n^{1/2}]_{\BB Z}^2$, we have
\eqb \label{eqn-no-order-lower-cond}
\BB P\left( \mcl E_n^{h,c}  \,|\, I > n \right) \succeq  n^{-1}
\eqe 
with the implicit constant depending only on $C$. 
In particular, with $\psi_0$ as in Lemma~\ref{prop-I-reg-var},
\eqb \label{eqn-no-order-lower}
\BB P\left( \mcl E_n^{h,c} \right) \succeq \psi_0(n) n^{-1-\mu} 
\eqe 
with implicit constant depending only on $C$. 
\end{prop}
\begin{proof} 
Fix $C>1$. For $h\in\BB N$ and $\delta>0$, let $m_h^\delta := \lfloor (1-\delta) h\rfloor$ and $r_h^\delta := \lfloor \delta h \rfloor$. 
By Lemma~\ref{prop-J^H-local-reg}, we can find $A >0$, depending only on $p$, such that for each $h\in\BB N$ and each $(k,l) \in [\frac12 \delta^2 h^2 ,  \delta^2 h^2]_{\BB Z} \times [- \delta  h  ,  \delta  h ]_{\BB Z}^2$,
\eqb  \label{eqn-J^H-D-lower}
\BB P\left( \sup_{j\in [1,J_{r_h^\delta}^H]_{\BB Z}} |X(n-j,n)| \leq A r_h^\delta ,\,   (J_{r_h^\delta}^H , L_{r_h^\delta}^H) = (k,l) \right) \succeq \delta^{-3} h^{-3} 
\eqe 
with the implicit constant depending only on $p$. 

Henceforth fix $\delta \in (0,1/2)$ such that $A\delta \leq (2C)^{-1}$, and note that $\delta$ depends only on $C$. By \cite[Theorem A.1]{gms-burger-cone}, for each $n \geq C^2$ and each $(h,c) \in [C^{-1} n^{1/2} , C n^{1/2}]_{\BB Z}^2$, we have
\eqb \label{eqn-KQ-pos}
\BB P\left( (K_{n,m_h^\delta}^H , Q_{n,m_h^\delta}^H) \in \left[n -   \delta^2 h^2 ,n -\frac12  \delta^2 h^2\right]_{\BB Z} \times \left[c- \delta  h  , c+ \delta  h \right]_{\BB Z}^2 \,|\, I > n\right) \succeq 1 .
\eqe 
By~\eqref{eqn-x-h-c-prob-I}, for any realization $x$ of $X_1\dots X_{K_{n,m_h^\delta}^H}$ for which $(K_{n,m_h^\delta}^H , Q_{n,m_h^\delta}^H) \in \left[n -   \delta^2 h^2 ,n -\frac12  \delta^2 h^2\right]_{\BB Z} \times \left[c- \delta  h  , c+ \delta  h \right]_{\BB Z}^2$ and $0 < K_{n,m_h^\delta}^H < I$, we have
\begin{align}  
&\BB P\left(\mcl E_n^{h,c}  \,|\,  X_1\dots X_{K_{n,m_h^\delta}^H}  =x   ,\, I > n \right)\notag \\
&\qquad \geq \frac{\BB P\left((J_{n,r_h^\delta }^H , L_{n,r_h^\delta}^H) = (n- |x| - 1, c - \frk c(x) ) ,\, \mcl N_{\tb C}\left(X(n-J_{n,r_h^\delta}^H,n)\right) \leq \frk c(x) \right)}{    \BB P\left(R_n(x)   \right)} ,
\end{align} 
where $\frk c(x) := \mcl N_{\tc C}\left(\mcl R(x)\right)$. By~\eqref{eqn-J^H-D-lower}, or choice of $\delta$, and Lemma~\ref{prop-J^H-time}, this quantity is
$\succeq (n-|x|)^{1/2} \delta^{-3} h^{-3} \succeq n^{-1}$
with the implicit constant depending only on $C$. By combining this with~\eqref{eqn-KQ-pos}, we obtain~\eqref{eqn-no-order-lower-cond}.
The estimate~\eqref{eqn-no-order-lower} follows from~\eqref{eqn-no-order-lower-cond} together with Lemma~\ref{prop-I-reg-var}. 
\end{proof}

Next we prove an upper bound for $\BB P(\mcl E_n^{h,c})$.

\begin{prop}[Upper bound] \label{prop-no-order-upper}
Let $\psi_0$ be the slowly varying function from Lemma~\ref{prop-I-reg-var}. For each $n\in\BB N$ and $(h,c) \in \BB N^2$,
\eqbn
\BB P\left(  \mcl E_n^{h,c}    \right) \preceq  \psi_0((h\wedge c)^2)   (h\wedge c)^{-2 - 2\mu} 
\eqen
with the implicit constant depending only on $p$. 
\end{prop}

To prove our upper bound for the probability of $\mcl E_n^{h,c}$, we first need the following lemma which is the source of a factor of $\psi_0((h\wedge c)^2)   (h\wedge c)^{  - 2\mu}$ in Proposition~\ref{prop-no-order-upper} (the other factor of $(h\wedge c)^{-2}$ will come from Lemma~\ref{prop-J^H-D-formulas}).

\begin{lem} \label{prop-I-burger-sup}
Let $\psi_0$ be the slowly varying function from Lemma~\ref{prop-I-reg-var} and let $I$ be as in~\eqref{eqn-I-def}. For each $m \in\BB N$, we have
\eqbn
\BB P\left(\sup_{i\in[1,I]_{\BB Z}} \left( \mcl N_{\tc H}(X(1,i)) \wedge \mcl N_{\tc C}(X(1,i))    \right) \geq m \right) \preceq \psi_0 (m^2) m^{-2\mu}  
\eqen
with the implicit constant depending only on $p$. 
\end{lem}
\begin{proof}
For $m\in\BB N$, let $T_m$ be the smallest $i\in \BB N$ for which $\mcl N_{\tc H}(X(1,i)) \wedge \mcl N_{\tc C}(X(1,i) ) \geq m/2$. Also let $\BB k_m$ be the largest $k\in\BB N$ for which $2^{ k -1} \leq m^2$. Observe that if
\[
\sup_{i\in[1,I]_{\BB Z}} \left( \mcl N_{\tc H}(X(1,i)) \wedge \mcl N_{\tc C}(X(1,i))    \right) \geq m
\]
then $T_m  < I$ and $\sup_{i \in [T_m+1, I]_{\BB Z}} |X(T_m+1,i)| \geq m/2$. Therefore,
\begin{align}
&\BB P\left(\sup_{i\in[1,I]_{\BB Z}} \left( \mcl N_{\tc H}(X(1,i)) \wedge \mcl N_{\tc C}(X(1,i))    \right) \geq m \right)\notag \\
&\qquad \leq \sum_{k=1}^{\BB k_m} \BB P\left( \sup_{i\in[ T_m  +1 , T_m + 2^k]_{\BB Z}} |X(T_m+1 ,i)| \geq m/2 ,\,  I - T_m \in  [2^{k-1},2^k]_{\BB Z} \right)   + \BB P\left(I > m^2 \right) \notag\\  
&\qquad \leq \sum_{k=1}^{\BB k_m} \BB P\left( \sup_{i\in[ T_m + 1, T_m + 2^k]_{\BB Z}} |X(T_m+1 ,i)| \geq m/2 \,|\,  I > T_m + 2^{k-1} \right) \BB P\left( I >      2^{k-1}\right) + \BB P\left(I > m^2 \right)  . \label{eqn-I-K-sum}
\end{align}
Let $x$ be any realization of $X_1\dots X_{T_m}$ for which $T_m < I$. Then for $k \in [1,\BB k_m]_{\BB Z}$, 
\alb
&\BB P\left( \sup_{i\in[T_m+1,T_m+2^k]_{\BB Z}} |X(T_m+1,i)| \geq  m /2 \,\big|\,  I > T_m + 2^{k-1} ,\, X_1\dots X_{T_m} = x \right)\\
&\qquad \leq \frac{\BB P\left( \sup_{i\in[T_m+1,2^k]_{\BB Z}} |X(T_m+1,i)| \geq  m /2 \,|\,  X_1\dots X_{T_m} = x  \right)}{\BB P\left( I > T_m + 2^{k-1} \,|\, X_1\dots X_{T_m} = x   \right)   } .
\ale
Since $\mcl R(x)$ contains no orders and at least $m/2$ burgers of each type, and since $2^{k-1} \leq m^2$, it follows from \cite[Theorem 2.5]{shef-burger} that $\BB P\left( I > T_m + 2^{k-1} \,|\, X_1\dots X_{T_m} = x   \right) $ is bounded below by a universal constant. 
By \cite[Lemma 3.13]{shef-burger}, 
\eqbn
\BB P\left( \sup_{i\in[T_m+1,2^k]_{\BB Z}} |X(T_m+1,i)| \geq  m /2 \,|\,  X_1\dots X_{T_m} = x  \right) \leq a_0 e^{-a_1 2^{-k/2} m}
\eqen
for universal constants $a_0 , a_1 > 0$. By averaging over all realizations of $X_1\dots X_{T_m}$ for which $T_m < I$, we obtain
\eqb \label{eqn-I-cond-sup}
\BB P\left( \sup_{i\in[T_m+1,2^k]_{\BB Z}} |X(T_m+1,i)| \geq  m /2 \,\big|\,  I > T_m + 2^{k-1} \right) \leq a_0 e^{-a_1 2^{-k/2} m} .
\eqe 
By Lemma~\ref{prop-I-reg-var}, 
\eqbn
\BB P\left(I > 2^{k-1} \right)  \preceq \psi_0(2^k ) 2^{-k \mu}  .
\eqen
Therefore,~\eqref{eqn-I-K-sum} is at most a constant (depending only on $p$) times
\alb
&\sum_{k=1}^{\BB k_m} \psi_0(2^k )2^{-k \mu}  e^{-a_1 2^{-k/2} m}   + \psi_0(m^2) m^{-2\mu} \\
&\qquad \preceq \psi_0(m^2) m^{-2\mu}  \sum_{k=1}^{\BB k_m} \frac{\psi_0(2^k)}{\psi_0(m^2)}  2^{(\BB k_m - k) \mu}  e^{-a_1 2^{(\BB k_m - k)/2} }   + \psi_0(m^2) m^{-2\mu}\\
%&\qquad \preceq \psi_0(m^2) m^{-2\mu}  \sum_{j=1}^{\BB k_m} \frac{\psi_0(2^{ \BB k_m - j} )}{\psi_0(2^{ \BB k_m  })}  2^{j \mu}  e^{-a_1 2^{j/2} }        \\
%&\qquad \preceq \psi_0(m^2) m^{-2\mu}  \sum_{j=1}^{\BB k_m} 2^{2(\BB k_m - j) o_j(1)}  2^{j \mu}  e^{-a_1 2^{j/2} }        \\
&\qquad \preceq \psi_0(m^2) m^{-2\mu} . \qedhere
\ale
\end{proof}

\begin{proof}[Proof of Proposition~\ref{prop-no-order-upper}]
Given $(h,c) \in \BB N^2$, let $m_h = \lfloor h/2\rfloor$, $r_h = h-m_h$, $m_c = \lfloor c/2\rfloor$, and $r_c  = c-m_c$. Also define $K_{n,m_h}^H$ and $K_{n,m_c}^C$ as in Section~\ref{sec-no-order-local-setup}. 
On the event $ \mcl E_n^{h,c} $, both $K_{n,m_h}^H$ and $K_{n,m_c}^C$ are $<n$ and non-zero. Furthermore, if $K_{n,m_c}^C < K_{n,m_h}^H$, then $\mcl N_{\tc C}\left(X(1,K_{n,m_h}^H)\right) \geq m_c$, and the analogous statement holds when $K_{n,m_h}^H \leq K_{n,m_c}^C$. Hence
\alb
\BB P\left(\mcl E_n^{h,c} \right) 
&= \BB P\left(\mcl E_n^{h,c}  ,\,  K_{n,m_c}^C \leq K_{n,m_h}^H < n <I   \right) + \BB P\left(\mcl E_n^{h,c} ,\,  K_{n,m_h}^H < K_{n,m_c}^C < n   < I \right) \\ 
& \leq \BB P\left(\mcl E_n^{h,c} ,\,   K_{n,m_h}^H < n  < I,\,  \mcl N_{\tc C}\left(X(1,K_{n,m_h}^H)\right) \geq  m_c \right) \\
&\qquad  + \BB P\left(\mcl E_n^{h,c} , \,    K_{n,m_c}^C < n  < I,\,  \mcl N_{\tc H}\left(X(1,K_{n,m_c}^C)\right) \geq m_h    \right)\\
& \leq \BB P\left(\mcl E_n^{h,c}  \,\big|\,   0< K_{n,m_h}^H   < I  ,\,  \mcl N_{\tc C}\left(X(1,K_{n,m_h}^H)\right) \geq m_c  \right) 
\BB P\left( 0< K_{n,m_h}^H   < I ,\,  \mcl N_{\tc C}\left(X(1,K_{n,m_h}^H)\right) \geq m_c \right) \\
&\qquad  + \BB P\left(\mcl E_n^{h,c}  \,\big|\,  0< K_{n,m_c}^C < I   , \,  \mcl N_{\tc H}\left(X(1,K_{n,m_c}^C)\right) \geq m_h   \right) 
\BB P\left( 0< K_{n,m_c}^C   < I  , \,  \mcl N_{\tc H}\left(X(1,K_{n,m_c}^C)\right) \geq m_h \right)   .
\ale
By symmetry, it suffices to show that
\eqb \label{eqn-no-order-upper-reduce}
\BB P\left(\mcl E_n^{h,c}  \,|\,   0< K_{n,m_h}^H   < I  ,\,  \mcl N_{\tc C}\left(X(1,K_{n,m_h}^H)\right) \geq m_c \right) \preceq (h\wedge c)^{-2 }   .
\eqe 
and
\eqb \label{eqn-no-order-upper-reduce'}
\BB P\left( 0< K_{n,m_h}^H   < I  ,\,  \mcl N_{\tc C}\left(X(1,K_{n,m_h}^H)\right) \geq m_c        \right)  \preceq \psi_0((h\wedge c)^2) (h\wedge c)^{-2\mu }   .
\eqe

To this end, fix a realization $x$ of $X_1\dots X_{K_{n,m_h}^H}$ for which $0< K_{n,m_h}^H < I $ and $\frk c(x) \geq m_c$. By~\eqref{eqn-x-h-c-prob}, we obtain
\begin{align} \label{eqn-burgers-even-ratio}
\BB P\left(\mcl E_n^{h,c}  \,|\,  X_1\dots X_{K_{n,m_h}^H}  =x \right)
 \leq \frac{\BB P\left((J_{n,r_h}^H , L_{n,r_h}^H) = (n-1-|x|, c - \frk c(x) ) \right)}{    \BB P\left(R_n(x)   \right)} ,
\end{align}  
with $J_{n,r_h}^H$ and $L_{n,r_h}^H$ are as in Section~\ref{sec-no-order-local-setup} and $R_n(x)$ as in~\eqref{eqn-J^H-hit-event}.
By Lemma~\ref{prop-J^H-time},  
\eqb \label{eqn-J^H-hit-prob}
\BB P\left(  R_n(x) \right)\asymp (n-|x|)^{-1/2}   
\eqe 
with the implicit constant depending only on $p$. 
By Lemma~\ref{prop-J^H-long-exact}, 
\eqbn
\BB P\left((J_{n,r_h}^H , L_{n,r_h}^H) = (n-1-|x|, c - \frk c(x) ) \right) \preceq  (n-|x|)^{-1/2} h^{-2} 
\eqen
with the implicit constant depending only on $p$. Hence~\eqref{eqn-burgers-even-ratio} yields
\eqbn
\BB P\left(\mcl E_n^{h,c} \,|\,  X_1\dots X_{K_{n,m_h}^H}  =x\right) \preceq h^{-2} .
\eqen 
By averaging over all choices of the realization $x$ for which $0<K_{n,h}^H < I $ and $\frk c(x) \geq m_c$, we obtain~\eqref{eqn-no-order-upper-reduce}. 

For~\eqref{eqn-no-order-upper-reduce'}, we observe that if $0 < K_{n,h}^H < I$ and $\mcl N_{\tc C}\left(X(1,K_{n,h}^H)\right) \geq m_c $, then 
\eqbn
\sup_{i\in [1,I]_{\BB Z}}  \left( \mcl N_{\tc H}(X(1,i)) \wedge \mcl N_{\tc C}(X(1,i))    \right) \geq m_c \wedge m_h .
\eqen
Hence~\eqref{eqn-no-order-upper-reduce'} follows from Lemma~\ref{prop-I-burger-sup}. 
\end{proof}
 
\subsection{Regularity estimates}
\label{sec-no-order-cont-reg}

In this subsection we will consider some regularity results for the conditional law of $X_1\dots X_n$ given $\mcl E_n^{h,c}$. 
Our first proposition tells us, roughly speaking, that the pair $(K_{n,m}^H , Q_{n,m}^H)$ is unlikely to be too far from $(n,c)$ if $m$ is close to $h$ and we condition on $\mcl E_n^{h,c}$. 

\begin{prop} \label{prop-no-order-endpoint}
For $n\in\BB N$, $(h,c) \in \BB N^2$, $\delta \in (0,1/2)$, and $A>1$, let
\eqb \label{eqn-h-c-nghd}
\mcl U_n^\delta(A , h,c) :=  \left[n - A^2  \delta^2 h^2 , n - A^{-2} \delta^2 h^2 \right]_{\BB Z} \times \left[ c - A \delta h  , c + A \delta h\right]_{\BB Z} .
\eqe 
Also write $m_h^\delta := \lfloor (1-\delta)h\rfloor$.
For each $C >1$ and $q\in (0,1)$, there exists $A > 1$ and $\delta_* >0$ (depending only on $C$ and $q$) such that the following is true.  
 For each $\delta \in (0,\delta_*]$, there is an $n_* = n_*(\delta,C,q) \in\BB N$ such that for $n\geq n_*$ and $(h,c) \in \left[C^{-1} n^{1/2} , C n^{1/2} \right]_{\BB Z}^2$, we have
\eqbn
\BB P\left((K_{n,m_h^\delta}^H ,Q_{n,m_h^\delta}^H) \in \mcl U_n^\delta(A,h,c)   \,|\, \mcl E_n^{h,c} \right) \geq 1-q .
\eqen
\end{prop}

Note that in the statement of Proposition~\ref{prop-no-order-endpoint} (and in several other lemmas in this paper) one should think of the parameter $q$ as small close to 0. 
The proof of Proposition~\ref{prop-no-order-endpoint} is technical, so is postponed to Section~\ref{sec-endpoint-reg} so that the reader can see the main ideas of the proof of Theorem~\ref{thm-local-conv} sooner.

Our next lemma tells us that if we condition on $\mcl E_n^{h,c}$ and a sufficiently nice realization of $X_1\dots X_{K_{n,m}^H}$ for $m$ slightly smaller than $h$, then it is unlikely that $|X(K_{n,m}^H+1 , i)|$ is very large for any $i\in [K_{n,m}^H +1, n]_{\BB Z}$. This result together with Proposition~\ref{prop-no-order-endpoint} will eventually allow us to deduce a scaling limit result for the conditional law of $Z^n$ given $\mcl E_n^{h,c}$ from statements about the conditional law of $X_1\dots X_{K_{n,m}^H}$ given $\mcl E_n^{h,c}$ for $m$ slightly smaller than $h$, since it implies that $Z^n$ does not move very much in the last $n^{-1}(n - K_{n,m}^H)$ units of time.

\begin{lem} \label{prop-no-order-reg}
Fix $C>1$, $A>0$, and $q\in (0,1)$. Define $\mcl U_n^\delta(A,h,c)$ as in~\eqref{eqn-h-c-nghd} and let $m_h^\delta := \lfloor (1-\delta) h\rfloor$ as in Proposition~\ref{prop-no-order-endpoint}. There is a $\delta_*>0$ and a $B>0$, depending only on $C$, $A$, and $q$, such that the following is true. For each $n\in\BB N$, each $\delta \in (0,\delta_*]$, each $(h,c) \in \left[C^{-1} n^{1/2} , Cn^{1/2}\right]_{\BB Z}^2$, and each realization $x$ of $X_1\dots X_{K_{n,m_h^\delta}^H}$ for which $(K_{n,m_h^\delta}^H , Q_{n,m_h^\delta}^H) \in \mcl U_n^\delta(A,h,c)$, we have
\eqb \label{eqn-no-order-reg}
\BB P\left( \sup_{i\in [K_{n,m_h^\delta}^H +1,n]_{\BB Z}} |X(K_{n,m_h^\delta}^H +1,i)| \leq B \delta n^{1/2} \,|\, \mcl E_n^{h,c} ,\, X_1\dots X_{K_{n,m_h^\delta}^H}   =x  \right) \geq 1-q .
\eqe 
\end{lem}
\begin{proof}
Let $\delta>0$, $n\in\BB N$, $(h,c) \in \left[C^{-1} n^{1/2} , Cn^{1/2}\right]_{\BB Z}^2$, and let $x$ be a realization of $X_{K_{n,m_h^\delta}+1}\dots X_n$ for which $(K_{n,m_h^\delta}^H , Q_{n,m_h^\delta}^H) \in \mcl U_n^\delta(A,h,c)$. Also let $r_h^\delta:= h-m_h^\delta$.  
By Lemmas~\ref{prop-J^H-D-equiv} and~\ref{prop-J^H-D-law}, the conditional law of $X_{K_{n,m_h^\delta}^H+1}\dots X_n$ given $\mcl E_n^{h,c} \cap\{ X_1\dots X_{K_{n,m_h^\delta}^H}   =x\} $ is the same as the conditional law of $X_{-J_{r_h^\delta}^H}\dots X_{-1}$ given that $J_{r_h^\delta}^H = n-|x|-1$, $L_{r_h^\delta}^H = c-\frk c(x)$, and $\mcl N_{\tb C}\left(X( -J_{r_h^\delta}^H , -1)\right) \leq \frk c(x)$, where here $(J_{r_h^\delta} , L_{r_h^\delta}^H)$ are as in~\eqref{eqn-J^H-def}. The statement of the lemma now follows from Lemma~\ref{prop-J^H-local-reg}. 
\end{proof}

 \subsection{Continuity estimates}
 \label{sec-no-order-cont}
 
In this subsection we will prove some lemmas to the effect that the conditional probability of $\mcl E_n^{h,c}$ given a realization of $X_1\dots X_{K_{n,m}^H}$ for $m$ slightly smaller than $h$ does not depend too strongly on the realization. These lemmas together with Bayes' rule will allow us to compare the conditional law of $Z^n$ given $\mcl E_n^{h,c}$ to its conditional law given $\{I > n\} = \{\text{$X(1,n)$ contains no orders}\}$ and the approximate (rather than exact) number of burgers of each type in the reduced word $X(1,n)$. We know the scaling limit of the law of $Z^n$ under the latter conditioning due to~\cite[Theorem A.1]{gms-burger-cone}.  
Throughout this subsection we define the sets $\mcl U_n^\delta(A,h,c)$ as in~\eqref{eqn-h-c-nghd} and let $m_h^\delta := \lfloor (1-\delta) h \rfloor$ be as in Proposition~\ref{prop-no-order-endpoint}. 

We first prove a lemma to the effect that the conditional probability of $\mcl E_n^{h,c}$ given $X_1\dots X_{K_{n,m_h^\delta}^H}$ does not change by too much if we slightly vary $h,c$, the realization of $K_{n,m_h^\delta}$, and/or the realizations of the numbers of burgers of each type in $X(1,K_{n,m_h^\delta})$.

\begin{lem}  \label{prop-no-order-cont}
For each $q \in (0,1)$, $C>1$, and $A>0$, there exists $\delta_* > 0$ such that for each $\delta \in (0,\delta_*]$, the following holds. There exists $\zeta_* > 0$ such that for each $\zeta \in (0,\zeta_*]$, there exists $n_* = n_*(\delta,\zeta,q,C) \in \BB N$ such that the following is true. 
Suppose $n\geq  n_*$ and $(h,c) , (h',c') \in [C^{-1} n^{1/2} , C n^{1/2}]_{\BB Z}^2$ with $|(h,c) - (h',c')| \leq \zeta n^{1/2}$. Suppose also that $x$ and $x'$ are realizations of $X_{1} \dots X_{K_{n,m_h^\delta}^H} $ and $X_{1} \dots X_{K_{n,m_{h'}^\delta}^H}  $, respectively, such that 
\begin{align} \label{eqn-close-realizations}
&(|x| , \frk c(x))  \in \mcl U_n^\delta(2A,h,c) ,\quad (|x'| , \frk c(x'))  \in \mcl U_n^\delta(2A,h',c') , \notag \\
&||x|- |x'|| \leq \zeta n ,\quad\op{and}\quad |\frk c(x) - \frk c(x')|   \leq \zeta n^{1/2} .
\end{align}
 Then we have
\eqb \label{eqn-no-order-cont} 
1-q\leq    \frac{\BB P\left( \mcl E_n^{h,c} \,|\,  X_{1} \dots X_{K_{n,m_h^\delta}^H}  = x,\,    I > n\right)  }{\BB P\left(\mcl E_n^{h',c'} \,|\,  X_{1} \dots X_{K_{n,m_h^\delta}^H} = x',\,  I > n\right) } \leq \frac{1}{1-q} .
\eqe 
\end{lem}
\begin{proof}%[Proof of Lemma~\ref{prop-no-order-cont}]
Fix $\alpha>0$ to be chosen later, depending only on $q$. 
Let $r_h^\delta := h-m_h^\delta$. Suppose given a realization $x$ of $X_1\dots X_{K_{n,m_h^\delta}^H}$ such that $(K_{n,m_h^\delta}^H , Q_{n,m_h^\delta}^H) \in \mcl U_n^\delta(2A,h,c)$ and $0< K_{n,m_h^\delta}^H < I$.  
By~\eqref{eqn-x-h-c-prob-I}, 
\alb
&\BB P\left(\mcl E_n^{h,c}  \,|\,  X_1\dots X_{K_{n,m_h}^H}  =x   ,\, I > n \right)\notag \\
&\qquad = \frac{\BB P\left((J_{n,r_h^\delta}^H , L_{n,r_h^\delta}^H) = (n- |x| - 1, c - \frk c(x) ) ,\, \mcl N_{\tb C}\left(X(n-J_{n,r_h^\delta}^H)\right) \leq \frk c(x) \right)}{    \BB P\left(R_n(x)  ,\, \mcl N_{\tb C}\left(X(n-J_{n,r_h^\delta}^H , n )\right) \leq \frk c(x)  \right)}  .
\ale
By Lemmas~\ref{prop-J^H-local-reg} and~\ref{prop-J^H-time-reg}, we can find $\delta_* > 0$, depending only on $\alpha$, $A$, and $C$, such that for each $\delta \in (0,\delta_* ]$ there exists $n_*^0  = n_*^0(\delta,C,A,\alpha) \in \BB N$ such that for $n\geq n_*^0$ and $(h,c) \in [C^{-1} n^{1/2} , C n^{1/2}]_{\BB Z}^2$, we have
\alb
&\BB P\left( \mcl N_{\tb C}\left(X(n-J_{n,r_h^\delta}^H, n)\right) \leq \frk c(x) \,|\, (J_{n,r_h^\delta}^H , L_{n,r_h^\delta}^H) = (n- |x| - 1, c - \frk c(x) )    \right) \geq 1-\alpha \quad \op{and}  \\
&\BB P\left(  \mcl N_{\tb C}\left(X(n-J_{n,r_h^\delta}^H, n)\right) \leq \frk c(x)  \,|\, R_n(x) \right) \geq 1-\alpha .
\ale
By Lemma~\ref{prop-J^H-time}, there is a constant $a_0  >0$, depending only on $p$, such that for each $\delta \in (0,\delta_* ]$ there exists $n_*^1 = n_*^1(\delta , C, A, \alpha) \in\BB N$ such that for $n\geq n_*^1$ and each realization $x$ as above,
\eqbn
b_1 (1-\alpha)(n-|x|)^{-1/2} \leq \BB P\left(R_n(x) \right) \leq \frac{b_1}{1-\alpha} (n-|x|)^{-1/2}  
\eqen
where here $b_1> 0$ is a constant depending only on $p$. 
Hence if $\delta \in (0,\delta_*]$, $n\geq n_*^1$, and  $(h,c) \in [C^{-1} n^{1/2} , C n^{1/2}]_{\BB Z}^2$ then
\eqb \label{eqn-D-J^H-compare-cont}
b_1^{-1} (1-\alpha)^3   \leq 
\frac{ \BB P\left(\mcl E_n^{h,c}  \,|\,  X_1\dots X_{K_{n,m_h}^H}  =x   ,\, I > n \right) }{(n-|x|)^{1/2} \BB P\left((J_{n,r_h^\delta}^H , L_{n,r_h^\delta}^H) = (n- |x| - 1, c - \frk c(x) ) \right) }   
\leq  \frac{b_1^{-1} }{(1-\alpha)^3}   .
\eqe 
By Proposition~\ref{prop-J^H-local} (in particular, by continuity of the function $g$ of that proposition), we can find $\zeta  > 0$, depending only on $\delta$, $C$, and $q$, and $n_*^2  = n_*^2(\delta , C, \alpha) \geq n_*^1$ such that for any $n\geq n_*^2$, any two pairs $(h,c) , (h',c') \in [C^{-1} n^{1/2} , C n^{1/2}]_{\BB Z}^2$ with $|(h,c) - (h',c')| \leq \zeta n^{1/2}$, and any realizations $x$ of $X_1\dots X_{K_{n,m_h^\delta}^H}$ and $x'$ of $X_1\dots X_{K_{n,m_{h'}^\delta}^H}$ for which~\eqref{eqn-close-realizations} holds, we have
\eqbn
1-\alpha \leq \frac{ \BB P\left((J_{n,r_h^\delta}^H , L_{n,r_h^\delta}^H) = (n- |x| - 1, c - \frk c(x) ) \right)  }{\BB P\left((J_{n,r_{h'}^\delta}^H , L_{n,r_{h'}^\delta}^H) = (n- |x'| - 1, c' - \frk c(x') ) \right)} \leq \frac{1}{1-\alpha} .
\eqen
By combining this with~\eqref{eqn-D-J^H-compare-cont} and choosing $\alpha$ sufficiently small, depending only on $q$, we obtain the statement of the lemma. 
\end{proof}

For $n,m,k,l \in \BB N$ and $\zeta > 0$, define 
\eqb \label{eqn-P-event}
\mcl P_{n,m}^{k,l}(\zeta) := \left\{ 0< K_{n,m}^H < I  ,\,   |K_{n,m}^H - k| \leq \zeta n ,\, |Q_{n,m}^H - l| \leq \zeta n^{1/2} \right\} .
\eqe 
If $\zeta$ is close to 0 (but independent of the other parameters), $m$ is slightly smaller than $h$, $k$ is slightly smaller than $n$, and $l \approx c$, then $\mcl P_{n,m}^{k,l}$ is an approximate version of $\mcl E_n^{h,c}$ which depends only on $X_1\dots X_{K_{n,m}^H}$. The following lemma allows us to compare the conditional law of $X_1\dots X_{K_{n,m}^H}$ given $\mcl E_n^{h,c}$ and its conditional law given $\mcl P_{n,m}^{k,l}(\zeta)$.

\begin{lem} \label{prop-no-order-abs-cont}
For each $q \in (0,1)$, $C>1$, and $A>0$, there exists $\delta_* > 0$ such that for each $\delta \in (0,\delta_*]$, there exists $\zeta_* > 0$ such that for each $\zeta \in (0,\zeta_*]$, there exists $n_* = n_*(\delta,\zeta,q,C) \in \BB N$ such that the following is true. 
Suppose $n\geq  n_*$, $(h,c)   \in [C^{-1} n^{1/2} , C n^{1/2}]_{\BB Z}^2$, and $(k,l) \in \mcl U_n^\delta(A,h,c)$. The conditional law of $X_1\dots X_{K_{n,m_h^\delta}^H}$ given $\mcl P_{n,m_h^\delta }^{k,l}(\zeta) \cap \mcl E_n^{h,c}$ is mutually absolutely continuous with respect to its conditional law given only $\mcl P_{n,m_h^\delta}^{k,l}(\zeta) \cap \{I>n\}$, with Radon-Nikodym derivative bounded above by $(1-q)^{-1}$ and below by $1-q$.
\end{lem}
\begin{proof}
Let $\delta_*$ be chosen so that the conclusion of Lemma~\ref{prop-no-order-cont} holds. Given $\delta \in (0,\delta_*]$, let $\zeta_* > 0$ be chosen as in Lemma~\ref{prop-no-order-cont}. By possibly shrinking $\zeta_*$ we can arrange that for $\zeta \leq \zeta_*$, we have $(k', l') \in \mcl U_n^\delta(A,h,c)$ whenever $(k,l) \in \mcl U_n^\delta(A,h,c)$, $|k-k'| \leq \zeta n$, and $|l-l' | \leq \zeta n^{1/2}$. 
For $\zeta \in (0,\zeta_*]$, let $n_* = n_*(\delta,\zeta,q,C)$ be as in Lemma~\ref{prop-no-order-cont}. 

Let $n\geq n*$, $(h,c) \in \left[C^{-1} n^{1/2} , C n^{1/2}\right]_{\BB Z}$, and let $x$ be a realization of $X_{1} \dots X_{K_{n,m_h^\delta}} $ for which $\mcl P_{n,m_h^\delta}^{k,l}(\zeta)$ occurs. 
By applying~\eqref{eqn-no-order-cont} with $(h,c) = (h',c')$ and averaging over all choices of realization $x'$ for which $\mcl P_{n,m_h^\delta}^{k,l}(\zeta)$ occurs, we obtain
\eqb \label{eqn-x-B-compare}
1-q\leq    \frac{\BB P\left( \mcl E_n^{h,c} \,|\,  X_{1} \dots X_{K_{n,m_h^\delta}} = x,\,    I > n\right)  }{\BB P\left(\mcl E_n^{h ,c } \,|\,    \mcl P_{n,m_h^\delta}^{k,l}(\zeta) ,\, I > n \right) } \leq \frac{1}{1-q} .
\eqe 
By Bayes' rule (applied to the conditional probability measure $\BB P\left( \cdot \,|\, \mcl P_{n,m_h^\delta}^{k,l}(\zeta),\, I > n \right)$ and the events $\{X_{1} \dots X_{K_{n,m_h^\delta}} = x \} \subset \mcl P_{n,m_h^\delta}^{k,l}(\zeta) $ and $ \mcl E_n^{h,c} \subset \{I > n\}$), we have
\alb
&\BB P\left( X_{1} \dots X_{K_{n,m_h^\delta}} = x \,|\, \mcl P_{n,m_h^\delta}^{k,l}(\zeta) \cap \mcl E_n^{h,c}    \right)  \\
&\qquad = \frac{\BB P\left(    \mcl E_n^{h,c} \,|\, X_{1} \dots X_{K_{n,m_h^\delta}} = x ,\, I > n    \right) \BB P\left( X_{1} \dots X_{K_{n,m_h^\delta}} = x \,|\,   \mcl P_{n,m_h^\delta}^{k,l}(\zeta) ,\, I >n \right) }{ \BB P\left(\mcl E_n^{h ,c } \,|\,    \mcl P_{n,m_h^\delta}^{k,l}(\zeta) ,\, I> n \right) }. 
\ale
By combining this with~\eqref{eqn-x-B-compare} we conclude.
\end{proof}

\subsection{Proof of Theorem~\ref{thm-local-conv}}
\label{sec-local-conv-proof}

In this subsection we will prove Theorem~\ref{thm-local-conv}.
In what follows, let $\wh Z = (\wh U , \wh V)$ be a correlated two-dimensional Brownian motion as in~\eqref{eqn-bm-cov} conditioned to stay in the first quadrant for one unit of time. For $u > 0$, let $\wh\tau_u$ be the last time $t\in [0,1]$ such that $\wh U(t) \geq u$ for each $s\in [t,1]$; or $\wh\tau_u = 0$ if no such $t$ exists. The pair $(\wh \tau_u , \wh V(\wh\tau_u))$ is the continuum analogue of the pairs $(K_{n,m}^H , Q_{n,m}^H)$ of Section~\ref{sec-no-order-local-setup}, which leads to the following lemma. 

\begin{lem} \label{prop-last-conv}
Define the times $K_{n,m}^H$ and the quantities $Q_{n,m}^H$ as in Section~\ref{sec-no-order-local-setup}. For each $\ep > 0$ and $C>1$, there exists $n_* = n_*(\ep , C) \in\BB N$ such that for each $n\geq n_*$ and each $m \in [C^{-1} n^{1/2} , Cn^{1/2}]_{\BB Z}$, the Prokhorov distance between the conditional law of $(Z^n , n^{-1} K_{n,m}^H , n^{-1/2} Q_{n,m}^H )$ given $\{I >n\}$ and the law of $(\wh Z,\wh\tau_{m/n^{1/2}} , \wh V(\wh\tau_{m/n^{1/2}}))$ is at most $\ep$. 
\end{lem}

Lemma~\ref{prop-last-conv} will be used in the proof of Theorem~\ref{thm-local-conv} as well as in the proof of Proposition~\ref{prop-no-order-endpoint} in Section~\ref{sec-endpoint-reg} below. The proof of the lemma does not use any of the other results in this paper, however, so there is no circularity. 
We emphasize that the bound on the Prokhorov distance in Lemma~\ref{prop-last-conv} is uniform over all $m \in [C^{-1} n^{1/2} , Cn^{1/2}]_{\BB Z}$. This will be important in the proof of Theorem~\ref{thm-local-conv} (since the theorem asserts a uniform bound on the Prokhorov distance between the conditional law of $Z^n$ given $\mcl E_n^{h,c}$ and the law of the limiting Brownian motion) as well as in the argument of Section~\ref{sec-endpoint-reg}. 

\begin{proof}[Proof of Lemma~\ref{prop-last-conv}]
By \cite[Theorem A.1]{gms-burger-cone} and the Skorokhod theorem, we can find a coupling of a sequence of words $(\wh X^n)$ distributed according to the conditional law of the word $X$ given $\{I>n\}$; and the path $\wh Z$ such that in this coupling $\wh Z^n \rta \wh Z$ uniformly a.s.\ on $[0,1]$, where $\wh Z^n = (\wh U^n , \wh V^n)$ is the path constructed from the word $\wh X^n$ as in~\eqref{eqn-Z^n-def}. For $u > 0$, let $\wh\tau_u^n$ be the last time $t\in [0,1]$ for which $U^n(s) > u$ for each $s  \in (t,1]$; and otherwise let $\wh\tau_u^n =0$. If we construct the pair $(n^{-1} K_{n,m}^H , n^{-1/2} Q_{n,m}^H)$ as in Section~\ref{sec-no-order-local-setup} from the word $\wh X^n$, then for $m \in [C^{-1} n^{1/2} , Cn^{1/2}]_{\BB Z}$ it holds with probability at least $1-o_n(1)$ that 
\eqbn
(n^{-1} K_{n,m}^H , n^{-1/2} Q_{n,m}^H)  =  (\wh\tau_{m/n^{1/2}}^n + o_n(1) , \wh V^n(\wh\tau_{m/n^{1/2}}^n) + o_n(1)) ,
\eqen
where the $o_n(1)$ is uniform in $m$ (it comes from rounding error). 
Hence it suffices to show that in our coupling, $(\wh\tau_u^n ,  \wh V^n(\wh\tau_u^n)) \rta (\wh\tau_u , \wh V(\wh\tau_u))$ in probability, uniformly for $u\in [C^{-1} , C]$. 

Fix $\ep > 0$ and $C>1$. We observe the following.
\begin{enumerate}
\item By equicontinuity, we can find $\delta_1   \in (0,\ep]$ such that with probability at least $1-\ep$, we have $|\wh Z(t) - \wh Z(s)| \leq \ep$ and $|\wh Z^n(t) - \wh Z^n(s)| \leq \ep$ whenever $n\in\BB N$ and $s,t\in [0,1]$ with $|s-t|\leq \delta_1$. \label{item-tau_u-cont}
\item The conditional law of $ \wh U(\cdot +\wh\tau_u) - u$ given $\wh\tau_u$ on the event $\{\wh\tau_u > 0\}$ is that of a one-dimensional Brownian motion conditioned to stay positive until time 1. Hence we can find $\delta_2  \in (0,\delta_1/2]$, independent from $u$, such that with probability at least $(1-\ep) \BB P(\wh\tau_u >  0)$, we have $\wh\tau_u  > 0$ and $\wh U(t) \geq u + \delta_2$ for each $t \in [\wh\tau_u + \delta_1 , 1]$.  \label{item-tau_u-stay}
\item We can find $\delta_3 >0$, independent from $u \in [C^{-1} , C]$, such that for $u\in [C^{-1} , C]$ it holds with probability at least $(1-\ep)\BB P(\wh\tau_u > 0)$ that $\wh\tau_u  > 0$ and $\inf_{t\in [\wh\tau_u - \delta_1 , \wh\tau_u]} \wh U(t) \leq u - \delta_3$. \label{item-tau_u-inf}
%The laws of $\wh\wh\tau_u$ and $\wh V(\wh\tau_u)$ depend continuously on $u$, so we can find $R>1$ such that for each $u\in [C^{-1},C]$ it holds with probability at least $(1-\ep/2)\BB P(\wh\tau_u > 0)$ that $R^{-1} \leq \wh\tau_u \leq 1-R^{-1}$ and $R^{-1} \leq \wh V(\wh\tau_u) \leq R$. If we condition on $\wh\tau_u$ and $\wh V(\wh\tau_u)$ we get a BM conditioned to stay in the first quadrant until time $\wh\tau_u$ and end up at $(u,\wh V(\wh\tau_u))$. This law also depends continuously on its parameters. Furthermore, for each fixed choice of time and endpoint there a.s.\ exists some $\delta_3$ such that the desired condition holds. By continuity this $\delta_3$ can be chosen uniformly for $(u,\wh\tau_u , \wh V(\wh\tau_u)) \in  [C^{-1} , C]\times [R^{-1} , 1-R^{-1}]\times [R^{-1} , R]$
\item By uniform convergence, we can find $N_1 \in \BB N$ such that with probability at least $1-\ep$ it holds for $n\geq N_1$ that $|\wh Z(t) - \wh Z^n(t)| \leq (\delta_2 \wedge \delta_3)/2$ for each $t\in [0,1]$. \label{item-tau_u-conv}
\end{enumerate}
Let $E_u$ be the event that $\wh\tau_u > 0$ and the above four conditions hold, so that for $u\in [C^{-1} , C]$ we have $\BB P(E_u) \geq (1-2\ep) \BB P(\wh\tau_u > 0) - 2\ep$. 

By conditions~\ref{item-tau_u-stay},~\ref{item-tau_u-inf}, and~\ref{item-tau_u-conv}, if $E_u$ occurs adn $\wh\tau_u > 0$ then for $n\geq N_1$ there exists $t_* \in [\wh\tau_u - \delta_1 , \wh\tau_u]$ with $\wh Z^n(t_*)  < u - \delta_3/2$ and $\wh Z^n(t) \geq u +\delta_2/2$ for each $t \in [\wh\tau_u + \delta_1 , 1]$. 
Therefore $\wh\tau_u^n \in [\wh\tau_u -\delta_1 ,\wh\tau_u + \delta_1]$. By conditions~\ref{item-tau_u-cont} and~\ref{item-tau_u-conv}, on $E_u \cap \{\wh\tau_u > 0\}$ we have $|\wh V^n(\wh\tau_u^n) - \wh V(\wh\tau_u)| \leq 2\ep$. Hence for $n\geq N_1$, 
\eqb \label{eqn-close-nonzero}
\BB P\left(|(\wh\tau^n_u , \wh V^n(\wh\tau_u^n)) - (\wh\tau_u , \wh V^n(\wh\tau_u))| \leq 3\ep ,\, \wh\tau_u > 0    \right) \geq (1-2\ep) \BB P(\wh\tau_u > 0) - 2\ep .
\eqe 

We have
\eqbn
\{\wh\tau_u = 0 \} = \left\{\sup_{t\in [0,1]} \wh U(t)   < u \right\} .
\eqen
Since $\BB P\left( \sup_{t\in [0,1]} \wh U(t)  \in [u- \zeta , u + \zeta]\right) \rta 0$ as $\zeta \rta 0$, uniformly over $u \in [C^{-1} , C] $, we can find a deterministic $N_2  \in\BB N$ such that for $n\geq N_2$ and $u\in [C^{-1} , C]$, 
\eqbn
\BB P\left(\left\{ \wh\tau_u^n = 0  ,\, \wh\tau_u >0\right\} \cup \left\{\wh\tau_u^n  > 0,\, \wh\tau_u =0\right\}\right) \leq \ep .
\eqen
It now follows from~\eqref{eqn-close-nonzero} that for $n\geq N_1 \vee N_2$ and $u\in [C^{-1} , C]$,
\eqbn
\BB P\left(|(\wh\tau^n_u , \wh V^n(\wh\tau_u^n)) - (\wh\tau_u , \wh V^n(\wh\tau_u))| \leq 3\ep   \right) \geq 1 - 5\ep .
\eqen
Since $\ep$ is arbitrary, we conclude.
\end{proof}
 
In the remainder of this subsection, for $(u,v) \in (0,\infty)^2$ and $t\in [0,1]$, let $\BB P_t^{u,v}$ be the regular conditional law of $\wh Z$ given $\{\wh Z(t) = (u,v)\}$, as described in Section~\ref{sec-bm-cond}. Let $\wh Z^{u,v} = (\wh U^{u,v} , \wh V^{u,v})$ be a path distributed according to the law $\BB P_1^{u,v}$. We record the following basic continuity property of the laws $\BB P_1^{u,v}$.

\begin{lem} \label{prop-bm-uniform}
Fix $C>1$. For each $\ep  >0$, there exists $\delta>0$ such that the following is true. Let $(u,v) \in \left[C^{-1} , C\right]^2 $ and $(t,v') \in [1-\delta^2, 1] \times [v-\delta , v+\delta]$. Also set $u_\delta:= (1-\delta) u$. The Prokhorov distance between any two of the following three laws is at most $\ep$. 
\begin{enumerate}
\item The regular conditional law of $\wh Z  |_{[0,t]}$ given $\{(\wh\tau_{u_\delta}  , \wh V (\wh\tau_{u_\delta}) ) = (t,v') \}$. 
\item The regular conditional law of $\wh Z^{u,v} |_{[0,t ]}$ given $\{(\wh\tau_{u_\delta}  , \wh V^{u,v}(\wh\tau_{u_\delta}) ) = (t,v') \}$. 
\item The law of $\wh Z^{u,v}|_{[0,t]}$.
\end{enumerate} 
\end{lem}
\begin{proof}
Since $\wh\tau_u$ is the \emph{last} $t \in [0,1]$ for which $\wh U$ crosses $u$, we infer from the Markov property of $\wh Z$ that for each $s\in (0,t)$, the regular conditional law of $\wh Z  |_{[0,s ]}$ given $\wh Z|_{[0,s]}$ and $\{(\wh\tau_{u_\delta}  , \wh V  (\wh\tau_{u_\delta}) ) = (t,v') \}$ is that of a correlated Brownian bridge from $\wh Z(s)$ to $(u_\delta , v')$ in time $t-s$ conditioned to stay in the first quadrant. By Lemma~\ref{prop-bm-tip}, the regular conditional law of $\wh Z|_{[0,t]}$ given $\{(\wh\tau_{u_\delta}  , \wh V  (\wh\tau_{u_\delta}) ) = (t,v') \}$ is $\BB P_t^{u_\delta,v'}$. By a similar argument, the regular conditional law of $\wh Z^{u,v} |_{[0,\wh\tau_{u_\delta}]}$ given $\{(\wh\tau_{u_\delta}  , \wh V^{u,v}(\wh\tau_{u_\delta}) ) = (t,v') \}$ is also given by $\BB P_t^{u_\delta,v'}$. From the description of the law $\BB P_t^{u,v}$ in Section~\ref{sec-bm-cond}, we see that this law depends continuously on its parameters. It follows that when $\delta$ is small and $(t,v') \in [1-\delta^2, 1] \times [v-\delta , v+\delta]$, the Prokhorov distance between the law of $\wh Z^{u,v}|_{[0,t]}$ and the law $\BB P_t^{u_\delta,v'}$ is at most $\ep$.  
\end{proof}

\begin{proof}[Proof of Theorem~\ref{thm-local-conv}] 
Fix $\ep >0$ and $C>1$. By Proposition~\ref{prop-no-order-endpoint} and Lemma~\ref{prop-no-order-reg}, we can find $A > 1$, $B>0$, and $\delta_*^0 >0$ (depending only on $\ep$ and $C$) such that the following is true.  
For each $\delta \in (0,\delta_*^0]$, there exists $n_*^0 = n_*^0(\delta,\ep, C) \in\BB N$ such that for $n\geq n_*^0$ and $(h,c) \in \left[C^{-1} n^{1/2} , C n^{1/2} \right]_{\BB Z}^2$, we have
\eqb \label{eqn-A-B-choice}
\BB P\left((K_{n,m_h^\delta}^H ,Q_{n,m_h^\delta}^H) \in \mcl U_n^\delta(A,h,c)  ,\, \sup_{i\in [K_{n,m_h^\delta}^H +1, n]_{\BB Z}} |X(K_{n,m_h^\delta}^H +1,i)| \leq B \delta n^{1/2}  \,|\, \mcl E_n^{h,c} \right) \geq 1-\ep ,
\eqe 
with $\mcl U_n^\delta(A,h,c)$ as in Proposition~\ref{prop-no-order-endpoint}.
Choose $\delta_*^1 \in (0,\delta_*^0]$ in such a way that $B \delta_*^1 \leq \ep$. 

By Lemma~\ref{prop-no-order-abs-cont}, there exists $\delta_*^2  = \delta_*^2(\ep,C) \in (0, \delta_*^0]$ such that for each $\delta \in (0,\delta_*^2]$, there exists $\zeta_* > 0$ such that for each $\zeta \in (0,\zeta_*]$, there exists $n_*^1 = n_*^1(\delta,\zeta,\ep ,C) \in \BB N$ such that the following is true. 
Suppose $n\geq  n_*^1$, $(h,c)   \in [C^{-1} n^{1/2} , C n^{1/2}]_{\BB Z}^2$, and $(k,l) \in \mcl U_n^\delta(A,h,c)$. Then the Prokhorov distance between the conditional law of $X_1\dots X_{K_{n,m_h^\delta}^H}$ given $\mcl P_{n,m_h^\delta}^{k,l}(\zeta) \cap \mcl E_n^{h,c}$ and its conditional law given only $\mcl P_{n,m_h^\delta}^{k,l}(\zeta)\cap \{I>n\}$ is at most $\ep$, where here $\mcl P_{n,m_h^\delta}^{k,l}(\zeta)$ is as in~\eqref{eqn-P-event} with $m= m_h^\delta = \lfloor (1-\delta) h \rfloor$. 

Define the times $\tau_u$, the laws $\BB P_t^{u,v}$, and the paths $\wh Z $ and $\wh Z^{u,v}$ as in the discussion just above Lemma~\ref{prop-bm-uniform}. For $\zeta>0$, $t\in (0,1)$, and $u,v > 0$, let
\eqb \label{eqn-P-event-bm}
\mcl P_u^{t,v}(\zeta) := \left\{     | \tau_u - t  | \leq \zeta  ,\, |  \wh V(\tau_u) - v    | \leq \zeta  \right\} .
\eqe 
For $\delta>0$, let $u_\delta := (1-\delta) u$. 
By averaging the estimate of Lemma~\ref{prop-bm-uniform} over pairs $(\wt t , \wt v')$ in a small neighborhood of a given pair $(t , v')$, we can find $\delta  = \delta (\ep,C) \in (0,\delta_*^2]$ such that the following is true. For each $(u,v) \in \left[C^{-1} , C\right]^2$, each $(t,v') \in [1- A^2\delta^2, 1] \times [v-2A\delta , v+2A\delta]$, and each $\zeta \in (0, \delta]$, the Prokhorov distance between the conditional law of $\wh Z  |_{[0,1 - A^2\delta^2 ]}$ given $\mcl P_{u_\delta}^{t,v'}(\zeta)$ and the law of $\wh Z^{u,v}|_{[0,1 - A^2\delta^2]}$ is at most $\ep$. By possibly decreasing $\delta$ and using continuity of the law $\BB P_1^{u,v}$ in $u$ and $v$, we can arrange that for each $(u,v) \in \left[C^{-1} , C\right]^2$, we have
\eqb \label{eqn-sup-choice-bm}
\BB P \left(\sup_{s\in  [ 1-A^2\delta^2  ,1 ]} |\wh Z^{u,v}(s) - (u,v) | \leq \ep\right) \geq 1-\ep .
\eqe 

By Lemma~\ref{prop-last-conv}, for each $\zeta \in (0,\zeta_*\wedge \delta]$ we can find $n_*^2 = n_*^2(\zeta,\delta,\ep,C) \geq n_*^1$ such that for each $n\geq n_*^2$, each $(h,c) \in \left[C^{-1}n^{1/2} , C n^{1/2} \right]_{\BB Z}$, and each $(k,l) \in \mcl U_n^\delta(A,h,c)$ the Prokhorov distance between the conditional law of $Z^n|_{[0,1]}$ given $\mcl P_{n,m_h^\delta}^{k,l}(\zeta) \cap \{I>n\}$ and the conditional law of $\wh Z|_{[0,1]}$ given $\mcl P_{ (1-\delta) h/n^{1/2}}^{k/n, l/n^{1/2}}(\zeta)$ is at most $\ep$. By our choices of parameters above, we also have that the Prokhorov distance between the conditional law of $Z^n|_{[0,1-A^2\delta^2]}$ given $\mcl P_{n,m_h^\delta}^{k,l}(\zeta) \cap \mcl E_n^{h,c}$ and the law of $\wh Z^{h/n^{1/2} , c/n^{1/2}}|_{[0,1-A^2\delta^2]}$ is at most $3\ep$. 

Since this holds for each choice of $(k,l) \in \mcl U_n^\delta(A,h,c)$ and by our choice of $A$, it follows that the Prokohorov distance between the conditional law of $Z^n|_{[0,1-A^2\delta^2]}$ given $ \mcl E_n^{h,c}$ and the conditional law of $\wh Z^{h/n^{1/2} , c/n^{1/2}}|_{[0,t-A^2\delta^2]}$ is at most $4\ep$. By~\eqref{eqn-A-B-choice}, our choice of $\delta_*^1$, and~\eqref{eqn-sup-choice-bm}, it follows that the Prokhorov distance between the conditional law of $Z^n|_{[0,1 ]}$ given $ \mcl E_n^{h,c}$ and the conditional law of $\wh Z^{h/n^{1/2} , c/n^{1/2}}|_{[0,1]}$ is at most $8\ep$. Since $\ep$ is arbitrary we conclude. 
\end{proof}

The reader should note that the proof of Theorem~\ref{thm-local-conv} fails in the case when either $u=0$ or $v=0$ (i.e., when we conditions on the reduced word $X(1,n)$ to contain no orders and  $h = o_n(n^{1/2})$ hamburgers and/or $h =  o_n(n^{1/2})$ cheeseburgers). Indeed, if the reduced word $X(1,n)$ contains $o_n(n^{1/2})$ hamburgers or $o_n(n^{1/2})$ cheeseburgers, then either $K_{n,m}^H$ or $K_{n,m}^C$ (or both) will be zero when $n$ is close to $m$. 
Furthermore, the conditional law of $X_{K_{n,m}^H+1}\dots X_n$ does not admit as simple a description as in our setting since this word must contain an unusually small number of hamburgers and/or cheeseburgers, so we cannot apply the estimates of Section~\ref{sec-J^H-local}.
We prove an analogue of Theorem~\ref{thm-local-conv} for $ u = v = 0$ in the subsequent work~\cite{gms-burger-finite} using a more sophisticated argument which builds on the estimates of the present paper.

\subsection{Proof of Proposition~\ref{prop-no-order-endpoint}}
\label{sec-endpoint-reg}

In this subsection, we will prove Proposition~\ref{prop-no-order-endpoint}, which is the key input in the proofs of Lemmas~\ref{prop-no-order-reg} and~\ref{prop-no-order-cont}. Throughout this subsection we continue to use the notation of Section~\ref{sec-no-order-local-setup}. 

The proof of Proposition~\ref{prop-no-order-endpoint} is by brute force. We start in Section~\ref{sec-no-order-local-bm} by proving estimates for a correlated two-dimensional Brownian motion $\wh Z = (\wh U , \wh V)$ as in~\eqref{eqn-bm-cov} started from 0 and conditioned to stay in the first quadrant.
In particular, we will consider the times $\wh \tau_u$ for $u > 0$ which are the continuum analogues of the times $K_{n,m}^H$ of Section~\ref{sec-no-order-local-setup} and prove bounds for the probability that $(\wh \tau_u , \wh V(\wh\tau_u))$ belongs to a given subset of $[0,1] \times [0,\infty)$ (see in particular Lemma~\ref{prop-bm-last-cond}). 
In Section~\ref{sec-no-order-local-endpoint}, we will use the scaling limit result Lemma~\ref{prop-last-conv} to transfer these estimates for $\wh Z$ to estimates for the word $X$ conditioned on the event $\{I>n\}$. 
We will then use these estimates together with a decomposition of $[0,n]_{\BB Z} \times [0,\infty)_{\BB Z}$ into exponential scales to prove Proposition~\ref{prop-no-order-endpoint}.

\subsubsection{Brownian motion estimates}
\label{sec-no-order-local-bm}

In this section we will prove some estimates for Brownian motion which are related to the scaling limits of the quantities considered in Section~\ref{sec-no-order-local-setup}. We start with some basic calculations for an unconditioned Brownian motion.

\begin{lem} \label{prop-bm-last-hit}
Let $B$ be a standard linear Brownian motion. For $u \in \BB R$, let $\tau_u $ be the first time $t\in [0,1]$ such that $B_s > u$ for each $s \in (t,1]$; or $\tau_u = 0$ if no such time exists. The density of $\tau_u$ restricted to the event $\{\tau_u > 0\}$ is given by
\eqbn
\BB P\left(t < \tau_u < t+dt \right) =   \frac{ e^{-u^2/2t}}{ \pi  \sqrt{t(1-t)}} \, dt  ,\quad \forall t \in (0,1) .
\eqen 
\end{lem}
\begin{proof}
Let $\sigma_u$ be the first time $B$ hits $u$ and for $s\geq 0$ set $\wt B_s := B_{s+\sigma_u} - B_{\sigma_u}$. Then $\wt B$ is a standard linear Brownian motion independent from $\sigma_u$. On the event $\{\sigma_u \leq 1\}$, let $\wt\tau_u  := \sup\{s\in [0,1-\sigma_u] \,:\,\wt  B_s= 0\}$. Then $\wt\tau_u$ is equal to $\tau_u - \sigma_u$ on the event $\{\wt B_s > 0 ,\: \forall s \in (\wt\tau_u , 1-\sigma_u]\} \cap \{  \sigma_u \leq 1\}$, which by symmetry has conditional probability $1/2$ given $Z|_{[0,\wt\tau_u  +\sigma_u]}$ on the event $\{ \sigma_u \leq 1\}$. 
The conditional law of $\wt\tau_u$ given $\sigma_u$ on the event $\{\sigma_u < 1\}$ is given by the arcsine distribution,
\eqbn
\BB P\left(\wt \tau_u \leq t \,|\, \sigma_u\right) %= \BB P\left( \frac{\wt\tau_u}{1-\sigma_u} \leq \frac{t}{1-\sigma_u}\right) 
= \frac{2}{\pi} \arcsin\left( \sqrt{\frac{t}{1-\sigma_u}} \right)  .
\eqen
The law of $\sigma_u$ is given by
\eqbn
\frac{|u|}{\sqrt{2\pi}} s^{-3/2} e^{-u^2/2s} \, ds .
\eqen
Hence for $t\in (0,1)$,
\alb
\BB P\left(0 < \tau_u \leq t\right) &=\frac12 \BB E\left( \BB P\left(0 < \wt\tau_u \leq t - \sigma_u\,|\, \sigma_u \right) \BB 1_{\sigma_u \leq t} \right) \\
&=  \frac{|u|}{ \sqrt{2} \pi^{3/2} }   \int_0^t \arcsin\left( \sqrt{\frac{t-s}{1-s}} \right) s^{-3/2} e^{-u^2/2s} \, ds .
\ale
By differentiating (using the Liebniz rule), we get that the density of $\tau_u$ is given by
\eqbn
 \frac{|u|}{2\sqrt{2} \pi^{3/2} } \int_0^t \frac{s^{-3/2} e^{-u^2/2s}}{ \sqrt{(1-t)(t-s)}} \,ds  = \frac{ e^{-u^2/2t}}{\pi \sqrt{t(1-t)}} .    
\eqen
\end{proof}
%FullSimplify[Abs[u]/(Sqrt[ 2] Pi^(3/2)) Integrate[(s^(-3/2) E^(-u^2/(2 s)))/(2 Sqrt[(1 - t) (t - s)]), {s, 0, t}], {u > 0, 0 < t < 1}]

\begin{lem} \label{prop-Z-last-hit}
Let $Z = (U,V)$ be a correlated two-dimensional Brownian motion as in~\eqref{eqn-bm-cov}. For $T>0$ and $u \in \BB R$, let $\tau_u(T)$ be the first time $t\in [0,T]$ such that $ U(s) > u$ for each $s \in (t,T]$; or $\tau_u(T) = 0$ if no such time exists. The joint density of $\tau_u(T) $ and $V(\tau_u(T) )$ on the event $\{\tau_u(T)  > 0\}$ is given by
\eqb \label{eqn-Z-last-hit}
\BB P\left( t <  \tau_u(T)  < t+ dt ,\, v <  V(\tau_u(T)  ) < v +  dv\right) =  \frac{ a_0    }{ t    \sqrt{ T-   t }  } \exp\left(-  \frac{ a_2 u^2 + a_3(  v-a_1   u )^2  }{   t} \right)  \, dt \, dv  
\eqe 
for $(t,v) \in (0,1) \times \BB R$, where $a_0,a_1,a_2,a_3>0$ are constants depending only on $p$.
\end{lem}
\begin{proof}
First consider the case where $T = 1$. The density of $ \tau_u(T)  $ with respect to Lebesgue measure on the event $\{ \tau_u(T) > 0\}$ is computed in Lemma~\ref{prop-bm-last-hit}. The Brownian motion  $V$ can be written as the sum of a constant $a_1$ (depending only on $p$) times $U$; and a Brownian motion $\wt V$ which is independent from $U$. On the event $\{\tau_u(T) > 0\}$, we therefore have $V(\tau_u(T)) = \wt V(\tau_u(T)) + a_1 u$. Hence the conditional law of $  V(\tau_u(T))$ given $\tau_u(T)$ on the event $\{\tau_u(T) > 0\}$ is that of a Gaussian random variable with mean $a_1 u$ and variance a constant depending only on $p$ times $\tau_u(T)$. 
This yields the formula~\eqref{eqn-Z-last-hit} in the case $T = 1$. 
%\BB P\left( \tau_u(T)  \in \, dt ,\, V(\tau_u(T)  ) \in \,dv\right) =  \frac{ a_0  }{ t \sqrt{   1-t }  } \exp\left(-    \frac{u^2 + (v-a_1 u )^2  }{ a_2 t} \right)  \, dt \, dv
For general $T>0$, we have by Brownian scaling that
%consider $T^{ 1/2} Z(T^{-1} \cdot)$, which has the same law as $Z$ and hits $ u$ for the last time before $T$ at time $T \tau_{T^{-1/2} u}(1)$
\eqbn
(\tau_u(T) , V(\tau_u(T))) \eqD \left(T   \tau_{T^{-1/2} u}(1) , T^{ 1/2} V( \tau_{T^{-1/2} u}(1) ) \right).
\eqen
This yields the formula~\eqref{eqn-Z-last-hit} in general.
\end{proof}
\begin{comment}
We have
\alb
\BB P\left( \tau_u > a ,\, V(\tau_u) > b   \right)
&= \BB P\left( T  \tau_{T^{-1/2} u}^1 > a,\,  T^{ 1/2} V( \tau_{T^{- 1/2} u}^1  > b   \right)\\
&= \int_{T^{-1} a}^\infty \int_{T^{-1/2} b}^\infty f_{T^{-1/2} u} (t,v     ) \,dt\,dv .
\ale
Make the change of variables $\wt t = T  t$, $\wt v = T^{ 1/2} v$, $dt = T^{-1} d\wt t$, $dv = T^{-1/2} d \wt v$ to get that this equals
\alb
\int_{ a}^\infty \int_{  b}^\infty T^{-3/2} f_{T^{-1/2} u} (T^{-1} t, T^{-1/2} v     ) \,dt\,dv .
\ale
This yields a density of
\alb
&\frac{ a_0 T^{-3/2}   }{ T^{-1} t \sqrt{ 1-T^{-1}  t }  } \exp\left(-    \frac{a_2 T^{-1} u^2 + a_3(T^{-1/2} v-a_1 T^{-1/2} u )^2  }{   T^{-1} t} \right)  \, dt \, dv \\
&=\frac{ a_0 T^{-1/2}   }{ t  \sqrt{ 1-T^{-1}  t }  } \exp\left(-     \frac{ a_2 u^2 + a_3(  v-a_1   u )^2  }{ a_2  t} \right)  \, dt \, dv \\
\ale
&=\frac{ a_0    }{ t \sqrt{ T-  t }  } \exp\left(-     \frac{a_2 u^2 + a_3(  v-a_1   u )^2  }{ a_2  t} \right)  \, dt \, dv \\
\ale
\end{comment}

We want to use Lemma~\ref{prop-Z-last-hit} and the Markov property of Brownian motion to prove an estimate for a Brownian motion conditioned to stay in the first quadrant. 
For this purpose we will first need the following lemma.

\begin{lem} \label{prop-last-unconditioned}
Suppose we are in the setting of Lemma~\ref{prop-Z-last-hit}.   
For $u,v,\ep_1,\ep_2 > 0$ and $T>0$, let
\eqbn
  E_{u,v}^{\ep_1,\ep_2}(T) := \left\{ ( \tau_u(T) , V(  \tau_u(T))) \in [T-\ep_1^2 ,T  ]\times [v-\ep_2 , v+\ep_2 ] \right\} .
\eqen
For $b  > 0$ and $\zeta \in (0,1)$, also let  
\eqbn
H_u^\zeta(T) := \left\{ \inf_{t\in [0,\tau_u]} U(t) \geq -\zeta \:\op{or}\: \inf_{t\in [0,\tau_u]} V(t) \geq -\zeta \right\} .
\eqen 
Then for $T>0$, $(u,v) \in [2b , \infty) \times (0,\infty)$ and $\ep_1,\ep_2 , \zeta \in (0,1)$, 
\eqb \label{eqn-last-unconditioned}
\BB P\left( E_{u,v}^{\ep_1,\ep_2} \cap H_b^\zeta(T)   \right)  \preceq  \ep_1 \ep_2 \zeta 
\eqe 
with the implicit constant depending only on $b$. 
\end{lem}
\begin{proof}
It follows from Lemma~\ref{prop-Z-last-hit} that for $T>0$, $(u,v) \in [  b , \infty) \times (0,\infty)$ and $\ep_1,\ep_2 ,\zeta \in (0,1)$,  
\eqb \label{eqn-last-uncond-E}
\BB P\left( E_{u,v}^{\ep_1,\ep_2}(T) \right)  \preceq   \ep_1 \ep_2
\eqe 
with the implicit constant depending only on $b$.  
We will deduce~\eqref{eqn-last-unconditioned} from this estimate using the strong Markov property.
Let $\sigma_b$ be the minimum of $T$ and the smallest $t  >0$ for which $U(t) = b$. 
For $u \geq b$ we have $\sigma_b \leq \tau_u$ on the event $\{\tau_u > 0\}$, so if $\tau_u > 0$ and $H_u^\zeta(T)$ occurs then the event 
\eqbn 
\wt H_b^\zeta(T) := \left\{ \inf_{t\in [0,\sigma_b]} U(t) \geq -\zeta \:\op{or}\: \inf_{t\in [0,\sigma_b]} V(t) \geq -\zeta \right\} 
\eqen
occurs. 
The probability of this event is at most a $b$-dependent constant times $\zeta$. 
By the strong Markov property, if $T,u,v,\ep_1,\ep_2,\zeta$ are as in~\eqref{eqn-last-unconditioned} then
\eqbn
\BB P\left(  E_{u,v}^{\ep_1,\ep_2}(T) \,|\, Z|_{[0,\sigma_b]} \right)  \BB 1_{\wt H_b^\zeta(T)}
 = \BB P\left( E_{u-U(\sigma_b) , v - V(\sigma_b)}^{\ep_1,\ep_2}(T - \sigma_b) \right) \BB 1_{\wt H_b^\zeta(T)} 
 \preceq \ep_1 \ep_2 \BB 1_{\wt H_b^\zeta(T)} ,
\eqen
where here we have used that $u\geq 2b$ so $u-U(\sigma_b) \geq b$. We conclude by taking expectations of both sides and recalling~\eqref{eqn-last-uncond-E} and our estimate for $\BB P( \wt H_b^\zeta(T))$. 
\end{proof}

Now we can prove our desired estimate for Brownian motion conditioned to stay in the first quadrant. 
 
\begin{lem} \label{prop-bm-last-cond}
Let $\wh Z = (\wh U ,\wh V)$ have the law of a correlated Brownian motion as in~\eqref{eqn-bm-cov} conditioned to stay in the first quadrant until time 1. For $u > 0$, let $\wh\tau_u$ be the first time $t\in [0,1]$ such that $\wh U(s) > u$ for each $s \in (t,1]$; or $\wh\tau_u = 0$ if no such time exists.

Fix $C>1$.
For each $(u,v) \in [C^{-1} , \infty) \times (0,\infty)$ and each $\ep_1 , \ep_2 \in (0,1)$, 
\eqb \label{eqn-last-cond}
\BB P\left((\wh\tau_u , \wh V(\wh\tau_u)) \in [1-\ep_1^2 , 1]\times [v-\ep_2 , v+\ep_2 ]\right) \preceq \ep_1  \ep_2 
\eqe 
with the implicit constant depending only on $C$. 

Furthermore, for $b\in (0,C^{-1})$ and $\zeta \in (0, (2C)^{-1})$, let $F_{u,b}(\zeta) $ be the event that there is a $t \in [0,\wh\tau_u]$ such that $ \wh U(t)  \leq \zeta$ and $\wh V(t) \geq b$. For each $(u,v) \in [C^{-1} , \infty) \times (0,\infty)$ and $\ep_1 , \ep_2 \in (0,1)$, we have
\eqb \label{eqn-last-cond-bad-event}
\BB P\left((\wh\tau_u , \wh V(\wh\tau_u)) \in [1-\ep_1^2 ,1 ]\times [v-\ep_2 , v+\ep_2 ],\, F_{u,b}(\zeta) \right) \leq \ep_1  \ep_2 o_\zeta(1)
\eqe 
where here the $o_\zeta(1)$ depends only on $\zeta$, $b$, and $C$. 
\end{lem}
\begin{proof}
For $(u,v) \in [C^{-1} , \infty) \times (0,\infty)$ and $\ep_1,\ep_2\in (0,1)$, let
\eqbn
\wh E_{u,v}^{\ep_1,\ep_2} := \left\{ (\wh\tau_u , \wh V(\wh\tau_u)) \in [1-\ep_1^2 ,1  ]\times [v-\ep_2 , v+\ep_2 ] \right\} 
\eqen
be the event appearing in the lemma.
The idea of the proof is to condition on $\wh Z$ run up to an appropriate time and use the Markov property and Lemma~\ref{prop-last-unconditioned} to estimate the conditional probability of $\wh E_{u,v}^{\ep_1,\ep_2}$. 

Let $Z = (U,V)$ be an unconditioned Brownian motion started from 0 with variances and covariances as in~\eqref{eqn-bm-cov} and let $\BB P^0$ denote its law.
Define the events $E_{u,v}^{\ep_1,\ep_2}(T)$ and $ H_u^\zeta(T)$ as in Lemma~\ref{prop-last-unconditioned} and for $T>0$ and $z = (z_1,z_2) \in (0,\infty)^2$, let
\eqbn
G^z(T) :=  \left\{ U(t) \geq -z_1 ,\, V(t) \geq -z_2 ,\, \forall t\in [0,T] \right\} .
\eqen
Let $\sigma$ be a stopping time for $\wh Z$ with $\sigma \leq 1$ a.s.\ (we will use a different choice of $\sigma$ for~\eqref{eqn-last-cond} and for~\eqref{eqn-last-cond-bad-event}).
By the Markov property of $\wh Z$ (see~\cite[Lemma 3.1]{gms-burger-cone} or~\cite[Section 3]{shimura-cone}), the conditional law of $(\wh Z(\cdot+\sigma) - \wh Z(\sigma))|_{[0,1-\sigma]}$ given $\wh Z|_{[0,\sigma]}$ is the same as the conditional law of $Z|_{[0,1-\sigma]}$ given $G^{\wh Z(\sigma)}(1-\sigma)$. 
Furthermore, on the event $\{\wh U(\sigma) \leq u\}$, 
\eqbn
 \{ \tau_{u-\wh U(\sigma)}(1-\sigma)  > 0\} \cap G^{\wh Z(\sigma)}(1-\sigma) \subset   H^{\wh U(\sigma) \wedge \wh V(\sigma) }_{u - \wh U(\sigma)}(1-\sigma) 
\eqen
with the latter events defined in terms of the process $(\wh Z(\cdot+\sigma) - \wh Z(\sigma))|_{[0,1-\sigma]}$.
Therefore, on $\{\wh U(\sigma) \leq u\}$,
\begin{align} \label{eqn-last-markov0}
\BB P\left(\wh E_{u,v}^{\ep_1,\ep_2}  , \, \wh\tau_u > \sigma \,|\, \wh Z|_{[0,\sigma]} \right)  
&= \BB P^0\left( E_{u - \wh U(\sigma) ,v - \wh V(\sigma) }^{\ep_1,\ep_2}(1-\sigma) \,|\, G^{\wh Z(\sigma)}(1-\sigma) \right) \notag \\
&\leq \frac{   \BB P^0 \left( E_{u - \wh U(\sigma) ,v - \wh V(\sigma) }^{\ep_1,\ep_2}(1-\sigma) \cap H^{\wh U(\sigma) \wedge \wh V(\sigma) }_{u - \wh U(\sigma)}(1-\sigma) \right)     }{\BB P^0\left( G^{\wh Z(\sigma)}(1-\sigma)  \right)}  .
\end{align}   
If $ \wh U(\sigma) \vee \wh V(\sigma)  \geq a$ for some $a>0$, then
\eqb \label{eqn-last-markov0'}
\BB P^0\left( G^{\wh Z(\sigma)}(1-\sigma)  \right)  \succeq  1 \wedge (\wh U(\sigma) \wedge \wh V(\sigma) )
\eqe  
with the implicit constant depending only on $a$. If this is the case and also $u - \wh U(\sigma) \geq a$, then Lemma~\ref{prop-last-unconditioned} and~\eqref{eqn-last-markov0} yields
\eqb \label{eqn-last-markov}
\BB P\left(\wh E_{u,v}^{\ep_1,\ep_2}  , \, \wh\tau_u > \sigma \,|\, \wh Z|_{[0,\sigma]} \right) 
\preceq \ep_1 \ep_2 . 
\eqe 
 
To prove~\eqref{eqn-last-cond}, suppose $u\geq C^{-1}$ and let $\sigma$ be the smallest $t \in [0,1]$ for which $\wh U(t) \geq \frac12 C^{-1}$ (or $\sigma = 1$ if no such $t$ exists). Then $\{\wh\tau_u > \sigma\} \subset \wh E_{u,v}^{\ep_1,\ep_2}  $ and on this event $(u - \wh U(\sigma)) \wedge  \wh U(\sigma) \geq \frac12 C^{-1}$ so~\eqref{eqn-last-markov} yields~\eqref{eqn-last-cond}.
 
Next we consider~\eqref{eqn-last-cond-bad-event}. If we let $\sigma$ be the smallest $t\in [0,1]$ for which $ \wh U(t)  \leq \zeta$ and $\wh V(t) \geq b$ (or $\sigma = 1$ if no such $u$ exists) then with $F_{u,b}(\zeta)$ as in~\eqref{eqn-last-cond-bad-event}, 
\eqbn
F_{u,b}(\zeta)\subset  \{ \sigma  < 1\} .
\eqen
Since $\wh Z$ a.s.\ does not hit the boundary of the first quadrant after time 0, it follows that $\BB P\left( \sigma < 1  \right) = o_\zeta(1)$, at a rate depending only on $b$. By choosing this $\sigma$ in~\eqref{eqn-last-markov}, we obtain that the left side of~\eqref{eqn-last-cond-bad-event} is at most
\eqbn
\BB E\left(  \BB P\left(\wh E_{u,v}^{\ep_1,\ep_2}  , \, \wh\tau_u > \sigma \,|\, \wh Z|_{[0,\sigma]} \right)  \BB 1_{\sigma < 1} \right) \preceq \ep_1\ep_2 o_\zeta(1) . \qedhere
\eqen
\end{proof}

\subsubsection{Proof of the proposition}
\label{sec-no-order-local-endpoint}

In this subsection, we will establish Proposition~\ref{prop-no-order-endpoint}.  
We continue to use the notation introduced in Section~\ref{sec-no-order-local-setup} plus the notation in the statement of Proposition~\ref{prop-no-order-endpoint} (recall in particular the notation $m_h^\delta := \lfloor (1-\delta) h \rfloor$). We also introduce the following additional notation. For $\delta \in (0,1/2)$ and $C>1$, let $\BB k_\delta =\BB k_\delta(C) $ be the smallest $k\in\BB N$ for which $2^k\delta \geq 2C$. For $n\in\BB N$, $(k_1 , k_2) \in (-\infty , \BB k_\delta]_{\BB Z} \times [1 , \BB k_\delta)_{\BB Z}$, and $(h,c) \in \BB N^2$, let $\mcl U_n^{k_1 , k_2} (\delta, h,c)$ be the set of those pairs $(r,l)\in\BB N^2$ such that 
\alb
n - r   \in \left[2^{2(k_1 - 1)} \delta^2 h^2 , 2^{2 k_1  } \delta^2 h^2 \right]_{\BB Z} \quad \op{and} \quad |l - c| \in \left[2^{k_2-1} \delta h , 2^{k_2} \delta h \right]_{\BB Z} .
\ale  
Also let $\mcl U_n^{k_1 , \BB k_\delta} (\delta, h,c)$ be the set of those pairs $(r,l)\in\BB N^2$ such that 
\alb
n - r  \in \left[2^{2(k_1 - 1)} \delta^2 h^2, 2^{2 k_1  } \delta^2 h^2 \right]_{\BB Z} \quad \op{and} \quad |l - c| \in \left[2^{\BB k_\delta} \delta h , \infty \right)_{\BB Z} ;
\ale  
and let $\mcl U_n^{k_1 , 0} (\delta, h,c)$ be the set of those pairs $(r,l)\in\BB N^2$ such that 
\alb
n - r \in \left[2^{2(k_1 - 1)} \delta^2 h^2, 2^{2 k_1  } \delta^2 h^2 \right]_{\BB Z} \quad \op{and} \quad |l - c| \in \left[0  ,  \delta h \right]_{\BB Z} .
\ale  
Note that if $(h,c) \in [C^{-1} n^2 , C n^2]_{\BB Z}$ (which is the setting of Proposition~\ref{prop-no-order-endpoint}), then the union of the sets $\mcl U_n^{k_1 , k_2} (\delta, h,c)$ over all $(k_1,k_2) \in  (-\infty , \BB k_\delta]_{\BB Z} \times [0 , \BB k_\delta]_{\BB Z}$ covers all of $[0,n-1]_{\BB Z} \times [0,\infty)_{\BB Z}$, which is the possible range for the pair $(K_{n,m_h}^\delta , |Q_{n,m_h^\delta} - l|)$. 

The idea of the proof of Proposition~\ref{prop-no-order-endpoint} is to show that for a large enough choice of $A$, the conditional probability given $\mcl E_n^{h,c}$ that the pair $(K_{n,m_h^\delta}^H , Q_{n,m_h^\delta}^H)$ in the proposition statement belongs to one of the sets $\mcl U^{k_1,k_2}_n(\delta,h,c)$ which is not contained in the ``good" region $\mcl U_n^\delta(A,h,c)$ in the proposition statement is small. 
For this purpose, we will need to show that if $\mcl U^{k_1,k_2}_n(\delta,h,c) \not\subset \mcl U_n^\delta(A,h,c)$ then
\begin{itemize}
\item If we condition on $\{I> n\}$, then it is unlikely that $ (K_{n,m_h^\delta}^H , Q_{n,m_h^\delta}^H) \in \mcl U^{k_1,k_2}_n(\delta,h,c) $. 
\item If we condition on $\{I > n\} \cap \{(K_{n,m_h^\delta}^H , Q_{n,m_h^\delta}^H) \in \mcl U^{k_1,k_2}_n(\delta,h,c)\}$, then it is unlikely that $\mcl E_n^{h,c}$ occurs.
\end{itemize} 
We start with the following basic estimate for the probability that the pair $(K_{n,m_h^\delta}^H , Q_{n,m_h^\delta}^H)$ belongs to one of the above defined sets, which will be deduced from the estimates of the previous subsection together with Lemma~\ref{prop-last-conv}.

\begin{lem} \label{prop-K-uniform}
Fix $C>1$. For $\delta\in (0,1/2)$, $n\in\BB N$, $(h,c) \in \left[C^{-1} n^{1/2} , C n^{1/2}\right]_{\BB Z}^2$, and $(k_1,k_2) \in (-\infty ,\BB k_\delta] \times [0,\BB k_\delta]_{\BB Z}$, we have
\eqb \label{eqn-K-uniform}
\BB P\left((K_{n,m_h^\delta}^H , Q_{n,m_h^\delta}^H) \in \mcl U_n^{k_1,k_2}(\delta, h,c) \,|\ I > n\right) \preceq   2^{ k_1 + k_2} \delta^2   + o_n(1) 
\eqe 
with the implicit constant depending only on $C$ and the rate of the $o_n(1)$ depending only on $\delta$ and $C$. Furthermore, for $\zeta>0$ and $b \in (0,C^{-1})$, let $F_{n,b}^H(\zeta)  $ be the event that there is an $i \in [1,K_{n,m_h^\delta}^H]_{\BB Z}$ such that $\mcl N_{\tc C}\left(X(1,i)\right) \geq b n^{1/2}$ and $\mcl N_{\tc H}\left(X(1,i)\right) \leq \zeta n^{1/2}$. Then we have
\eqb \label{eqn-K-uniform-zeta}
\BB P\left((K_{n,m_h^\delta}^H , Q_{n,m_h^\delta}^H) \in \mcl U_n^{k_1,k_2}(\delta, h,c) ,\,  F_{n,m_c^\delta}^H(\zeta)  \,|\ I > n\right) \leq 2^{k_1+k_2} \delta^2 o_\zeta(1) +o_n(1) ,
\eqe 
with the $o_\zeta(1)$ depending only on $\zeta$, $b$, and $C$ and the rate of the $o_n(1)$ depending only on $\zeta$, $b$, $\delta$, and $C$. 
\end{lem} 

We emphasize that the $o_n(1)$ error in~\eqref{eqn-K-uniform} tends to 0 as $n\rta\infty$ for fixed values of $\delta$ and $C$. In particular, for each fixed $\ep \in (0,1)$ and $\delta \in (0,1/2)$ we can choose $n$ large enough that this error is smaller than $\ep \delta^2$, uniformly over all possible values of $(k_1,k_2)$. Similarly for the error in~\eqref{eqn-K-uniform-zeta}. 

\begin{proof}[Proof of Lemma~\ref{prop-K-uniform}]
First we prove~\eqref{eqn-K-uniform}. Fix $\ep  \in (0,1/100)$. Define $\wh Z = (\wh U , \wh V)$ and $\wh\tau_u$ for $u > 0$ as in Lemma~\ref{prop-bm-last-cond}. 
For $\delta>0$, let $\ul{\BB k}_\delta^\ep$ be the smallest $k\in \BB Z$ for which $2^{2k }\delta^2 \geq 100\ep$. 
By Lemma~\ref{prop-bm-last-cond}, for $(k_1,k_2) \in [\ul{\BB k}_\delta^\ep  , \BB k_\delta]_{\BB Z} \times [\ul{\BB k}_\delta^\ep \vee 1 , \BB k_\delta]_{\BB Z}$ and $(h,c) \in \left[C^{-1} n^{1/2} , C n^{1/2}\right]_{\BB Z}^2$, the probability that $(\wh\tau_{m_h^\delta /n^{1/2}} , \wh V(\wh\tau_{m_h^\delta/n^{1/2}} ))$ belongs to 
\begin{align} \label{eqn-K-uniform-bm}
  [1-  2^{ 2k_1} C \delta^2 - \ep , 1   ] \times [n^{-1/2} c - 2^{k_2} C \delta - \ep   , n^{-1/2} c + 2^{k_2} C \delta + \ep ]  
\end{align}
is at most a constant (depending only on $C$) times $2^{k_1 + k_2} \delta^2$. Furthermore, if $(\wh\tau_{m_h^\delta /n^{1/2}} , \wh V(\wh\tau_{m_h^\delta/n^{1/2}} ))$ does not belong to any of the sets~\eqref{eqn-K-uniform-bm} then either $1 - \wh\tau_{m_h^\delta /n^{1/2}} \leq 200 C \ep$ or $|\wh V(\wh\tau_{m_h^\delta/n^{1/2}}) - n^{-1/2} c| \leq 200 C \ep$. The same argument used to prove Lemma~\ref{prop-bm-last-cond} (but with only bounds for either $\wh\tau_u$ or $\wh V(\wh\tau_u)$, not both, specified) shows that the probability that this is the case is at most a constant (depending only on $C$) times $\ep^{1/2}$. 
 
By Lemma~\ref{prop-last-conv}, we can find $n_* = n_*(\ep , \alpha, \delta ,   C) \in\BB N$ such that for each $n\geq n_*$ and each $m \in [C^{-1} n^{1/2} , Cn^{1/2}]_{\BB Z}$, the Prokhorov distance between the conditional law of $(Z^n , n^{-1} K_{n,m}^H , n^{-1/2} Q_{n,m}^H )$ given $\{I >n\}$ and the law of $(\wh Z,\wh\tau_{m/n^{1/2}} , \wh V(\wh\tau_{m/n^{1/2}}))$ is at most $\ep$. The above estimates then imply that for $n\geq n_*$ and $h,c,k_1,k_2$ as in the statement of the lemma, 
\eqbn
 \BB P\left((K_{n,m_h^\delta}^H , Q_{n,m_h^\delta}^H) \in \mcl U_n^{k_1,k_2}(\delta, h,c) \,|\ I > n\right) \preceq  2^{ k_1 + k_2} \delta^2   + O_\ep(\ep^{1/2}) ,
\eqen
with implicit constants depending only on $C$. This proves~\eqref{eqn-K-uniform}. We similarly obtain~\eqref{eqn-K-uniform-zeta} using~\eqref{eqn-last-cond-bad-event} of Lemma~\ref{prop-bm-last-cond}. 
\end{proof}

Our next estimate is an upper bound for the conditional probability of $\mcl E_n^{h,c}$ given $\{I > n\} \cap \{(K_{n,m_h^\delta}^H , Q_{n,m_h^\delta}^H) \in \mcl U^{k_1,k_2}_n(\delta,h,c)\}$, but with the additional regularity condition that $Q_{n,m_h^\delta}^H$ is not unusually small. This regularity condition will be shown to hold with high conditional probability given $\mcl E_n^{h,c}$ in Lemma~\ref{prop-bdy-regularity} below. 

\begin{lem} \label{prop-endpoint-realization}
Fix $C>1$. For $(h,c)\in\BB N^2$, $k_1,k_2\in\BB N$, and $\delta>0$, define $\BB k_\delta$ and $ \mcl U_n^{k_1,k_2}(\delta, h,c)$ as in the discussion just above Lemma~\ref{prop-K-uniform}. Fix $C >1$ and $\zeta>0$. For $\delta  \in(0,1)$, let $m_h^\delta := \lfloor (1-\delta) h\rfloor$ be as in Lemma~\ref{prop-no-order-endpoint}. There exists $\delta_* \in (0,1)$ (depending only on $C$ and $\zeta$) such that for each $\delta\in (0,\delta_*]$, there exists $n_* = n_*(\delta , C, \zeta) \in \BB N$ such that the following is true.

Suppose $n \geq n_*$, $(h,c) \in \left[C^{-1} n^{1/2} , C n^{1/2}\right]_{\BB Z}$, and $(k_1,k_2) \in (-\infty , \BB k_\delta]_{\BB Z} \times [0,\BB k_\delta]_{\BB Z}$. Let $x$ be a realization of $X_1\dots X_{K_{n,m_h^\delta}^H}$ for which $(K_{n,m_h^\delta}^H ,Q_{n,m_h^\delta}^H) \in \mcl U_n^{k_1,k_2}(\delta,h,c)$ and $Q_{n,m_h^\delta}^H \geq \zeta n^{1/2}$. There is a constant $a_0 > 0$, depending only on $C$ and $\zeta$, such that
\begin{align} \label{eqn-endpoint-realization}
 \BB P\left( \mcl E_n^{h,c} \,|\, X_1\dots X_{K_{n,m_h^\delta}^H} = x , \, I > n \right) 
 \preceq 
\begin{cases}
2^{-3 k_1  }  \exp\left( -a_0 2^{-2k_1 + 2k_2 } \right)  \delta^{-2} n^{-1} ,\quad &k_1 \geq 0\\
2^{ k_1} \exp\left( -a_0 2^{-2 k_1} - a_0   2^{-   k_1  +  k_2 }  \right) \delta^{-2} n^{-1} ,\quad &k_1 < 0  ,
\end{cases}
\end{align} 
with the implicit constant depending only on $C$ and $\zeta$.  
\end{lem}
\begin{proof}%[Proof of Lemma~\ref{prop-endpoint-realization}]
We will prove the lemma using Lemma~\ref{prop-J^H-D-formulas} and the estimates of Section~\ref{sec-J^H-local}.
Let $(h,c) \in \left[C^{-1} n^{1/2} , C n^{1/2}\right]_{\BB Z}^2$ and let $r_h^\delta := h-m_h^\delta \asymp \delta n^{1/2}$. 
Let $x$ be a realization of $X_1 \dots X_{K_{n,m_h^\delta}^H}$ as in the statement of the lemma, so that $\mcl R(x) $ contains no orders, $(|x| , \frk c(x) ) \in \mcl U_n^{k_1,k_2}(\delta,h,c)$, and $\frk c(x)  \geq \zeta n^{1/2}$. Also define $(J_{n,r_h^\delta}^H , L_{n,r_h^\delta}^H)$ as in Section~\ref{sec-no-order-local-setup} and let $R_n(x)$ be as in~\eqref{eqn-J^H-hit-event}. 

Since $\frk c(x) \geq \zeta n^{1/2}$, Lemma~\ref{prop-J^H-time-lower} implies that (if $k_1$ is such that $n-|x| \geq 1$, so that $ \mcl E_n^{h,c} \not=\emptyset$),
\eqb \label{eqn-R_n(x)-lower}
\BB P\left( \mcl N_{\tb C}\left(X(|x|+1 , n)\right) \leq \frk c(x)    ,\,  R_n(x) \right) \succeq (n-|x|)^{-1/2}  \asymp 2^{- k_1} \delta^{-1} n^{-1/2} 
\eqe 
with the implicit constant depending only on $\zeta$.  

In the case $k_1 \geq 0$, Proposition~\ref{prop-J^H-local} and Lemma~\ref{prop-g-form} imply
\alb
\BB P\left((J_{n,r_h^\delta}^H , L_{n,r_h^\delta}^H) = (n-|x|-1 , c - \frk c(x)) \right)& \preceq  \left(2^{-4 k_1  } \exp\left(-a_0 2^{-2k_1 + 2k_2 }\right) + o_{\delta h}(1) \right) \delta^{-3} h^{-3} \\
&\preceq  \left(2^{-4 k_1 } \exp\left( -a_0 2^{-2k_1 + 2k_2} \right) + o_{\delta^2 n}(1) \right) \delta^{-3} n^{-3/2}
\ale
with $a_0 > 0$ and the implicit constants and the rate of the $o_{\delta^2 n}$ depending only on $C$ and $\zeta$ (and in particular not on $x$, $k_1$, or $k_2$). It follows that we can find $n_* = n_*(\delta , \zeta , C)$ such that for each $(k_1,k_2) \in [0, \BB k_\delta]_{\BB Z} \times [0,\BB k_\delta]_{\BB Z}$, and each realization $x$ as above, we have 
\eqb \label{eqn-endpoint-realization'}
\BB P\left((J_{n,r_h^\delta}^H , L_{n,r_h^\delta}^H) = (n-|x|-1 , c - \frk c(x)) \right) \preceq  2^{-4 k_1  }  \exp\left(-a_0 2^{-2k_1 + 2k_2 }\right) \delta^{-3} n^{-3/2}
\eqe  
By combining this with~\eqref{eqn-x-h-c-prob-I} and~\eqref{eqn-R_n(x)-lower}, we obtain~\eqref{eqn-endpoint-realization} in the case $k_1\geq 0$.  

In the case $k_1 < 0$,~\eqref{eqn-J^H-long-exact-sup} of Lemma~\ref{prop-J^H-long-exact} (applied with $m = r_h^\delta$, $k = n-|x|-1 \asymp 2^{2k_1} \delta^2 n$ and $R = |c - \frk c(x)| \asymp 2^{k_2} \delta n^{1/2}$) implies
\alb
\BB P\left((J_{n,r_h^\delta}^H , L_{n,r_h^\delta}^H) = (n-|x|-1 , c - \frk c(x)) \right) 
&\preceq   \exp\left( -a_0 2^{-2 k_1} - a_0   2^{-   k_1  +  k_2 }        \right)   \delta^{-3} n^{-3/2} 
\ale
for a constant $a_0 > 0$ depending only on $C$. By combining this with~\eqref{eqn-x-h-c-prob-I} and~\eqref{eqn-R_n(x)-lower} we obtain~\eqref{eqn-endpoint-realization} in the case $k_1 <  0$.
\end{proof}

To complement Lemma~\ref{prop-endpoint-realization}, we will now prove that $Q_{n,m_h^\delta}^H$ is unlikely to be smaller than $\zeta n^{1/2}$. 

\begin{lem} \label{prop-bdy-regularity}
Let $C>1$ and $q \in (0,1)$. For $(h,c) \in \BB N^2$ and $\delta >0$, write $m_h^\delta = \lfloor (1-\delta) h\rfloor$ (as above) and $m_c^\delta = \lfloor (1-\delta) c \rfloor$. There exists $\delta_* \in (0,1)$ and $\zeta > 0$ (depending only on $C$ and $q$) such that for each $\delta \in (0,\delta_*]$, we can find $n_* = n_*(\delta, \zeta , C , q) \in \BB N$ such that for $n\geq n_*$ and $(h,c) \in \left[C^{-1} n^{1/2} , C n^{1/2}\right]_{\BB Z}$, we have
\eqbn
\BB P\left( Q_{n,m_h^\delta}^H \wedge Q_{n,m_c^\delta}^C \geq \zeta n^{1/2}      \,|\, \mcl E_n^{h,c} \right) \geq 1-q .
\eqen
\end{lem}
\begin{proof}
Let $(h,c) \in \left[C^{-1} n^{1/2} , C n^{1/2}\right]_{\BB Z}$. 
By definition, 
\eqbn
\inf_{i \in [K_{n,m_h^\delta}^H +1 , n]_{\BB Z}} \mcl N_{\tc H}\left(X(1,i)\right) \geq m_h^\delta \geq (2C)^{-1} n^{1/2}
\eqen
so if $\delta \in (0, 1/2)$, $ Q_{n,m_c^\delta}^C < (2C)^{-1} n^{1/2}$, $I>n$, and $\mcl E_n^{h,c}$ occurs, then $K_{n,m_h^\delta}^H \geq K_{n,m_c^\delta}^C$ and hence  
\eqbn
Q_{n,m_h^\delta}^H \geq \inf_{i \in [K_{n,m_c^\delta}^C +1 , n]_{\BB Z}} \mcl N_{\tc C}\left(X(1,i)\right) \geq m_c^\delta \geq   (2C)^{-1} n^{1/2} .
\eqen
Therefore, for $\zeta \leq (2C)^{-1}$, 
\begin{align} \label{eqn-bdy-reg-reduce}
 \BB P\left(  \mcl E_n^{h,c}      ,\,     Q_{n,m_c^\delta}^C < \zeta n^{1/2} \,|\, I > n  \right)  
 \leq \BB P\left(   \mcl E_n^{h,c}      ,\,     Q_{n,m_h^\delta}^H \geq (2C)^{-1} n^{1/2} ,\, F_{n,b}^H(\zeta) \,|\, I > n  \right)  ,
\end{align}
where $F_{n,b}^H(\zeta)$ is the event of Lemma~\ref{prop-K-uniform} with $b  = (2C)^{-1}$. 
We will now estimate the right side of~\eqref{eqn-bdy-reg-reduce} using~\eqref{eqn-K-uniform-zeta} and a union bound argument.

By Lemma~\ref{prop-endpoint-realization} (applied with $(2C)^{-1}$ in place of $\zeta$), we can find $  \delta_*  \in (0,1)$, depending only on $C$, such that for each $\delta \in (0, \delta_*]$, there exists $\wt n_* = \wt n_*(\delta,C) \in\BB N$ such that for each $n\geq n_*$, each $(h,c) \in \left[C^{-1} n^{1/2} , C n^{1/2}\right]_{\BB Z}$, and each $(k_1,k_2) \in (-\infty , \BB k_\delta]_{\BB Z} \times [0,\BB k_\delta]_{\BB Z}$,  
\begin{align} \label{eqn-bdy-reg-split}
 \BB P\left(  \mcl E_n^{h,c}   \,|\, E^{k_1,k_2}_n(\delta , h,c)    ,\, I> n  \right) \preceq 
\begin{cases}
 2^{-3 k_1  } \exp\left(- a_0 2^{-2k_1 + 2k_2 }\right) \delta^{-2} n^{-1} ,\qquad &k_1 \geq 0 \\
2^{ k_1} \exp\left(   -a_0 2^{-2k_1}  - a_0   2^{-k_1  + k_2  }        \right)  \delta^{-2} n^{-1},\qquad &k_1 < 0  ,
\end{cases}  
\end{align}
with $a_0 > 0$ the implicit constants depending only on $C$, where here
\eqbn
E^{k_1,k_2}_n(\delta , h,c) := \left\{(K_{n,m_h^\delta}^H ,Q_{n,m_h^\delta}^H) \in \mcl U_n^{k_1,k_2}(\delta,h,c) ,\,   Q_{n,m_h^\delta}^H \geq (2C)^{-1} n^{1/2} ,\, F_{n,b}^H(\zeta) \right\} .
\eqen 

By the second assertion of Lemma~\ref{prop-K-uniform}, for any given $\alpha \in (0,1)$, we can find $\zeta>0$ (depending on $C$ and $\alpha$) such that for each $\delta \in (0,\delta_*]$, there exists $n_* = n_*(\delta , \zeta  ,C) \geq \wt n_*$ such that for each $n\geq n_*$, $(k_1,k_2) \in  (-\infty , \BB k_\delta]_{\BB Z} \times [0,\BB k_\delta]_{\BB Z}$, and $(h,c) \in \left[C^{-1} n^{1/2} , C n^{1/2}\right]_{\BB Z}$,  
\eqbn
\BB P\left( E^{k_1,k_2}_n(\delta , h,c) \,|\, I >n   \right) \preceq    2^{k_1 \vee 0 + k_2  } \delta^2    \alpha 
\eqen  

By~\eqref{eqn-bdy-reg-split}, for $n\geq n_*$ and $(k_1,k_2 )\in (-\infty , \BB k_\delta]_{\BB Z} \times [0,\BB k_\delta]_{\BB Z}$ we have
\alb
\BB P\left(  \mcl E_n^{h,c} \cap  E^{k_1,k_2}_n(\delta , h,c)   \,|\, I> n  \right) \preceq 
\begin{cases}
 2^{-2 k_1  + k_2  }  \exp\left(-a_0 2^{-2k_1 + 2k_2 }\right)  \alpha n^{-1}  ,\qquad &k_1 \geq 0 \\
2^{ k_1 + k_2 } \exp\left( -a_0 2^{-2k_1}  - a_0   2^{-k_1  + k_2  }  \right)  \alpha   n^{-1} , \qquad &k_1 < 0  ,
\end{cases}  
\ale
with the implicit constants depending only on $C$. By summing over all such $k_1$ and $k_2$, we infer
\eqbn
\BB P\left(   \mcl E_n^{h,c}    ,\, Q_{n,m_h^\delta}^H \geq (2C)^{-1} n^{1/2} ,\, F_{n,b}^H(\zeta) \,|\, I > n \right) \preceq \alpha n^{-1} .
\eqen
 
By combining this with~\eqref{eqn-bdy-reg-reduce} and Proposition~\ref{prop-no-order-lower}, we get
\eqbn
\BB P\left(       Q_{n,m_c^\delta}^C < \zeta n^{1/2} \,|\,\mcl E_n^{h,c} \right) 
= \frac{ \BB P\left(  \mcl E_n^{h,c} ,\,     Q_{n,m_c^\delta}^C < \zeta n^{1/2} \,|\, I > n  \right)      }{  \BB P\left( \mcl E_n^{h,c} \,|\, I > n \right)  }
\preceq \alpha ,
\eqen
with the implicit constant depending only on $C$. By symmetry and the union bound, also
\eqbn
\BB P\left(       Q_{n,m_c^\delta}^C \wedge Q_{n,m_h^\delta}^H < \zeta n^{1/2} \,|\, \mcl E_n^{h,c} \right) \preceq \alpha .
\eqen
By choosing $\alpha$ sufficiently small (and hence $\zeta$ sufficiently small and $n_*$ sufficiently large), depending only on $q $ and $C$, we conclude.  
\end{proof}

\begin{proof}[Proof of Proposition~\ref{prop-no-order-endpoint}]
Fix $\alpha>0$ to be chosen later, depending only on $q$ and $C$. 
Choose $\zeta>0$ and $ \delta_1  \in (0,1)$ such that the conclusion of Lemma~\ref{prop-bdy-regularity} holds with $\delta_1$ in place of $\delta_*$ and $\alpha$ in place of $q$. Then choose $\delta_* \in (0,\delta_1]$ such that the conclusion of Lemma~\ref{prop-endpoint-realization} holds with this choice of $\zeta$. Then for $\delta \in (0,\delta_*]$, there exists $\wt n_* = \wt n_*(\delta , C , \alpha)$ which simultaneously satisfies the conclusions of Lemmas~\ref{prop-endpoint-realization} and~\ref{prop-bdy-regularity}. 
Note that for $\delta \in (0,\delta_*]$ and $n\geq \wt n_*$, 
\eqb \label{eqn-use-bdy-reg}
\BB P\left( Q_{n,m_h^\delta}^H \leq \zeta n^{1/2}      \,|\, \mcl E_n^{h,c} \right) \leq \alpha .
\eqe
 
By the first assertion of Lemma~\ref{prop-K-uniform} for $\delta \in (0,\delta_*]$, we can find $ n_* = n_*^1(\delta , C , \alpha) \geq \wt n_*$ such that for $n\geq n_*$ and $(k_1,k_2) \in (-\infty , \BB k_\delta]_{\BB Z} \times [0,\BB k_\delta]_{\BB Z}$, we have
\eqbn
\BB P\left(    (K_{n,m_h^\delta}^H ,Q_{n,m_h^\delta}^H) \in \mcl U_n^{k_1,k_2}(\delta,h,c)    \,|\, I >n   \right) \preceq 2^{k_1 \wedge 0 + k_2  } \delta^2   ,
\eqen 
with the implicit constants depending only on $C$. By combining this with Lemma~\ref{prop-endpoint-realization}, we infer
\alb
&\BB P\left(   \mcl E_n^{h,c}    ,\,  (K_{n,m_h^\delta}^H ,Q_{n,m_h^\delta}^H) \in \mcl U_n^{k_1,k_2}(\delta,h,c) ,\,    Q_{n,m_h^\delta}^H \geq \zeta n^{1/2}   \,|\, I> n  \right)\\
& \qquad \preceq 
\begin{cases}
 2^{-2 k_1  + k_2   }  \exp\left(-a_0 2^{-2k_1 + 2k_2 }\right)   n^{-1}  ,\qquad &k_1 \geq 0 \\
2^{ k_1 + k_2 } \exp\left(-a_0 2^{-2k_1}  - a_0   2^{-k_1  + k_2  } \right)  n^{-1}  ,\qquad &k_1 < 0  , 
\end{cases}  
\ale
with $a_0 > 0$ and the implicit constants depending only on $C$ and $\zeta$. By summing this estimate over those $(k_1,k_2) \in (-\infty, \BB k_\delta]_{\BB Z} \times [0,\BB k_\delta]_{\BB Z}$ for which $\mcl U_n^{k_1,k_2}(\delta,h,c) \not\subset \mcl U_n^\delta(A,h,c)$, we infer that there exists $A > 2$, depending only on $C$, $\alpha$, and $\zeta$, such that with $\mcl U_n^\delta(A,h,c)$ as in the statement of the proposition,  
\eqbn
\BB P\left(  \mcl E_n^{h,c}  ,\,  (K_{n,m_h^\delta}^H ,Q_{n,m_h^\delta}^H)  \notin \mcl U_n^\delta(A,h,c)  ,\,  Q_{n,m_h^\delta}^H \geq \zeta n^{1/2} \,|\, I > n\right) \leq \alpha  n^{-1}  .
\eqen
By dividing this by the estimate of Proposition~\ref{prop-no-order-lower}, we get
\eqbn
\BB P\left(    (K_{n,m_h^\delta}^H ,Q_{n,m_h^\delta}^H)  \notin \mcl U_n^\delta(A,h,c)  ,\,  Q_{n,m_h^\delta}^H \geq \zeta n^{1/2} \,|\, \mcl E_n^{h,c}  \right) \preceq \alpha ,
\eqen
with the implicit constant depending only on $C$. 
By combining this with~\eqref{eqn-use-bdy-reg}, we find that
\eqbn
\BB P\left(    (K_{n,m_h^\delta}^H ,Q_{n,m_h^\delta}^H)  \notin \mcl U_n^\delta(A,h,c) \,|\, \mcl E_n^{h,c}   \right) \preceq \alpha ,
\eqen
with the implicit constant depending only on $C$. We now conclude by choosing $\alpha$ sufficiently small that $\alpha$ times this implicit constant is at most $q$.
\end{proof}

\section{Local estimates with few orders}
\label{sec-few-order-local}

In the previous section, we proved various estimates associated with the event that the reduced word $X(1,n)$ contains no orders and a particular number of burgers of each type. In this subsection we will use the results of the previous sections to prove analogous estimates for the event that a reduced word contains a small number of orders and approximately a particular number of burgers of each type. The main result of this subsection is Proposition~\ref{prop-corner-local} below, which is needed for the proof of the upper bound in Theorem~\ref{thm-empty-prob}, and will also be used in \cite{gms-burger-finite}. Unlike in Section~\ref{sec-no-order-local}, in this section we will read the word backward, rather than forward so as to make use of the backward stopping times $J^H$ of Section~\ref{sec-J^H-local}. 

Before stating and proving our main estimate, in Section~\ref{sec-interval-reg-var} we will prove a regular variation type estimate for the probability that there is a time $j$ in a given discrete interval with the property that $X(-j,-1)$ contains no orders.

\subsection{Probability of a time with no orders in a small interval}
\label{sec-interval-reg-var}
 
In this section we will prove a sharp estimate for the probability that there is a time slightly smaller than $n$ for which $X(-j,-1)$ contains no orders. We recall the definitions of regular varying and slowly varying from Section~\ref{sec-reg-var}. 

\begin{lem} \label{prop-P-interval-reg-var}
Let $\psi_0$ be the slowly varying function from Lemma~\ref{prop-I-reg-var}. There is a slowly varying function $\psi_2$ such that for $n\in\BB N$ and $k \in [1 ,  n]_{\BB Z}$, we have
\eqbn
\BB P\left(\text{$\exists \, j \in [n- k ,n  ]_{\BB Z}$ such that $X(-j,-1)$ contains no orders}\right) = (1+o_{k/n}(1)) \psi_0(n) \psi_2(k)  (k/n)^\mu ,
\eqen
with $\mu$ as in~\eqref{eqn-cone-exponent}.
\end{lem}

To prove Lemma~\ref{prop-interval-reg-var} we need some basic facts about regularly varying functions. We leave the proof of the first fact to the reader.

\begin{lem} \label{prop-sum-reg-var}
Let $g : (0,\infty) \rta (0,\infty)$ be bounded and regularly varying with exponent $\alpha \in (0,1)$. For $t > 0$, let $\wt g(t) := \int_0^{\lfloor t \rfloor} g(s) \, ds$. Then $\wt g$ is regularly varying with exponent $-(1-\alpha)$.
\end{lem}
 
\begin{comment}
\begin{proof}
Let $\wh g(t) := \int_0^t g(s) \, ds$. For each $t > 0$, we have
\eqbn
|\wh g(t) - \wt g(t)| = \int_{\lfloor t\rfloor}^t g(s) \, ds \leq g(\lfloor t \rfloor) .
\eqen
Since $\alpha \in (0,1)$, $\wt g(t) \rta \infty$ as $t\rta\infty$ and $g(t) \rta 0$ as $t\rta \infty$. Hence it suffices to show that $\wh g$ is regularly varying with exponent $-(1-\alpha)$. To this end,
fix $c>1$. By a change of variables, 
\eqbn
\frac{\wh g(c t)}{\wh g(t)} = \frac{c \int_0^{t}  g(c s ) \, ds}{\int_0^t  g(s) \, ds } .
\eqen
For each $\ep > 0$, there exists $t_\ep > 0$ such that for $s \geq t_\ep$, we have 
\eqbn
(1-\ep)c^{-\alpha} \leq \frac{g( cs)}{g(s)} \leq (1+\ep ) c^{-\alpha}    .
\eqen
Hence
\alb
\limsup_{t\rta\infty} \frac{\wh g(c t)}{\wh g(t)} &\leq \limsup_{t\rta\infty} \frac{c \int_{t_\ep}^t g(c s) \, ds}{\int_{t_\ep}^t g(s) \,ds + \int_0^{t_\ep} g(s) \, ds}\\
&= \limsup_{t\rta\infty} \frac{c \int_{t_\ep}^t g(c s) \, ds}{\int_{t_\ep}^t g(s) \,ds  }  
\leq  (1+\ep) c^{1-\alpha} .
\ale
Similarly, $\limsup_{t\rta\infty} \frac{\wh g(c t)}{\wh g(t)} \geq (1-\ep) c^{1-\alpha}$. Since $\ep> 0$ is arbitrary we conclude.
\end{proof}
\end{comment}

\begin{lem} \label{prop-interval-reg-var}
Let $\alpha \in (0,1)$. Let $(Y_j)$ be an iid sequence of non-negative random variables. For $m\in\BB N$, let $S_m := \sum_{j=1}^m Y_j$. For $n\in\BB N$ and $k\in [1,   n]_{\BB Z}$, let $G_n^k$ be the event that there is an $m\in\BB N$ with $S_m \in [n- k , n]_{\BB Z}$. Assume that the functions $t\mapsto \BB P(Y_1 \geq t)$ and $t\mapsto \BB P\left(\text{$\lfloor t \rfloor = S_m$ for some $m\in\BB N$}\right)$ are regularly varying with exponents $\alpha$ and $1-\alpha$, respectively, so that in particular $\BB P\left(\text{$n = S_m$ for some $m\in\BB N$}\right) = \psi(n) n^{-(1-\alpha)}$ for some slowly varying function $\psi$. There is a slowly varying function $\wt \psi$, depending only on the law of $Y_1$, such that
\eqbn
\BB P\left( G_n^k \right) = (1+o_{k/n}(1)) \psi(n) \wt\psi(k ) (k/n)^{1-\alpha} .
\eqen 
\end{lem}
\begin{proof} 
For $n\in\BB N$ and $k\in [1 , n]_{\BB Z}$, let $M_n^k$ be the largest $m\in\BB N$ for which $S_m \in [n- k,n  ]_{\BB Z}$ if such a time exists and $M_n^k := n+1$ otherwise. Then we have
\eqbn
\BB P\left( G_n^k  \right) = \sum_{i=  n-k }^n \BB P\left(S_{M_n^k} = i\right). 
\eqen
For $i\in [n- k ,n ]_{\BB Z}$, the event $\{S_{M_n^k} = i\}$ is the intersection of the event that $i = S_m$ for some $m\in\BB N$ and the event that for this $m$, we have $Y_{m+1} \geq n-i + 1$. By the Markov property, the conditional probability of the latter event given the former is the same as the probability that $Y_1 \geq n-j+1$. Thus
\eqbn
\BB P\left(S_{M_n^k} = i\right) = f(i) g(n-i+1)  
\eqen
where $f(i) :=  \BB P\left(\text{$i = S_m$ for some $m\in\BB N$}\right)$ and $g(i) := \BB P\left(Y_1 \geq i\right)$. 
Since $f(n) = \psi(n) n^{-(1-\alpha)}$, we have
\alb
\BB P\left(G_n\right) &= \sum_{i=n-k}^n f(i) g(n-i+1)  
 = (1+o_{k/n}(1)) f(n) \sum_{i= n-k }^n  g(n-i+1) \\
&= (1+o_{k/n}(1)) \psi(n) n^{-(1-\alpha)} \sum_{j=1}^{k} g(j) .
\ale 
By Lemma~\ref{prop-sum-reg-var}, $t\mapsto \sum_{j=1}^{\lfloor  t  \rfloor} g(j)$ is regularly varying of exponent $- (1-\alpha)$, so there is a slowly varying function $\wt\psi$ such that
\eqbn
\sum_{j=1}^{k} g(j) = \wt \psi(k) k^{1-\alpha} .
\eqen
The statement of the lemma follows.
\end{proof}

\begin{proof}[Proof of Lemma~\ref{prop-P-interval-reg-var}]
Let $P_0 = 0$ and for $m\in\BB N$, let $P_m$ be the $m$th smallest $j\in\BB N$ for which $X(-j,-1)$ contains no orders. Then the increments $P_m - P_{m-1}$ are iid, and by~\cite[Lemma A.8]{gms-burger-cone}, $P_1$ is regularly varying with exponent $1-\mu$. By translation invariance,
\eqbn
\BB P\left(\text{$\exists \, m \in\BB N$ such that $P_m = n$}  \right)    =  \BB P\left(I > n\right) .
\eqen
By Lemma~\ref{prop-I-reg-var}, this quantity is regularly varying in $n$ with exponent $\mu$ and slowly varying correction $\psi_0$.  
The statement of the lemma now follows from Lemma~\ref{prop-interval-reg-var}.
\end{proof}

\subsection{Statement and proof of local estimates with few orders} 
\label{sec-few-order-local-proof}
 
In this subsection we will use the following notation. For $n,k \in\BB N$ and $(h,c) \in \BB N^2$, let $E_{n,k,r}^{h,c} $ be the event that there exists $j\in\BB N$ such that the following is true:
\begin{align}  \label{eqn-corner-local-event}
&j \in [n-k,n]_{\BB Z},\quad \mcl N_{\tc H}\left(X(-j,-1)\right) \in \left[h - r, h\right]_{\BB Z}, \quad   \mcl N_{\tc C}\left(X(-j,-1)\right) \in \left[c - r, c\right]_{\BB Z} , \notag \\
& \mcl N_{\tb H}\left(X(-j , -1)\right) + \mcl N_{\tb F}\left(X(-j,-1)\right) \leq r , \quad \op{and} \quad \mcl N_{\tb C}\left(X(-j , -1)\right) + \mcl N_{\tb F}\left(X(-j,-1)\right) \leq r. 
\end{align}
Let $J_{n,k,r}^{h,c}$ be the minimum of $n+1$ and the smallest $j\in [n-k,n]_{\BB Z}$ for which~\eqref{eqn-corner-local-event} holds. 
Heuristically speaking, if $k$ and $r$ are small then $E_{n,k,r}^{h,c}$ is an approximate version of the event $\mcl E_n^{h,c}$ of Definition~\ref{def-local-event} but with the word read backward rather than forward (we read the word backward to enable us to apply the estimates of Section~\ref{sec-J^H-local}). The goal of this subsection is to estimate the probability of this event in terms of $n,k,r,h,$ and $c$. 

Instead of proving a lower bound for the probability of $E_{n,k,r}^{h,c}$, we will instead prove a lower bound for a smaller event which we now define. 
For $m\in\BB N$, let $J_m^H$ and $L_m^H$ be as in~\eqref{eqn-J^H-def}.  
Let $\wt E_{n,k,r}^{h,c} $ be the event that the following is true: 
\eqb  \label{eqn-corner-local-event-J^H}
J_h^H \in \left[n- k , n   \right]_{\BB Z}, \quad   L_h^H  \in \left[c - r, c\right]_{\BB Z} , \quad \op{and} \quad \mcl N_{\tc C}\left(X(-J_h^H , -1)\right) \leq c .
\eqe  
We note that $L_h^H = \mcl N_{\tc C}\left(X(-J_h^H , -1)\right) - \mcl N_{\tb C}\left(X(-J_h^H , -1)\right)$, so if $\wt E_{n,k,r}^{h,c}$ occurs then 
\eqb \label{eqn-C-orders-small}
\mcl N_{\tb C}\left(X(-J_h^H , -1)\right) \leq r .
\eqe
 Since $X(-J_h^H , -1)$ contains no hamburger orders or flexible orders, it follows that
\eqbn
\wt E_{n,k,r}^{h,c} \subset E_{n,k,r}^{h,c} .
\eqen

The main result of this section is the following proposition, which is an analogue of Propositions~\ref{prop-no-order-lower} and~\ref{prop-no-order-upper} of Section~\ref{sec-no-order-local} with the event $E_{n,k,r}^{h,c}$ or $\wt E_{n,k,r}^{h,c}$ in place of the event $\mcl E_n^{h,c}$. We also include a regularity estimate. 
 
\begin{prop} \label{prop-corner-local} 
Suppose we are in the setting described just above. Let $\psi_0$ be the slowly varying function of Lemma~\ref{prop-I-reg-var} and let $\psi_2$ be the slowly varying function of Lemma~\ref{prop-P-interval-reg-var}. Fix $C>1$. Suppose $n ,k,r,h,c \in\BB N $ with 
\eqb  \label{eqn-corner-local-cond}
k\leq n,\quad C^{-1} k\leq r^2 \leq C k, \quad \op{and} \quad  r  \leq C (h\wedge c)  .
\eqe
\begin{enumerate}
\item For each $n,h,c,r,k \in \BB N$ satisfying~\eqref{eqn-corner-local-cond}, we have
\eqbn
\BB P\left(   E_{n,k,r}^{h,c}   \right) \preceq  \psi_0((h\wedge c)^2)  \psi_2(r^2)   (h\wedge c)^{-2 - 2\mu}  k^{ 1+\mu}
\eqen 
with the implicit constants depending only on $C$. \label{item-corner-upper}
\item There exists $m_* \in\BB N$, depending only on $C$, such that for each $n,h,c,r,k\in \BB N$ satisfying~\eqref{eqn-corner-local-cond} with $n,h,c,r,k \geq m_*$ and $(h,c) \in \left[C^{-1} n^{1/2} , C n^{1/2}\right]_{\BB Z}^2$, we have
\eqbn
\BB P\left( \wt E_{n,k,r}^{h,c} \right) \succeq  \psi_0(n) \psi_2(k) n^{-1- \mu}  k^{1+\mu} 
\eqen
with the implicit constants depending only on $C$. \label{item-corner-lower}
\item For each $q\in (0,1)$, there exists $A >0$ and $m_* \in\BB N$ (depending only on $C$ and $q$) such that for each $n,h,c,r,k\in \BB N$ satisfying~\eqref{eqn-corner-local-cond} with $n,h,c,r,k \geq m_*$ and $(h,c) \in \left[C^{-1} n^{1/2} , C n^{1/2}\right]_{\BB Z}^2$, we have
\eqbn
\BB P\left(  \sup_{j \in \left[1, J_{n,k,r}^{h,c} \right]_{\BB Z}} |X(-j,-1)| \leq A n^{1/2}   \,|\,  E_{n,k,r}^{h,c}   \right) \geq 1-q. 
\eqen
\label{item-corner-reg}
\end{enumerate}
\end{prop}

\begin{remark}
Only assertion~\ref{item-corner-upper} of Proposition~\ref{prop-corner-local} is needed for the proof of Theorem~\ref{thm-empty-prob}. However, the other assertions do not take much additional effort to prove and will be used in~\cite{gms-burger-finite}. 
\end{remark}
 
We will extract Proposition~\ref{prop-corner-local} from the following lemma, which in turn is a consequence of the results of Section~\ref{sec-no-order-local} together with Lemma~\ref{prop-P-interval-reg-var}. 

\begin{lem} \label{prop-interval-local}
For $n,k\in\BB N$, let $\ol P_{n,k}   $ be the largest $j\in  \left[n- k , n    \right]_{\BB Z}$ for which $X(-j,-1)$ contains no orders (or $\ol P_{n,k}  = 0$ if no such $j$ exists).
For $n,k,r \in \BB N$ and $(h,c) \in \BB N^2$, let $\ol{E}_{n,k,r}^{h,c} $ be the event that $\ol P_{n,k} > 0$ and
\eqbn
\left|\mcl N_{\tc H}\left(X(-\ol P_{n,k} , -1)\right)  - h \right|\leq r \quad \op{and}\quad \left|\mcl N_{\tc C}\left(X(-\ol P_{n,k}, -1)\right)  - c\right|\leq r .
\eqen
Let $\psi_0$ be the slowly varying function of Lemma~\ref{prop-I-reg-var} and let $\psi_2$ be the slowly varying function of Lemma~\ref{prop-P-interval-reg-var}. Fix $C>1$. Suppose $n\in\BB N$, $(h,c) \in\BB N^2$ with $h , c \geq n^{\xi/2}$, $k\in [1 ,   n /2]_{\BB Z}$, and $r    \in [1,  C (h\wedge c) ]_{\BB Z}$.
\begin{enumerate}
\item For each $n,h,c,r,k$ as above, we have
\eqbn
\BB P\left(   \ol{E}_{n,k,r}^{h,c}   \right) \preceq    \psi_0((h\wedge c)^2)  \psi_2(k)   (h\wedge c)^{-2 - 2\mu}  r^2 k^\mu  
\eqen 
with the implicit constants depending only on $C$. \label{item-interval-upper}
\item There exists $n_* \in \BB N$ (depending only on $C$) such that for each $n\geq n_*$ and each $n,h,c,r,k$ as above with $(h,c) \in \left[C^{-1} n^{1/2} , C n^{1/2}\right]_{\BB Z}^2$, we have
\eqbn
\BB P\left( \ol{E}_{n,k,r}^{h,c} \right) \succeq   \psi_0(n) \psi_2(k) n^{-1- \mu}  r^2 k^\mu 
\eqen
with the implicit constants depending only on $C$. \label{item-interval-lower}
\item For each $q\in (0,1)$, there exists $A >0$ and $n_*\in\BB N$ (depending only on $C$ and $q$) such that for each $n,h,c,r,k$ as above with $n\geq n_*$ and $(h,c) \in \left[C^{-1} n^{1/2} , C n^{1/2}\right]_{\BB Z}^2$, we have
\eqbn
\BB P\left(  \sup_{j \in [1,\ol P_{n,k}]_{\BB Z}} |X(-j,-1)| \leq A n^{1/2}   \,|\,  \ol{E}_{n,k,r}^{h,c}   \right) \geq 1-q. 
\eqen
\label{item-interval-reg}
\end{enumerate}
\end{lem}

The reader should think of the event of Lemma~\ref{prop-interval-local} as an approximation to the events of Proposition~\ref{prop-corner-local}, which concerns a particular time (namely $\ol P_{n,k}  $) which is easier to analyze. We will eventually deduce Proposition~\ref{prop-corner-local} from Lemma~\ref{prop-interval-local} by comparing the events considered in the two results.

\begin{proof}[Proof of Lemma~\ref{prop-interval-local}]
Observe that 
\[
\{\ol P_{n,k} \not= 0 \} = \{\text{$\exists \, j\in \left[n- k , n  \right]_{\BB Z}$ such that $X(-j,-1)$ contains no orders}\} .
\]
By Lemma~\ref{prop-P-interval-reg-var}, 
\eqb \label{eqn-olP-prob}
\BB P\left(\ol P_{n,k} > 0\right) \asymp    \psi_0(n) \psi_2(k) (k/n)^\mu
\eqe 
with $\psi_0$ as in Lemma~\ref{prop-I-reg-var}, $\psi_2$ as in Lemma~\ref{prop-P-interval-reg-var}, and the implicit constants depending only on $p$. 

For $m\in [n- k , n  ]_{\BB Z}$, the event $\{\ol P_{n,k} = m\}$ is the same as the event that $X(-m,-1)$ contains no orders and there is no $j\in [m+1,n]$ for which $X(-j,-m-1)$ contains no orders. Hence the conditional law of $X_{-m} \dots X_{-1}$ given $\{\ol P_{n,k} = m\}$ is the same as its conditional law given that $X(-m,-1)$ contains no orders. 
By Lemma~\ref{prop-I-reg-var} and Proposition~\ref{prop-no-order-upper}, it follows that for each $m\in [n-k , n ]_{\BB Z}$,
\eqb \label{eqn-olP-preceq}
\BB P\left( \mcl E_m^{h,c} \,|\,   \ol P_{n,k} = m\right)    \preceq    \frac{\psi_0((h\wedge c)^2)   (h\wedge c)^{-2 - 2\mu}  }{\psi_0(n) n^{-\mu} }  
\eqe 
with the implicit constant depending only on $p$. By Proposition~\ref{prop-no-order-lower}, for each $C>1$ there exists $n_*\in\BB N$ such that for each $n\geq n_*$ and each $(h,c) \in \left[(2C)^{-1} n^{1/2} , 2C n^{1/2}\right]_{\BB Z}^2$, 
\eqb \label{eqn-olP-succeq}
\BB P\left(\mcl E_m^{h,c} \,|\,   \ol P_{n,k} = m \right)    \succeq  n^{ - 1}
\eqe 
with the implicit constant depending only on $C$. By Theorem~\ref{thm-local-conv}, for each $C>1$ and $q\in (0,1)$, there exists $A >0$ and $n_*\in\BB N$ (depending only on $C$ and $q$) such that for each $n\in\BB N$, each $m\in [n-k , n ]_{\BB Z}$, and each $(h,c) \in \left[(2C)^{-1} n^{1/2} , 2C n^{1/2}\right]_{\BB Z}^2$, we have
\eqb \label{eqn-olP-reg}
\BB P\left(    \sup_{j \in [1, m]_{\BB Z}} |X(-j,-1)| \leq A n^{1/2}     \,|\,   \mcl E_m^{h,c} ,\,    \ol P_{n,k} = m\right)    \geq 1-q .
\eqe 

We obtain assertion~\ref{item-interval-upper} by combining~\eqref{eqn-olP-prob} and~\eqref{eqn-olP-preceq} then summing over all $(h' , c') \in [0\vee(h-r), h+ r]_{\BB Z} \times [0\vee (c-r) , c+r]_{\BB Z}$. We similarly obtain assertion~\ref{item-interval-lower} from~\eqref{eqn-olP-prob} and~\eqref{eqn-olP-succeq}. Assertion~\ref{item-interval-reg} is immediate from~\eqref{eqn-olP-reg}. 
\end{proof}

The following lemma will be used to deduce Proposition~\ref{prop-corner-local} from Lemma~\ref{prop-interval-local}.

\begin{lem} \label{prop-corner-to-interval}
For $n,k,r,h,c\in\BB N$, let $ \ol P_{n,k}$ and $  \ol{E}_{n,k,r}^{h,c} $ be defined as in Lemma~\ref{prop-interval-local}. Let $E_{n,k,r}^{h,c}$ be defined as in~\eqref{eqn-corner-local-event} and let $J_{n,k,r}^{h,c}$ be as in the discussion just after. For each $C>1$ and each $q\in (0,1)$, there is an $m_*\in\BB N$ and a $ B \geq C$, depending only on $C$ and $q$, such that for each $n,h,c,r,k \in\BB N$ with $k\leq n$, $C^{-1} k\leq r^2 \leq C k$, and $r  \leq C (h\wedge c) $ and each realization $x$ of $X_{-J_{n,k,r}^{h,c}} \dots X_{-1}$ for which $E_{n,k,r}^{h,c}$ occurs, we have
\eqbn
\BB P\left(  \ol E_{n+ B r^2, B r^2 , B r}^{h,c} \,|\,  X_{-J_{n,k,r}^{h,c}} \dots X_{-1} =x \right) \geq 1-q  .
\eqen 
\end{lem}
\begin{proof}
Given $n,h,c,r,k$, a realization $x$ as in the statement of the lemma, and $B\geq C$, let $F(x;B) = F_{n,k,r}^{h,c}(x;B)$ be the event that there is a $j\in [|x|+k,|x|+B r^2]_{\BB Z}$ such that $X(-j,-|x|-1)$ contains no orders and at least $r$ burgers of each type; and $\sup_{j\in [|x|+1,|x|+B_1 r^2]_{\BB Z}} |X(-j,-|x|-1)| \leq B_2 r$.
Then we have
\eqbn
\{ X_{-J_{n,k,r}^{h,c}} \dots X_{-1} =x\} \cap F(x;B) \subset \ol E_{n+ B r^2, B r^2 , B r}^{h,c} 
\eqen
so we just need to show that for large enough $B$,
\eqb \label{eqn-no-order-next-increment}
\BB P\left(F(x;B) \,|\,X_{-J_{n,k,r}^{h,c}} \dots X_{-1} =x\right) \geq 1-q. 
\eqe   
By \cite[Theorem 2.5]{shef-burger}, we can find $B_0 > C^2  \, (\geq k/r^2)$ such that for each $r >0$, it holds with probability at least $1-q/3$ that the word $X(-|x| - B_0 r^2 , -|x|-1)$ contains at least $r$ burgers of each type. 
By Lemma~\ref{prop-P-reg-var} and the Dynkin-Lamperti theorem~\cite{dynkin-limits,lamperti-limits}, we can find $B_1 > B_0$, depending only on $q$, such that with probability at least $1-q/3$, there is a $j\in [|x| + B_0 r^2, |x| +B_1 r^2]_{\BB Z}$ such that $X(-j , -|x|-1)$ contains no orders. 
By a second application of \cite[Theorem 2.5]{shef-burger}, we can find $B_2  > 0$ such that with probability at least $1-q/3$, we have $\sup_{j\in [|x|+1,|x|+B_1 r^2]_{\BB Z}} |X(-j,-|x|-1)| \leq B_2 r$. Since $J_{n,k,r}^{h,c}$ is a stopping time for $X$, read backward, the word $\dots X_{-|x|-2} X_{-|x|-1}$ is independent from the event $\{ X_{-J_{n,k,r}^{h,c}} \dots X_{-1} =x\}$. By combining these observations, we obtain~\eqref{eqn-no-order-next-increment} with $B = C\vee B_0 \vee B_1 \vee B_2$. 
\end{proof}
 
\begin{proof}[Proof of Proposition~\ref{prop-corner-local}] 
Suppose given $n,k,r,h,c$ as in the statement of the proposition and define the event $\ol E_{n,k,r}^{h,c}$ as in Lemma~\ref{prop-interval-local}. 
By Lemma~\ref{prop-corner-to-interval}, we can find $B > 0$, depending only on $C$, such that  
\eqbn
\BB P\left( \ol E_{n + B r^2, B r^2 , B r}^{h,c}    \,|\,     E_{n,k,r}^{h,c}   \right) \geq \frac12 .
\eqen 
By combining this with assertion~\ref{item-corner-upper} of Lemma~\ref{prop-interval-local}, we obtain assertion~\ref{item-corner-upper} of the present proposition. 
 
From now on we assume further that that $(h,c) \in \left[C^{-1} n^{1/2} , C n^{1/2}\right]_{\BB Z}$. 
To obtain assertion~\ref{item-corner-lower}, let $\ul P_{n,k,r}^{h,c} $ be the smallest $j \in [n-k , n]_{\BB Z}$ for which 
\eqbn
\left|\mcl N_{\tc H}\left(X(-j , -1)\right)  - h \right|\leq r \quad\op{and}\quad \left|\mcl N_{\tc C}\left(X(-j, -1)\right)  - c\right|\leq r ,
\eqen
or $\ul P_{n,k,r}^{h,c} = n+1$ if no such $j$ exists. Observe that $\ul P_{n,k,r}^{h,c}$ is a stopping time for the filtration generated by $X$, read backward; and $\ul P_{n,k,r}^{h,c} \leq n$ on $\ol E_{n,k,r}^{h,c}$. By assertion~\ref{item-interval-lower} of Lemma~\ref{prop-interval-local}, if $n$ is chosen sufficiently large (depending only on $C$) then
\eqb \label{eqn-ulP-prob}
\BB P\left(     \ul P_{n,k,r}^{h,c} \leq n    \right) \succeq   \psi_0(n) \psi_2(k) r^2 k^\mu n^{-1-2\mu} 
\eqe 
with the implicit constant depending only on $C$. Set $\ul P^* :=   \ul P_{n-r^2,r^2 ,r/8}^{h-r/8  ,c-r/8}$.  
We observe that $\ul P^*$ is a stopping time for the word $X$, read forward. 
By the strong Markov property and~\cite[Theorem 2.5]{shef-burger} (and since $C^{-1} r^2 \leq k \leq Cr^2$) we can find $m_* \in\BB N$, depending only on $C$, such that for $n,h,c,k,r $ as above with $n,h,c,k,r \geq m_*$,   
\eqbn
\BB P\left( \wt E_{n,k,r}^{h,c}  \,|\,    \ul P^*  \leq n-r^2  \right) \succeq 1 ,
\eqen
implicit constant depending only on $C$. Here we recall that $\wt E_{n,k,r}^{h,c} $ is defined as in~\eqref{eqn-corner-local-event-J^H}. By~\eqref{eqn-ulP-prob} (applied to $\ul P^*$) we obtain assertion~\ref{item-corner-lower}. 
 
It remains to prove assertion~\ref{item-corner-reg}. For $A>0$, let
\eqbn
G(A) = G_{n,k,r}^{h,c}(  A) :=  E_{n,k,r}^{h,c}  \cap \left\{ \sup_{j \in \left[1,J_{n,k,r}^{h,c}\right]_{\BB Z}} |X(-j,-1)| \leq A n^{1/2} \right\} .
\eqen 
By Lemma~\ref{prop-corner-to-interval}, we can find a $B  > 1$, depending only on $C$, such that for each $n,h,c,k,r$ as above, we have
\eqb \label{eqn-corner-reg1}
\BB P\left(  \ol E    \,|\, G( A)^c \cap E_{n,k,r}^{h,c}  \right) \geq \frac12 ,
\eqe 
where $\ol E := \ol E_{n+Br^2 , Br^2  , B r}^{h,c}$. 
By assertions~\ref{item-interval-upper} and~\ref{item-interval-lower} of Lemma~\ref{prop-interval-local} together with assertions~\ref{item-corner-upper} and~\ref{item-corner-lower} of the present lemma, there is an $\wt m_* \in \BB N$, depending only on $C$, such that for $n,h,c,r,k$ as above with $n,h,c,r,k \geq \wt m_*$, we have
\eqbn
\BB P\left(\ol E \right) \asymp \BB P\left( E_{n,k,r}^{h,c}  \right)
\eqen
with the implicit constant depending only on $C$. 
 
By assertion~\ref{item-interval-reg} of Lemma~\ref{prop-interval-local}, for each $\alpha>0$ we can find $A > 0$ and $m_* \geq \wt m_*$, depending only on $C$ and $\alpha$, such that for $m\geq m_*$, we have
\eqbn
\BB P\left( \sup_{j \in [1, n]_{\BB Z}} |X(-j,-1)| > A n^{1/2}      \,|\, \ol E  \right) \leq \alpha ,
\eqen
which implies
\eqb \label{eqn-corner-reg0} 
\BB P\left( G(A)^c \cap E_{n,k,r}^{h,c}      \,|\, \ol E  \right) \leq \alpha .
\eqe 
By~\eqref{eqn-corner-reg1} and~\eqref{eqn-corner-reg0}, 
\alb
\BB P\left(G(A)^c \,|\, E_{n,k,r}^{h,c}   \right) 
= \frac{  \BB P\left(  G(A)^c \cap E_{n,k,r}^{h,c}    \,|\, \ol E    \right) \BB P\left(\ol E  \right)   }{  \BB P\left(  \ol E    \,|\, G(A)^c \cap E_{n,k,r}^{h,c} \right)   \BB P\left(   E_{n,k,r}^{h,c} \right)   }  
 \preceq \alpha
\ale
with the implicit constant depending only on $C$. We now conclude by choosing $\alpha$ smaller than $q$ divided by this implicit constant. 
\end{proof}

\section{Empty reduced word exponent}
\label{sec-empty-prob}

In this section we will prove Theorem~\ref{thm-empty-prob}. The proof of the lower bound for $\BB P(X(1,2n) = \emptyset)$ is easier, and is given in Section~\ref{sec-empty-lower}. The argument is similar to some of the proofs given in \cite[Section 3]{gms-burger-cone} (and in fact appeared in an earlier version of that paper). In Section~\ref{sec-empty-estimates}, we will prove some estimates which are needed for the proof of the upper bound, namely a variant of \cite[Lemma 3.7]{gms-burger-cone} which tells us that a word with few hamburger orders is very unlikely to contain too many flexible orders; and an estimate to the effect that the path $Z^n$ is unlikely to stay close to the boundary of the first quadrant for a long time. In Section~\ref{sec-empty-upper}, we use the estimates of Section~\ref{sec-few-order-local} and Section~\ref{sec-empty-estimates} to prove the upper bound in Theorem~\ref{thm-empty-prob}.

\subsection{Lower bound} 
\label{sec-empty-lower}

In this subsection we will prove the lower bound in Theorem~\ref{thm-empty-prob}. The content of this subsection appeared in an earlier version of~\cite{gms-burger-cone}, but was moved to its present location to present a more logical progression of ideas. We thank Cheng Mao for his contributions to this subsection.

The argument of this subsection is similar to the proofs in~\cite[Section 3.2]{gms-burger-cone}, and proceeds as follows. 
We prove the lower bound with $X(-2n,-1)$ in place of $X(1,2n)$ (which we can do by translation invariance). 
Let $\delta>0$ be a small parameter which is independent from $n$. We divide the interval $[-2n,-1]_{\BB Z}$ into the interval $[-2n,-n-1]_{\BB Z}$ plus $\BB k_n = \left\lceil \frac{\log n}{\log \delta} \right\rceil$ intervals of the form $[-\delta^{k-1} n , -\delta^k n]_{\BB Z}$. We will define events $ G_{n,k} $ corresponding to these intervals such that the intersection of the $G_{n,k}$'s over all of the intervals is contained in $\{X(1,2n) =\emptyset\}$. By~\cite[Theorem 2.5]{shef-burger}, one can approximate the path $D = (d, d^*)$ of Notation~\ref{def-theta-count} on all but finitely many of the intervals $[-\delta^{k-1} n , -\delta^k n]_{\BB Z}$ (i.e., those for which $\delta^k n$ is at least some $\delta$-dependent constant) by a correlated two-dimensional Brownian motion as in~\eqref{eqn-bm-cov}.  
This will lead to a lower bound for the probability of the intersection of the events $G_{n,k}$ for all but finitely many values of $k$. The finitely many small intervals remaining do not pose a problem since there are only finitely many possibilities for the symbols in the word restricted to these intervals. 

To carry out the above argument we first need to estimate an appropriate probability for Brownian motion. 
 
\begin{lem} \label{prop-bm-point-asymp}
Fix a constant $C>1$. Let $z \in [C^{-1} , C]^2$. Let $Z$ be a Brownian motion as in~\eqref{eqn-bm-cov} (started from 0). For $\delta > 0$, let $F_\delta^z$ be the event that $U(t) \geq - \delta^{1/2}$ and $V(t) \geq -\delta^{1/2}$ for each $t\in [0,1]$; and $|Z(1) - z| \leq \delta^{1/2} $. Then $\BB P\left(F_\delta^z\right) \succeq \delta^{ \mu+1}$, where $\mu$ is as in~\eqref{eqn-cone-exponent} and the implicit constant depends on $C$, but not $z$ or $\delta$. 
\end{lem}
\begin{proof}
Let $\wt E_\delta$ be the event that $U(t) \geq -\delta^{1/2}$ and $V(t) \geq -\delta^{1/2}$ for each $t\in [0,1/2]$; and $Z(1/2) \in [C^{-1} , C]^2$. 
By \cite[Lemma~3.2]{gms-burger-cone} and scale invariance (see also \cite[Section 4]{shimura-cone}), we have $\BB P\left(\wt E_\delta\right) \asymp \delta^\mu$. The conditional law of $Z|_{[1/2,1]}$ given $Z|_{[0,1/2]}$ is that of a Brownian motion with covariances as in~\eqref{eqn-bm-cov} started from $Z(1/2)$. On the event $\wt E_\delta$, the probability that such a Brownian motion stays in the first quadrant until time 1 and satisfies $|Z(1) - z| \leq \delta^{1/2}$ is proportional to $\delta$. Hence $\BB P\left(F_\delta^z \,|\, \wt E_\delta\right) \asymp \delta$. The statement of the lemma follows.
\end{proof}

\begin{proof}[Proof of Theorem~\ref{thm-empty-prob}, lower bound]
Fix $\delta>0$ and $C>100$, to be chosen later. Let $\BB k_n = \left\lceil \frac{\log n}{\log \delta^{-1} } \right\rceil$ be the smallest $k\in\BB N$ such that $\delta^k n \leq 1$.
Also fix $\nu\in (1-\mu ,1/2)$. 

Let $b_0^H$ (resp. $b_0^C$) be the number of hamburgers (resp. cheeseburgers) in $X(-2n , -n-1)$.
For  $k\in [0,  \BB k_n ]_{\BB Z}$ let 
\eqbn
b_k^H := \mcl N_{\tc H}\left(X(-2n,-\lfloor \delta^k n \rfloor-1) \right) \quad \op{and} \quad
b_k^C := \mcl N_{\tc C}\left(X(-2n,-\lfloor \delta^k n \rfloor-1) \right)  .
\eqen 
Let $G_{n,0} $ be the event that $X(-2n,-n-1)$ contains no orders, at least $7 C^{-1} n^{1/2}$ hamburgers, at least $7 C^{-1} n^{1/2}$ cheeseburgers, and at most $C n^{1/2}$ total burgers.  
For $k\in [1,   \BB k_n]_{\BB Z}$, let $G_{n,k}$ be the event that the following is true.
\begin{enumerate}
\item $\mcl N_\theta \left( X(  -  \lfloor \delta^{k-1} n \rfloor ,- \lfloor \delta^k n \rfloor-1 ) \right) \leq  0\vee ( C^{-1} (\delta^k n)^{1/2} - (\delta n)^{\nu (k-1) } -1 )$ for $\theta\in\{\tc H,\tc C\}$. \label{item-empty-order-hc}
\item $b_{k-1}^H - 4 C^{-1} (\delta^k n )^{1/2}   - (\delta n)^{\nu (k-1) } \leq \mcl N_{\theta}\left( X( -  \lfloor \delta^{k-1} n \rfloor ,- \lfloor \delta^k n \rfloor-1 ) \right) \leq   b_{k-1}^H - 3 C^{-1} (\delta^k n )^{1/2}   -  (\delta n)^{\nu (k-1) }$ for $\theta \in \{\tb H , \tb C\}$. \label{item-empty-lower-b}
\item $\mcl N_{\tb F}\left( X(-  \lfloor \delta^{k-1} n \rfloor , -\lfloor \delta^k n \rfloor-1 ) \right) \leq  (\delta n)^{\nu (k-1) }$. \label{item-empty-lower-F}
\end{enumerate} 
Roughly speaking, $G_{n,k}$ is the event that $X(  -   \lfloor \delta^{k-1} n \rfloor , \lfloor \delta^k n \rfloor-1 ) $ has at most $\approx C^{-1}(\delta^k n)^{1/2}$ burgers of each type; the orders in this word fulfill all but $\asymp (\delta^k n)^{1/2}$ of the burgers in $X(-2n, -\lfloor \delta^{k-1} n \rfloor -1)$; and $X(  -   \lfloor \delta^{k-1} n \rfloor , \lfloor \delta^k n \rfloor-1 ) $ does not have too many leftover $\tb F$'s. 

By conditions~\ref{item-empty-lower-b} and~\ref{item-empty-lower-F}, if $k\in [1,\BB k_n]_{\BB Z}$ and $ G_{n,k}$ occurs then every order in $X(    -  \lfloor \delta^{k-1} n \rfloor , \lfloor \delta^k n \rfloor-1 ) $ has a match in $X(-2n,-\lfloor \delta^{k-1} n \rfloor)$. By induction, on $\bigcap_{j=0}^k G_{n,j}$ the word $X(-2n, -\lfloor \delta^k n \rfloor-1)$ contains no orders. Furthermore, using condition~\ref{item-empty-lower-b} in the definition of $G_{n,k}$, we find that if $\lfloor \delta^k n \rfloor$ is at least some $(\delta,C)$-dependant constant, then
\eqb \label{eqn-remaining-burgers}
C^{-1}  (\delta^k n )^{1/2}     \leq b_k^H  \leq 6C^{-1}  (\delta^k n )^{1/2}  \quad \op{and} \quad C^{-1}  (\delta^k n )^{1/2}     \leq b_k^C  \leq 6C^{-1}  (\delta^k n )^{1/2}  .
\eqe  

We will now estimate the conditional probability of $G_{n,k}$ given $X_{-2n} \dots X_{-\lfloor \delta^{k-1} n \rfloor -1} $ using the scaling limit result~\cite[Theorem 2.5]{shef-burger} and Lemma~\ref{prop-bm-point-asymp}. 
By \cite[Lemma 3.3]{gms-burger-cone}, if $C$ is chosen sufficiently large then we have $\BB P(G_{n,0}) \geq n^{-\mu + o_n(1)}$. 
With $D = (d,d^*)$ the walk as in~\eqref{eqn-discrete-path}, we have
\alb
\mcl N_{\tc H} \left( X(  - \lfloor \delta^{k-1} n \rfloor , -\lfloor \delta^{k } n \rfloor -1 ) \right) & =   -\inf_{j \in[ \lfloor \delta^{k-1} n \rfloor , -\lfloor \delta^{k } n \rfloor -1]_{\BB Z}} (d(-j)  - d(-\lfloor \delta^{k } n \rfloor -1))  \quad \op{and}\\
\mcl N_{\tb H}\left( X(  - \lfloor \delta^{k-1} n \rfloor , -\lfloor \delta^{k } n \rfloor -1 ) \right) &=  d(-\lfloor \delta^{k-1} n \rfloor) -  \inf_{j \in[ \lfloor \delta^{k-1} n \rfloor , -\lfloor \delta^{k } n \rfloor  -1]_{\BB Z}} d(-j) + \xi_n   ,
\ale 
where $\xi_n$ denotes an error term which is bounded above in absolute value by one plus the number of $\tb F$'s in $X(  - \lfloor \delta^{k-1} n \rfloor , -\lfloor \delta^{k } n \rfloor -1 )$.
Similar formulas hold for cheeseburgers and cheeseburger orders with $d^*$ in place of $d$. 
By \cite[Theorem 2.5]{shef-burger} and translation invariance, the restriction to $[0,1]$ of
\eqbn
t\mapsto  \left( \delta^{(k-1)/2} - \delta^{k/2} \right)^{-1/2}  n^{-1/2} \left( D( - \lfloor t (  \lfloor \delta^{k-1} n \rfloor  - \lfloor \delta^{k } n \rfloor ) + \lfloor \delta^k n \rfloor  \rfloor ) - D(- \lfloor \delta^{k } n \rfloor  -1)  \right)
\eqen
converges in law to the restriction to $[0 ,1]$ of a correlated Brownian motion $Z$ as in~\eqref{eqn-bm-cov}. 
By the strong Markov property, the same holds if we condition on $X_{-2n} \dots X_{-\lfloor \delta^{k-1} n \rfloor -1}$. 
By combining these observations with \cite[Corollary 6.2]{gms-burger-cone} (to deal with condition~\ref{item-empty-lower-F}) and Lemma~\ref{prop-bm-point-asymp}, we see that there exists $m_* \in \BB N$, independent of $n$, and a constant $q   \in(0,1)$, independent of $n$ and $\delta$, such that whenever $k\in [1,   \BB k_n]_{\BB Z}$ with $\lfloor \delta^k n \rfloor \geq m_*$,~\eqref{eqn-remaining-burgers} holds on the event $G_{n,k}$ and
\[
\BB P\left(G_{n,k} \,|\,  X_{-2n} \dots X_{-\lfloor \delta^{k-1} n \rfloor -1}  \right) \BB 1_{G_{n,k-1}} \geq q \delta^{\mu+1} \BB 1_{G_{n,k-1}} .
\]
Let $k_*$ be the largest $k\in [1,    \BB k_n]_{\BB Z}$ for which $\lfloor \delta^k n \rfloor \geq m_*$. Then
\eqb \label{eqn-P-big-events}
\BB P\left(\bigcap_{k=0}^{k_*} G_{n,k} \right) \geq q^{ \BB k_n} \delta^{ \BB k_n(\mu+1)} n^{-\mu  + o_n(1)}   \geq n^{-2\mu-1+o_n(1) +o_\delta(1)}
\eqe 
with the $o_\delta(1)$ independent of $n$. Here we used that
\eqbn
\log \left( q^{-\BB k_n} \right) \leq  \frac{ \log q^{-1} \log n}{\log\delta^{-1}} \leq  o_\delta(1) \log n .
\eqen
It follows from~\eqref{eqn-remaining-burgers} that on $\bigcap_{k=0}^{k_*} G_{n,k}$, the word $X(-2n,-\lfloor \delta^{k_*} n \rfloor -1)$ contains no orders and fewer than $\lfloor \delta^{k_*} n \rfloor$ burgers. Since $m_*$ does not depend on $n$ and $\lfloor \delta^{k_*} n \rfloor \leq \delta^{-1} m_*$, we have
\eqb \label{eqn-P-reduced|big-events} 
\BB P\left(  X(1,2n) = \emptyset   \,|\, \bigcap_{k=0}^{k_*} G_{n,k}   \right) \succeq 1, 
\eqe 
with the implicit constant independent of $n$. 
By combining~\eqref{eqn-P-big-events} and~\eqref{eqn-P-reduced|big-events} and sending $\delta\rta 0$, we obtain the lower  bound in Theorem~\ref{thm-empty-prob}. 
\end{proof}

\subsection{Some preliminary estimates}
\label{sec-empty-estimates}

In this subsection we will prove some estimates which are needed to rule out pathological behavior in the proof of the upper bound in Theorem~\ref{thm-empty-prob}. 
The reader may wish to skim this section and refer back to it while reading Section~\ref{sec-empty-upper}. 

First we prove a variant of \cite[Lemma 3.7]{gms-burger-cone} where we consider the number of hamburger orders in a word, rather than the length of the word.

\begin{lem} \label{prop-few-F-H}
Let $\mu'$ be as in~\eqref{eqn-cone-exponent}. For $ m\in\BB N$ and $\nu >  \mu'$, we have
\eqb \label{eqn-few-F-H}
\BB P\left(\text{$\exists \, n\in\BB N$ such that $\mcl N_{\tb H}\left(X(1,n)\right) \geq m $ and $ \mcl N_{\tb F}\left(X(1,n)\right) \geq  \mcl N_{\tb H}\left(X(1,n)\right)^{2\nu} $}\right)  = o_m^\infty(m) .
\eqe 
Furthermore, for each fixed $\eta >0$ and each $n\in\BB N$, we have
\eqb \label{eqn-few-F-H'}
\BB P\left(\text{$\exists \, j \in [1,n]_{\BB Z}$ such that $ \mcl N_{\tb F}\left(X(-j,-1)\right) \geq  \mcl N_{\tb H}\left(X(-j,-1)\right)^{2 \nu} \vee n^{\eta  }$}\right)  = o_n^\infty(n ) .
\eqe 
\end{lem}
 
To prove Lemma~\ref{prop-few-F-H}, we first require the following further lemma. 

\begin{lem} \label{prop-few-F-orders}
Let $\mu'$ be as in~\eqref{eqn-cone-exponent}. For $m\in\BB N$, let $\wt I_m^H$ be the $m$th smallest $i\in\BB N$ for which the word $X(1,i)$ contains no hamburgers. Then for each $\nu > \mu'$, we have
\eqbn
\BB P\left(\text{$\exists \, k \geq m$ such that $ \mcl N_{\tb F}\left(X(1, \wt I_k^H)\right) \geq k^{2\nu}$} \right) = o_m^\infty(m) .
\eqen
\end{lem}
\begin{proof} 
Let $N_0 = 0$ and for $l\in\BB N$, let $N_l$ be the $l$th smallest $m\in\BB N$ such that $X(1,\wt I_m^H)$ contains no cheeseburgers (equivalently no burgers). Note that the increments $\{N_l - N_{l-1}\}_{l\in\BB N}$ are iid. Let $M_m := \sup\{l\in\BB N\,:\, N_l \leq m\}$. By \cite[Lemma 3.7]{gms-burger-cone}, for each sufficiently large $C>1$ we have 
\eqb \label{eqn-big-I_M}
\BB P\left(\wt I_{N_1}^H \geq C^2 m^2 ,\, \mcl N_{\tb H}\left(X(1,C^2m^2)\right) > m \right) \geq m^{-2\mu' + o_m(1)} .
\eqe 
A hamburger order can only be added to the reduced word at one of the times $\wt I_k^H$ for $k\in\BB N$. Hence on the event in~\eqref{eqn-big-I_M}, we have $N_1 \geq m$, so
\eqbn
\BB P\left(N_1 \geq m\right) \geq m^{-2\mu'+o_m(1)} .
\eqen
The quantity $\mcl N_{\tb F}\left(X(1, \wt I_m^H)\right)$ can increase by at most 1 when $m$ increases by 1, and can increase only if $m = N_l$ for some $l\in\BB N$. The statement of the lemma therefore follows from \cite[Lemma 3.10]{gms-burger-cone}.
\end{proof}

\begin{proof}[Proof of Lemma~\ref{prop-few-F-H}]
Define the times $\wt I_m^H$ as in Lemma~\ref{prop-few-F-orders}. 
For $m\in\BB N$, let $E_m^H$ be the event that $\wt I_m^H = \wt I_{m-1}^H + 1$ and $X_{\wt I_m^H} = \tb H$. The events $E_m^H$ are independent and by~\eqref{eqn-theta-prob}, we have $\BB P\left(E_m^H\right) = \BB P\left( X_{\wt I_{m-1}^H + 1} = \tb H  \right) =  (1-p)/4$ for each $m\in\BB N$. By Hoeffding's inequality, 
\eqb \label{eqn-I^H-hoeff}
\BB P\left(\mcl N_{\tb H}\left(X(1,\wt I_{m}^H)\right) \leq \frac{1-p}{8} m \right)  = o_m^\infty(m) .
\eqe 
By summing this estimate over $[m,\infty)_{\BB Z}$, we find that the probability that there exists even a single $k \geq m$ for which $\mcl N_{\tb H}\left(X(1,\wt I_{k}^H)\right) \leq \frac{1-p}{8} k$ is of order $o_m^\infty(m)$.
 
Suppose there exists $n\in\BB N$ such that $h_n := \mcl N_{\tb H}\left(X(1,n)\right)   \geq m$ and $ \mcl N_{\tb F}\left(X(1,n)\right) \geq h_n^{2\nu}$. 
For each such $n$, let $m_n$ be the largest $k\in\BB N$ such that $\wt I_k^H \leq n$. We clearly have $m_n \geq h_n$, so by the above estimate it holds except on an event of probability $o_m^\infty(m)$ that $m_n \in [h_n , \frac{8}{1-p} h_n]_{\BB Z}$ for each such $n$. Therefore, for each such $n$ we have $m_n \geq m$ and $\mcl N_{\tb F}\left(X(1,\wt I_{m_n}^H)\right) \geq \left( \frac{1-p}{8} m_n\right)^{2\nu}$. We now obtain~\eqref{eqn-few-F-H} by means of Lemma~\ref{prop-few-F-orders}. 

To obtain~\eqref{eqn-few-F-H'}, we use translation invariance and~\eqref{eqn-few-F-H} with $m = \left\lfloor \frac{1-p}{8} n^{\eta } \right\rfloor$ to find that for each $j \in [1,n]_{\BB Z}$, the probability that $\mcl N_{\tb H}\left(X(-j,-1)\right) \geq \left\lfloor \frac{1-p}{8} n^{\eta } \right\rfloor$ and $ \mcl N_{\tb F}\left(X(-j,-1)\right) \geq  \mcl N_{\tb H}\left(X(-j,-1)\right)^{2\nu}$ is $o_n^\infty(n)$. By~\eqref{eqn-I^H-hoeff}, the probability that $\mcl N_{\tb F}\left(X(-j,-1) \right) \geq n^{\eta }$ and $\mcl N_{\tb H}\left(X(-j,-1)\right) < \left\lfloor \frac{1-p}{8} n^{\eta} \right\rfloor $ is also of order $o_n^\infty(n)$. We then sum over all $j\in [1,n]_{\BB Z}$. 
\end{proof}

Next we will prove a result to the effect that it is very unlikely that the path $n\mapsto D(n)$ of Section~\ref{sec-burger-prelim} stays close to the coordinate axes for a long time. 

\begin{lem} \label{prop-no-tube}
Define $D = (d,d^*)$ as in Notation~\ref{def-theta-count}. There is a constant $a_0 > 0$, depending only on $p$, such that for each $n\in\BB N$ and $r  > 0$, we have 
\eqb \label{eqn-no-tube-time}
\BB P\left(\sup_{j\in [1,n]_{\BB Z}} \left|    d\left(X(-j,-1)\right) \right| \wedge \left|    d^*\left(X(-j,-1)\right) \right|  \leq r  \right) \leq  e^{-a_0 r^{-2} n }  + o_n^\infty(n)  
\eqe  
with the rate of the $o_n^\infty(n)$ depending only on $p$. 
\end{lem}

Note that the right side of~\eqref{eqn-no-tube-time} is of order $o_n^\infty(n)$ provided $r \leq n^{1/2-\ep}$ for any $\ep > 0$; and is of constant order provided $r \geq n^{1/2}$. 

\begin{proof}[Proof of Lemma~\ref{prop-no-tube}]
The idea of the proof is to use the scaling limit result~\cite[Theorem 2.5]{shef-burger} for $D$ to argue that in each $r^2$-length interval of times, $D$ has positive conditional probability to exit the $r$-neighborhood of the coordinate axes, regardless of what has happened in the past; then multiply over $\asymp r^{-2} n$ such intervals. 
There is a small amount of additional complication arising from the fact that $D$ is not exactly Markovian, due to the presence of $\tb F$'s. 
 
Let
\eqbn
G_n(r):= \left\{  \sup_{j\in [1,n]_{\BB Z}}\left|    d\left(X(-j,-1)\right) \right| \wedge \left|    d^*\left(X(-j,-1)\right) \right|  \leq r \right\}  
\eqen
be the event in the statement of the lemma.
Let $\mu $ be as in~\eqref{eqn-cone-exponent} and fix $\nu_1 , \nu_2 \in (1-\mu ,1/2)$ with $\nu_1 < \nu_2$.
For $n\in\BB N$, let  
\eqbn
F_n :=\left\{ \sup_{j\in [n^{1-2\nu_2} , \infty)} j^{-\nu_1} \mcl N_{\tb F}\left(X(-j,-1)\right) \leq 1 \right\} .
\eqen
By \cite[Corollary 6.2]{gms-burger-cone}, $\BB P(F_n) = 1-o_n^\infty(n)$. Hence is suffices to bound $\BB P\left(G_n(r) \cap F_n\right)$. 

For $n\in\BB N$, $r > 0$, and $k\in \left[1,  r^{-2} n   \right]_{\BB Z}$, let 
\eqbn
E_n^k(r) := \left\{ \sup_{j\in [(k-1)r^2+1 , k r^2 ]_{\BB Z}} \left( |d\left(X(-j,-(k-1) r^2 - 1)\right) | \wedge  |d^*\left(X(-j,-(k-1) r^2 - 1)\right) | \right) \leq 3r+1 \right\}    
\eqen
be the event that (roughly speaking) the walk $D = (d,d^*)$ stays close to the coordinate axes during the time interval $[(k-1)r^2+1 , k r^2 ]_{\BB Z}$. 

We claim that if $r \geq n^{ \nu_1}$, then
\eqb \label{eqn-sup-and-F-contain}
G_n(r ) \cap F_n \subset \bigcap_{k=\lfloor n^{1-2\nu_2} \rfloor+1}^{\lfloor r^{-2} n \rfloor}  E_n^k(r) .
\eqe 
Indeed, suppose by way of contradiction that $G_n(r ) \cap F_n $ occurs but there exists  $k_* \in [n^{1-2\nu_2}+1 , r^{-2} n  ]_{\BB Z}$ such that $E_n^{k_*}(r)$ fails to occur. Then there exists $j \in  [(k_*-1) r^2 , k_* r^2  ]_{\BB Z}$ for which $|d\left( X(-j,-(k_*-1)r^2 - 1)\right)| \geq 3 r + 2$ and $|d^*\left( X(-j,-(k_*-1)r^2 - 1)\right)| \geq 3 r + 2$. By definition of $G_n(r)$, either $\left|d\left( X( -(k_*-1) r^2  , -1 )\right)\right| \leq r$ or $\left|d^*\left( X( -(k_*-1) r^2  , -1 )\right)\right| \leq r$. By symmetry we can assume without loss of generality that we are in the former case. 
Then
\alb
&\left| d\left(X(-j,-1)\right) \right| \\
&\qquad \geq \left|d\left( X(-j,-(k_*-1) r^2 - 1)\right)\right| -  \left|d\left( X( -(k_*-1) r^2  , -1 )\right)\right| - \mcl N_{\tb F}\left(X( -(k_*-1) r^2  , -1 )\right)   \\
&\qquad \geq 2r + 1 - \mcl N_{\tb F}\left(X( -(k-1) r^2  , -1 )\right) .
\ale
On $F_n$, we have 
\eqbn
\mcl N_{\tb F}\left(X( -(k-1) r^2  , -1 )\right)
 \leq (k-1)^{ \nu_1 }  r^{2\nu_1}
\leq n^{\nu_1} 
\leq  r  . 
\eqen
Hence $\left|d\left(X(-j,-1)\right) \right| \geq r+1$, which contradicts occurrence of $G_n(r)$. Thus~\eqref{eqn-sup-and-F-contain} holds.

The events $E_n^k(r)$ for different values of $k$ are independent. By \cite[Theorem 2.5]{shef-burger}, there is a $q > 0$, independent of $n$ and $r$, such that $\BB P\left(E_n^k(r)^c \right) \geq q$ for each $n\in\BB N$, each $r \geq n^{\nu_1}$, and each $k\in [n^{1-2\nu_2} , r^{-2} n ]_{\BB Z}$. Hence it follows from~\eqref{eqn-sup-and-F-contain} that for $r \geq n^{\nu_1  }$, we have
\eqbn
\BB P\left(G_n(r ) \cap F_n \right) \leq (1-q)^{r^{-2} n - n^{1-2\nu_2} -1}  \leq e^{-a_0 r^{-2} n}
\eqen
for an appropriate choice of $a$ as in the statement of the lemma.
It is clear that $\BB P\left(G_n(r) \cap F_n\right)$ is increasing in $r$, so for a general choice of $n\in\BB N$ and $r>0$, we have
\eqbn
\BB P\left(G_n( r ) \cap F_n \right) \leq  e^{-a_0 r^{-2} n} \vee e^{-a_0   n^{1-2\nu_1} }    \leq  e^{-a_0 r^{-2} n} +  o_n^\infty(n) .
\eqen
This yields~\eqref{eqn-no-tube-time}. 
\end{proof}

\subsection{Upper bound}
\label{sec-empty-upper}

We are now ready to prove the upper bound of Theorem~\ref{thm-empty-prob}, and thereby complete the proof of the theorem.
We will read the word backward and prove an estimate for $\BB P\left( X(-2n,-1) = \emptyset \right)$ (which is sufficient by translation invariance) so as to make use of Proposition~\ref{prop-corner-local}. 
 
For the proof, we will define successive stopping times $N_n^k$ and $K_n^k$ and associated events $E_n^k$ for $k\in [1,\BB k]_{\BB Z}$, where $\BB k \in \BB N$ is large but independent of $n$.
These stopping times are defined so that if $E_n^k$ occurs and $\{X(-2n,-1) = \emptyset\}$, then a certain reduced word will have an unusually small number of orders, so Proposition~\ref{prop-corner-local} gives a polynomial upper bound for the probability that this is the case. 
On the other hand, we will deduce from the exponential tail for the length of a reduced word~\cite[Lemma 3.13]{shef-burger} and the estimates of Section~\ref{sec-empty-estimates} that $\BB P\left((E_n^k)^c ,\, X(1,2n) = \emptyset \right) = o_n^\infty(n)$, so that we obtain an upper bound for $\BB P\left( X(-2n,-1) =\emptyset  \right)$. 
See Figure~\ref{fig-empty-upper} for an illustration of the proof. 
We now proceed with the details.

\begin{figure}[ht!]
 \begin{center}
\includegraphics{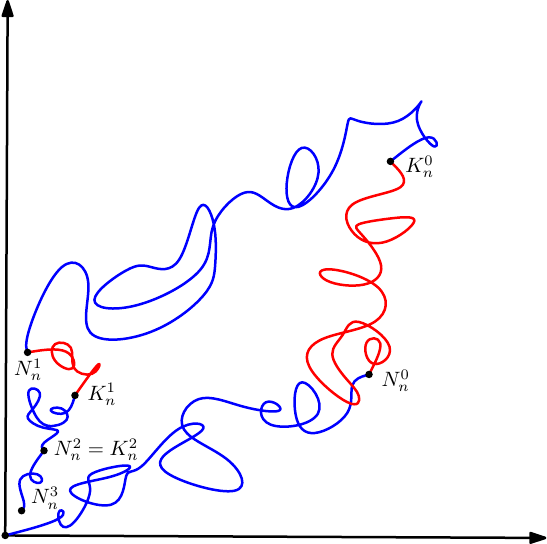} 
\caption{An illustration of the setup for the proof of the upper bound in Theorem~\ref{thm-empty-prob}. Fix $\nu \in (\mu',1/2)$, $\xi \in (2\nu,1)$,  and a small $\zeta>0$. We first grow the word $X$ backward up to a stopping time $K_n^0$ between $N_n^0 = \lfloor n/2 \rfloor$ and $n$ such that it holds on event of probability $1-o_n^\infty(n)$ that the word $X(-K_n^0,-1)$ contains at least $n^{1/2-\zeta}$ hamburger orders and at least $n^{1/2-\zeta}$ cheeseburger orders. Then we iterate the following procedure for $k\in\BB N$: grow until the first time $N_n^k$ after $K_n^{k-1}$ at which we have at most $n^{\xi^k+\zeta}$ orders; then, if we do not have at least $n^{\xi^k-\zeta}$ hamburger orders and at least $n^{\xi^k-\zeta}$ cheeseburger orders, grow until the first time $K_n^k$ at which we have $n^{\xi^k-\zeta}$ hamburger orders and at least $n^{\xi^k-\zeta}$ cheeseburger orders. It holds except on an event of probability $o_n^\infty(n)$ that this latter time is finite by Lemma~\ref{prop-no-tube}. Furthermore, using Lemma~\ref{prop-few-F-H}, we obtain that it holds except on an event of probability $o_n^\infty(n)$ that each of the words $X(-K_n^k , -1)$ contains at most $n^{2\nu \xi^k}$ flexible orders. The figure shows three iterations of this procedure. Segments of the path corresponding to words of the form $X_{N_n^k} \dots X_{-K_n^{k-1}}  $ are shown in blue and those corresponding to words of the form $  X_{-K_n^k} \dots X_{-N_n^k-1}$ are shown in red. Note that it is possible that $K_n^k = N_n^k$ (which is the case for $k=2$ here). }
\label{fig-empty-upper}
\end{center}
\end{figure}

Let $\mu'$ be as in~\eqref{eqn-cone-exponent}. Fix $\nu \in (  \mu' , 1/2)$, $\xi \in (2\nu ,1)$, $\BB k\in\BB N$, and $\zeta \in \left(0, \frac{\nu-\mu'}{100} \nu^{\BB k} \right)$. We will eventually send $\zeta \rta 0$ then $\BB k\rta\infty$. 
For $n\in\BB N$, let $N_n^0 := \lfloor n/2 \rfloor$ and let $K_n^0$ be minimum of $n+1$ and the smallest $j  \geq N_n^0$ such $X(-j,-1)$ contains at least $n^{1/2-\zeta}$ hamburger orders and at least $n^{1/2-\zeta}$ cheeseburger orders. For $k\in [1,\BB k]_{\BB Z}$, 
\begin{itemize}
\item Let $N_n^k$ be the smallest $j \geq  2n-   2 n^{2(\xi^k+\zeta)}  $ such that $X(-j,-1)$ contains at most $n^{\xi^k + \zeta}$ orders. 
\item Let $K_n^k$ be the minimum of $\lfloor 2n - n^{2\xi^k } \rfloor$ and the smallest $j\geq N_n^k$ such that $X(-j,-1)$ contains at least $n^{\xi^k - \zeta}$ hamburger orders and at least $n^{\xi^k-\zeta}$ cheeseburger orders. 
\end{itemize}
Observe that the times $N_n^k$ and $K_n^k$ are stopping times for the word $X$, read backward, and
\eqbn
N_n^0 \leq K_n^0 \leq \dots \leq N_n^{\BB k} \leq K_n^{\BB k} .
\eqen 
The reason why we work with words of length approximately $n^{2\xi^k}$ is that (in light of Lemma~\ref{prop-few-F-H}) the number of $\tb F$'s in the word $X(-K_n^{k-1} , -1)$, which is of order at most $n^{2\nu \xi^{k-1} + o_\zeta(1)}$, is typically smaller order than the number of orders in the word $X(-K_n^k , -K_n^{k-1}-1)$, which is of order $n^{\xi^k + o_\zeta(1)}$. This is important because we only have local estimates for the numbers of hamburgers, cheeseburgers, hamburger orders, and cheeseburger orders in our words so we need to ensure that the contribution of the $\tb F$'s will be insignificant. 

We now define an event associated with each time $k$. 
Let $E_n^0$ be the event that $K_n^0 \leq n$, $X(-K_n^0,-1)$ contains no burgers, and $\mcl N_{\tb F}\left(X(-K_n^0 , -1)\right) \leq n^{\nu }$. For $k\in [1,\BB k]_{\BB Z}$, let $E_n^k$ be the event that the following is true. 
\begin{enumerate}
\item $N_n^k \leq 2n -  2n^{2\xi^k}$ and $  K_n^k < \lfloor 2n - n^{2 \xi^k  } \rfloor  $. \label{item-empty-prob-times}
\item $X(-K_n^k , -1)$ contains at most $n^{\xi^k +  2\zeta}$ orders. \label{item-empty-prob-orders}
\item $X(-K_n^k ,-1)$ contains no burgers. \label{item-empty-prob-burgers}
\item $\mcl N_{\tb F}\left(X(-K_n^k , -1)\right) \leq n^{2 \nu \xi^{k  } }$. \label{item-empty-prob-F}
\end{enumerate}
Heuristically speaking, if $X(-2n,-1)=\emptyset$, then $E_n^k$ should typically occur, since the reduced word $X(-2n , -2n  + 2n^{2\xi^k})$ should typically contain approximately $n^{\xi^k}$ hamburger orders and cheeseburger orders (by Brownian scaling) and a word with $n^{\xi^k}$ hamburger orders is very unlikely to contain more than $n^{2\nu \xi^{k }}$ flexible orders (Lemma~\ref{prop-few-F-H}). 
In fact, we have the following lemma.
 
\begin{lem} \label{prop-reduced-empty-event-prob}
For each $k\in [0,\BB k]_{\BB Z}$, we have 
\eqb \label{eqn-reduced-event-prob}
\BB P\left( (E_n^k)^c ,\,  X(-2n,-1) = \emptyset\right) = o_n^\infty(n)   .
\eqe 
\end{lem}
\begin{proof} 
The estimate~\eqref{eqn-reduced-event-prob} in the case where $k=0$ follows from Lemma~\ref{prop-no-tube} together with \cite[Corollary 6.2]{gms-burger-cone}. 
Next suppose $k\in [1,\BB k]_{\BB Z}$. If $N_n^k >  2n-  2n^{2 \xi^k  }   $ then $ X(-N_n^k , -1) $ contains at least $n^{\xi^k + \zeta}$ orders. If also $X(-2n,-1) =\emptyset$, then $X(-2n , -N_n^k-1)$ contains at least $n^{\xi^k+\zeta}$ burgers. Hence
\eqbn
\sup_{j\in [2n-2 n^{2\xi^k} , 2n]_{\BB Z}} |X(-j, -2n +  2n^{2\xi^k})| \geq n^{\xi^k + \zeta} .
\eqen
By \cite[Lemma 3.13]{shef-burger}, the probability that this is the case is of order $o_n^\infty(n)$. 

Let $\nu_\zeta \in (\mu' , \nu)$ be chosen so that $    \nu_\zeta (\xi^{\BB k} + \zeta)  \leq \nu \xi^{\BB k}$, which implies that $  \nu_\zeta  (\xi^k +  \zeta) \leq   \nu \xi^{k}$ for $k \leq \BB k$. 
By Lemma~\ref{prop-few-F-H}, the probability that $\mcl N_{\tb H}\left(X(-N_n^k , -1)\right) \leq n^{\xi^k +\zeta}$ and $\mcl N_{\tb F}\left(X(-N_n^k ,-1)\right) \geq n^{2\nu_\zeta (\xi^k + \zeta)}$ is of order $o_n^\infty(n)$. 
Hence except on an event of probability $o_n^\infty(n)$, whenever $X(-2n,-1) =\emptyset$ we have 
\eqb  \label{eqn-reduced-event-N}
N_n^k \leq  2n-  2n^{2 \xi^k  } \quad \op{and} \quad \mcl N_{\tb F}\left(X(-N_n^k ,-1)\right) \leq n^{2 \nu \xi^{k }} .
\eqe 
 
 Next we will show that it is very unlikely that $X(-2n,-1) =\emptyset$ and $K_n^k \geq 2n-n^{2\xi^k}$. To this end, suppose $X(-2n,-1) =\emptyset$,~\eqref{eqn-reduced-event-N} holds, and $K_n^k  > N_n^k$. Then $X(-N_n^k , -1)$ either has fewer than $n^{\xi^k-\zeta}$ hamburger orders or fewer than $n^{\xi^k-\zeta}$ cheeseburger orders. Assume without loss of generality that we are in the former setting. Let $\wt K_n^k$ be the smallest $j \geq N_n^k+1$ for which $X(-j,-N_n^k-1)$ contains at least $n^{\xi^k-\zeta}$ hamburger orders. By~\eqref{eqn-reduced-event-N} and Lemma~\ref{prop-no-tube}, if $X(-2n,-1) =\emptyset$ then except on an event of probability $o_n^\infty(n)$ we have $\wt K_n^k \leq  2n - (3/2) n^{2\xi^k}$. If $K_n^k > \wt K_n^k$, then $X(-\wt K_n^k , -1)$ contains at most $n^{\xi^k-\zeta}+1$ hamburger orders and at most $n^{\xi^k-\zeta}$ cheeseburger orders. By Lemma~\ref{prop-few-F-H} (c.f.\ the argument above), if this is the case then except on an event of probability $o_n^\infty(n)$, the word $X(-\wt K_n^k,-1)$ contains at most $n^{2\nu \xi^{k}} $ flexible orders. Hence if $K_n^k \geq 2n-n^{2\xi^k}$ and $X(-2n,-1)=\emptyset$, then except on an event of probability $o_n^\infty(n)$ we have
\eqbn
\sup_{j\in [\wt K_n^k +1, 2n - n^{2\xi^k}]_{\BB Z}} \left| d\left( X(-j, -\wt K_n^k-1) \right) \right| \wedge \left| d^*\left( X(-j, -\wt K_n^k-1) \right) \right|  \leq n^{\xi^k-\zeta} +n^{2\nu \xi^k} + 1 .
\eqen
By Lemma~\ref{prop-no-tube}, the probability that this is the case is of order $o_n^\infty(n)$.  
It follows that the probability that $ X(-2n,-1) =\emptyset $ and condition~\ref{item-empty-prob-times} in the definition of $E_n^k$ fails to occur is $o_n^\infty(n)$. 
 
We have $K_n^k \geq N_n^k \geq 2n-  n^{2(\xi^k + \zeta)}$ by definition, so \cite[Lemma 3.13]{shef-burger} implies that except on an event of probability at most $o_n^\infty(n)$, 
\eqb \label{eqn-K-to-2n-sup}
\sup_{j\in [K_n^k+1 , 2n]_{\BB Z}} |X(-j,-K_n^k-1 )| \leq n^{\xi^k + 2\zeta}  .
\eqe 
If $X(-2n,-1) =\emptyset$, then $X(-2n,-K_n^k-1)$ contains at least $|X(-K_n^k,-1)|$ burgers. From this and~\eqref{eqn-K-to-2n-sup}, we infer that condition~\ref{item-empty-prob-orders} in the definition of $E_n^k$ holds with probability at least $1-o_n^\infty(n)$.

It is clear that condition~\ref{item-empty-prob-burgers} in the definition of $E_n^k$ occurs whenever $X(-2n,-1) =\emptyset$. 
By Lemma~\ref{prop-few-F-H}, applied as in the argument leading to~\eqref{eqn-reduced-event-N} (but with $2\zeta$ in place of $\zeta$) the probability that condition~\ref{item-empty-prob-orders} in the definition of $E_n^k$ occurs but condition~\ref{item-empty-prob-F} in the definition of $E_n^k$ fails to occur is of order $o_n^\infty(n)$. This completes the proof of~\eqref{eqn-reduced-event-prob}. 
\end{proof}

\begin{proof}[Proof of Theorem~\ref{thm-empty-prob}, upper bound]
Define the times $N_n^k$ and $K_n^k$ and the events $E_n^k$ as above. We will show that
\eqb \label{eqn-reduced-event-cap}
\BB P\left(\bigcap_{k=0}^{\BB k} E_n^k  \right) \leq n^{-1-2\mu + 2(1+\mu) \xi^{\BB k}  + o_\zeta(1) + o_n(1)}    .
\eqe
The desired upper bound for $\BB P(X(-2n,-1) =\emptyset)$ will follow by combining this estimate with Lemma~\ref{prop-reduced-empty-event-prob} and sending $\zeta \rta 0$ and then $\BB k\rta\infty$. 
  
By \cite[Lemma 6.1]{gms-burger-cone}, 
\eqb \label{eqn-empty-E0-prob}
\BB P\left(E_n^0\right) \leq n^{-\mu + o_n(1)} .
\eqe
Now let $k\in [1,\BB k]_{\BB Z}$. We will estimate the conditional probability of $E_n^k$ given $\bigcap_{j=0}^{k-1} E_n^j$ by applying assertion~\ref{item-corner-upper} of Proposition~\ref{prop-corner-local} to the word $X_{-2n} \dots X_{-K_n^{k-1}-1}$. 
To do this, we need to check that if $\bigcap_{j=0}^{k-1} E_n^j$ occurs, then $E_n^k$ is contained in the event~\eqref{eqn-corner-local-event} of Proposition~\ref{prop-corner-local} for an appropriate choice of parameter values depending only on $X_{-K_n^{k-1}}\dots X_{-1}$. 

To this end, suppose that $\bigcap_{j=0}^{k} E_n^j$ occurs. 
By condition~\ref{item-empty-prob-orders} in the definition of $E_n^{k}$ and the definition of $K_n^{k}$, the word $X(-K_n^{k} , -1)$ contains at most $n^{\xi^k + 2\zeta}$ orders. 
Hence the word $X(-K_n^k,-K_n^{k-1}-1)$ contains at least
\eqb \label{eqn-burger-count-lower}
\mcl N_{\tb H}\left(X(-K_n^{k-1} , -1)\right) - n^{\xi^k + 2\zeta}  
\eqe  
hamburgers, and similarly for cheeseburgers. 
Furthermore, 
\eqb \label{eqn-order-upper}
\text{the total number of orders in $X(-K_n^k,-K_n^{k-1}-1)$ is at most $n^{\xi^k + 2\zeta}$}
\eqe
since orders cannot be removed when we add additional symbols to the right of a word.  

By condition~\ref{item-empty-prob-F} in the definition of $E_n^{k-1}$ (or the definition of $E_n^0$ if $k=1$) and condition~\ref{item-empty-prob-burgers} in the definition of $E_n^k$, the word $X(-K_n^k , -K_n^{k-1} -1)$ contains at most 
\eqb \label{eqn-burger-count-upper}
\mcl N_{\tb H}\left(X(-K_n^{k-1} , -1)\right) + 
\begin{cases}
n^\nu ,&\quad k=1 \\
n^{2\nu \xi^{k-1} } , &\quad k \in [2,\BB k]_{\BB Z} .
\end{cases}
\eqe 
hamburgers, and similarly for cheeseburgers.
By definition of $K_n^k$, for each $k\in [1,\BB k]_{\BB Z}$, 
\eqb  \label{eqn-time-approx}
K_n^k \in 
 \left[  2n - 2 n^{2 (\xi^k+2\zeta)   } , 2 n    \right]_{\BB Z}  .
\eqe

By condition~\ref{item-empty-prob-times} in the definition of $E_n^{k-1}$ (or the definition of $E_n^0$ if $k=1$),
\eqbn
2n - K_n^{k-1} \geq \begin{cases}
n ,&\quad k=1 \\
n^{2\xi^{k-1}} , &\quad k \in [2,\BB k]_{\BB Z} .
\end{cases}
\eqen
and by this same condition together with the definition of $K_n^{k-1}$,  
\eqb \label{eqn-empty-count}
\mcl N_{\tb H}\left(X(-K_n^{k-1} , -1)\right)
\geq \begin{cases}
  n^{1/2-\zeta} , &\quad k=1 \\
 n^{\xi^{k-1}-\zeta},  &\quad k \in [2,\BB k]_{\BB Z} ;
 \end{cases}
\eqe 
and similarly with $\tb C$ in place of $\tb H$. 
  
For $k \in [2,\BB k]_{\BB Z}$, we infer from~\eqref{eqn-burger-count-lower}, \eqref{eqn-order-upper}, \eqref{eqn-burger-count-upper}, and \eqref{eqn-time-approx} that $\bigcap_{j=0}^{k} E_n^j$ is contained in the event $E_{\wt n,\wt k , r}^{h,c}$ of~\eqref{eqn-corner-local-event} defined with $X_{-2n} \dots X_{-K_n^{k-1}-1}$ in place of $X_{-n} \dots X_{-1}$, 
\alb
\wt n &=   2n- K_n^{k-1} \geq  n^{  2 \xi^{k-1}} \: \text{in place of $n$} \\
h &=  \mcl N_{\tb H}\left(X(-K_n^{k-1} , -1)\right) + n^{2\nu \xi^{k-1}} \\
c &=  \mcl N_{\tb C}\left(X(-K_n^{k-1} , -1)\right)  + n^{2\nu \xi^{k-1}} \\
\wt k &=  2 n^{2 (\xi^k + \zeta)  } \leq \wt n \: \text{in place of $k$}  \quad \op{and} \\
r &= n^{\xi^k + 2\zeta} + n^{2\nu \xi^{k-1} } \in \left[ \frac12 \wt k^{1/2}  , 2 n^{\xi^k+2\zeta} \right]_{\BB Z} .
\ale
Note that the conditions in the definition of $E_{\wt n,\wt k , r}^{h,c}$ are satisfied with $X_{-2n} \dots X_{-K_n^{k-1}-1}$ in place of $X_{-n}\dots X_{-1}$ and $j = K_n^k -K_n^{k-1}$. 

We have $E_{\wt n,\wt k , r}^{h,c} \subset E_{\wt n , 2 r^2 , r}^{h,c}$ and by~\eqref{eqn-empty-count}, the above choices of parameters satisfy $r \leq n^{\xi^{k-1} -\zeta} \leq h\wedge c$ on $E_n^{k-1}$. By assertion~\ref{item-corner-upper} of Proposition~\ref{prop-corner-local} together with~\eqref{eqn-empty-count}, on the event $ \bigcap_{j=1}^{k-1} E_n^j$ it holds that 
\alb
&\BB P\left(E_n^k   \,|\,  X_{-K_n^{k-1}} \dots X_{-1} \right)    \\
&\qquad  \leq n^{o_n(1)} \left( \mcl N_{\tb H}\left(X(-K_n^{k-1} , -1)\right) \wedge \mcl N_{\tb C}\left(X(-K_n^{k-1} , -1)\right)  + n^{\xi^k} \right)^{-2-2\mu} n^{(1+\mu) (2\xi^k + 4\zeta )} \\
&\qquad \leq  n^{ - 2(1+\mu)( \xi^{k-1} -   \xi^{k } ) +o_\zeta(1) + o_n(1)} .
\ale
with the $o_\zeta(1)$ depending on $\BB k$ but not $n$. Similarly,
\eqbn
\BB P\left( E_n^1 \,|\, X_{-K_n^0} \dots X_{-1} \right) \leq n^{ - 2(1+\mu)(1/2 -   \xi ) +o_\zeta(1) + o_n(1)} .
\eqen
\begin{comment}
\alb
\wt n &=   2n- K_n^{k-1} \geq  n \: \text{in place of $n$} \\
h &=  \mcl N_{\tb H}\left(X(-K_n^{0} , -1)\right) + n^{\nu } \\
c &=  \mcl N_{\tb C}\left(X(-K_n^{0} , -1)\right)  + n^{\nu} \\
\wt k &=  2 n^{2 \xi   } \leq \wt n \: \text{in place of $k$}  \quad \op{and} \\
r &= n^{\xi  + 2\zeta} + n^{ \nu   } \in \left[ \frac12 \wt k^{1/2}  , 2 n^{\xi^k+2\zeta} \right]_{\BB Z} .
\ale
\end{comment}
Recalling~\eqref{eqn-empty-E0-prob} and multiplying our estimates over all $k\in [0,\BB k]_{\BB Z}$, we infer that~\eqref{eqn-reduced-event-cap} holds.
By combining this with Lemma~\ref{prop-reduced-empty-event-prob}, we get
\eqbn
\BB P\left( X(-2n,-1) = \emptyset \right) \leq n^{-1-2\mu + 2(1+\mu) \xi^{\BB k}  + o_\zeta(1) + o_n(1)} .
\eqen
By sending $\zeta\rta 0$ and then $\BB k\rta\infty$, we conclude.
\end{proof}

\appendix

\section{Index of symbols}
\label{sec-index}

Here we record some commonly used symbols in the paper, along with reminders of their definitions and the locations where they are precisely defined. Notations which are only used locally (i.e., just in one subsection) are not included.

\begin{multicols}{2}
\begin{itemize}
\item $p$; parameter of the model. Section~\ref{sec-burger-prelim}. 
\item $\mcl R(\cdot)$; reduced word. Section~\ref{sec-burger-prelim}.
\item $|x|$; number of symbols in the word $x$. Section~\ref{sec-burger-prelim}.
\item $X$; bi-infinite word with iid symbols. Section~\ref{sec-burger-prelim}.
\item $X(a,b)$; reduction of $X_{\lfloor x \rfloor} \dots X_{\lfloor b\rfloor}$.~\eqref{eqn-X(a,b)}. 
\item $\phi$; match operator. Notation~\ref{def-match-function}. 
\item $\mcl N_{\tb F}(x)$, etc.; number of symbols of a given type in $x$. Definition~\ref{def-theta-count}. 
\item $D = (d,d^*)$; $d = \mcl N_{\tc H} - \mcl N_{\tb H}$ and $d^* =\mcl N_{\tc C} - \mcl N_{\tb C}$. Definition~\ref{def-theta-count}.
\item $Z^n = (U^n ,V^n)$; re-scaled discrete walk.~\eqref{eqn-Z^n-def}.
\item $Z = (U,V)$; correlated two-dimensional Brownian motion.~\eqref{eqn-bm-cov}.
\item $I$; first time $X(1,i)$ contains an order.~\eqref{eqn-I-def}.
\item $J^H$; first time $X(-j,-1)$ contains a $\tc H$.~\eqref{eqn-J^H-def}.
\item $L^H = d^*(X(-J^H,-1))$.~\eqref{eqn-J^H-def}. 
\item $\mu  =    \frac{\pi}{2\left( \pi - \arctan \frac{\sqrt{1-2p} }{p} \right) } \in (1/2,1)$.~\eqref{eqn-cone-exponent}.
\item $\mu' = \frac{\pi}{2\left( \pi + \arctan \frac{\sqrt{1-2p} }{p} \right) }   \in (1/3,1/2)  $.~\eqref{eqn-cone-exponent}.
\item $\nu$; typically an exponent in $(\mu',1/2)$.
\item $\mcl E_n^{h,c}$; event that $X(1,n)$ contains no orders, $h$ $\tc H$'s, and $c$ $\tc C$'s. Definition~\ref{def-local-event}. 
\item $\psi_0$; slowly varying correction for the tail of $I$. Lemma~\ref{prop-I-reg-var}.
\item $\psi_1$; slowly varying correction for the first time $X(-j,-1)$ contains no orders. Lemma~\ref{prop-P-reg-var}.
\item $\frk h(x) = \mcl N_{\tc H}(\mcl R(x))$. Definition~\ref{def-frk}.
\item $\frk c(x)= \mcl N_{\tc C}(\mcl R(x))$. Definition~\ref{def-frk}.
\item $K_{n,m}^H $; last time $i$ before $n$ such that $\mcl N_{\tc H}(X(1,K_{n,m}^H))=m$ and $X_{i+1}$ is a burger not consumed before time $n$. Section~\ref{sec-no-order-local-setup}. 
\item $Q_{n,m}^H = \mcl N_{\tc C}(X(1,K_{n,m}^H))$. Section~\ref{sec-no-order-local-setup}. 
\item $J_{n,r}^H$. Smallest $j\in\BB N$ such that $X(n-j,n)$ contains $r$ hamburgers. Section~\ref{sec-no-order-local-setup}. 
\item $L_{n,r}^H  = d^*\left(X(n-J_{n,r}^H , n)\right)$. Section~\ref{sec-no-order-local-setup}. 
\item $R_n(x)  = \{ \text{$n-1- |x|  =   J_{n,r}^H$ for some $r\in\BB N$}\}$.~\eqref{eqn-J^H-hit-event}.
\item $m_h^\delta  = \lfloor (1-\delta) h\rfloor$. Proposition~\ref{prop-no-order-endpoint}.
\item $\mcl U_n^\delta(A , h,c)  =  \left[n - A^2  \delta^2 h^2 , n - A^{-2} \delta^2 h^2 \right]_{\BB Z} \times \left[ c - A \delta h  , c + A \delta h\right]_{\BB Z}$. Proposition~\ref{prop-no-order-endpoint}. 
\item $E_{n,k,r}^{h,c}$; event that there exists $j\in [n-k,n]_{\BB Z}$ such that $X(-j,-1)$ contains approximately $h$ $\tc H$'s, approximately $c$ $\tc C$'s, and few orders.~\eqref{eqn-corner-local-event}.
\item $J_{n,k,r}^{h,c}$; the smallest $j$ satisfying the definition of $E_{n,k,r}^{h,c}$. 
\item $\wt E_{n,k,r}^{h,c}$; variant of $E_{n,k,r}^{h,c}$ which is contained in $\wt E_{n,k,r}^{h,c}$.~\eqref{eqn-corner-local-event-J^H}.
\end{itemize}
\end{multicols}

\bibliography{cibiblong,cibib} 
\bibliographystyle{hmralphaabbrv}

\end{document}